\documentclass[11pt, a4paper]{article}

\usepackage[margin=2cm, bottom=3cm]{geometry}

\usepackage[utf8]{inputenc}    
\usepackage[T1]{fontenc}       
\usepackage{times}             

\usepackage{amsmath, amssymb, amsthm, mathtools}
\usepackage{stmaryrd}          
\usepackage{wasysym}           
\usepackage{cases}             
\usepackage{bigints}           

\usepackage{graphicx}          
\usepackage{float}             
\usepackage[labelfont=bf]{caption}  
\usepackage{placeins}          

\usepackage[colorlinks=true,
linkcolor=blue,
citecolor=blue,
urlcolor=blue]{hyperref}   
\usepackage{cleveref}                  

\usepackage{array, multirow, makecell}
\usepackage[shortlabels]{enumitem}
\usepackage[math]{cellspace}          
\usepackage{longtable}

\usepackage[round, authoryear]{natbib}  

\usepackage{setspace}
\newcommand*\dif{\mathop{}\!\mathrm{d}}
\renewcommand{\eqref}[1]{Equation~(\ref{#1})}

\theoremstyle{definition}
\newtheorem{thm}{Theorem}
\newtheorem{defn}{Definition}
\newtheorem{prop}{Proposition}

\newtheorem{rem}{Remark}
\newtheorem{exam}{Example}
\title{ 
	\textbf{Infinitely refinable generalization of quad-mesh rigid origami: from linear and equimodular couplings}
}
	
\author{Zeyuan He$^{1,2}$, Kentaro Hayakawa$^{2,3}$ and Makoto Ohsaki$^2$ 
	\\ \small $^1$Department of Engineering, University of Cambridge
	\\ \small $^2$Department of Architecture and Architectural Engineering, Kyoto University 
	\\ \small Department of Conceptual Design, College of Industrial Technology, Nihon University \\ \small zh299@cam.ac.uk, hayakawa.kentaro@nihon-u.ac.jp, ohsaki@archi.kyoto-u.ac.jp}
	
\date{}

\begin{document}
	
\maketitle

\begin{abstract}
A quad-mesh rigid origami is a continuously deformable panel-hinge structure where planar, rigid, zero-thickness quadrilateral panels are connected by rotational hinges in the combinatorics of a grid. This article provides a comprehensive exposition of two new families of infinitely refinable quad-mesh rigid origami, generated from linear and equimodular couplings. These constructions expand the current landscape beyond well-known variations such as the Miura-ori, V-hedron (discrete Voss surface or eggbox pattern), anti-V-hedron (flat-foldable pattern), and T-hedron (trapezoidal pattern). We conjecture that as the mesh is refined to infinity, these quad-mesh rigid origami converges to special ruled surfaces in the limit, supported by multiple lines of evidence.
\end{abstract}

\quad {\small \textbf{Keywords:} flexibility, rigid-foldability, Kokotsakis quadrilateral, surface approximation}

\section{Introduction}
\label{sec: introduction}

Quad-mesh rigid origami refers to a class of structures formed by planar and zero-thickness quadrilateral panels connected by rotational hinges in a grid-like configuration, enabling continuous isometric deformations while preserving the rigidity of each panel. Such deformations are commonly referred to as \textit{flexes}, \textit{flexions}, or \textit{folding motions} in the literature. An example of the most well-known quad-mesh rigid origami, the Miura-ori \citep{miura_method_1985}, is shown in Figure \ref{fig: quad mesh introduction}(a). The most commonly studied quad-mesh rigid origami structures are (anti-)V-hedra and T-hedra, as shown in Figure \ref{fig: quad mesh introduction}(b)--(f). In the origami community, research has traditionally focused on developable origami, where the sum of the sector angles around every interior vertex equals $2\pi$, ensuring that the discrete Gaussian curvature at each vertex is zero (details provided at the end of Section \ref{section: discrete curvature}). In our framework, however, the sum of sector angles at each interior vertex is \textbf{not necessarily} equal to $2\pi$, allowing for the inclusion of both developable and non-developable origami structures.

In this article, we focus on more generalized quad-mesh rigid origami that extend beyond (anti-)V-hedra and T-hedra, namely, those generated from linear and equimodular couplings. The mathematical description of these terms, originating from \citet{izmestiev_classification_2017}, are provided in Sections \ref{section: linear coupling} and \ref{section: equimodular coupling}. We present a detailed framework for generating these new variations, focusing particularly on those meshes that are infinitely refinable. We further conjecture that, under infinite refinement, these patterns converge to smooth surfaces, and we derive explicit forms of these limiting surfaces. This mesh of refinement and convergence gives rise to novel curved crease origami designs. All design examples and tools are made available in an open-source MATLAB application. 

To clarify the distinctions among these different types of quad-mesh rigid origami, we begin with a brief overview of (anti-)V-hedra and (anti-)T-hedra.

\subsection*{(Anti-)V-hedra and (anti-)T-hedra}

A V-hedron (Figure \ref{fig: quad mesh introduction}(b) and \ref{fig: quad mesh introduction}(c)) refers to a non-developable quad-mesh rigid origami where opposite sector angles are equal at every interior vertex ($\alpha = \gamma, ~\beta = \delta$ if the sector angles at a vertex are denoted by $\alpha, ~\beta, ~\gamma, ~\delta$ in counterclockwise order). It has a special state where the folding angle at every vertex is $\{\pm \pi, ~0, ~\pm \pi, ~0\}$ (in counterclockwise order), which can be folded to another special state where the folding angle at every vertex is $\{0, ~\pm \pi, ~0, ~\pm \pi\}$ (in the same counterclockwise order). The name V-hedron is from the early research on Voss surface \citep{voss_diejenigen_1888, bianchi_sopra_1890, eisenhart_transformations_1914} and discrete Voss surface \citep{sauer_uber_1931}, which is also called an eggbox pattern in the origami community \citep{tachi_freeform_2010-1}. An anti-V-hedron (Figure \ref{fig: quad mesh introduction}(d)) is a developable quad-mesh rigid origami where opposite sector angles are supplementary to $\pi$ at every interior vertex ($\alpha + \gamma = \pi, ~\beta + \delta = \pi$). This pattern is widely recognized as a developable and flat-foldable quad-mesh rigid origami \citep{tachi_generalization_2009}. It has a \textit{planar} state where all the folding angles are zero, which can be folded continuously to another \textit{flatly-folded} state where the folding angles are all $\pm \pi$. 

In \cite{he_rigid_2020} we showed that \textit{switching a strip} --- changing the sector angles on a row or column of quadrilateral panels to their supplements with respect to $\pi$ --- maps a quad-mesh rigid origami to another quad-mesh rigid origami and preserves the flexibility. A V-hedron becomes an anti-V-hedron after switching alternating strips, and becomes a hybrid V-hedron (also discussed in \cite{tachi_freeform_2010-1}) if only switching some strips. Details on the flexibility of quad-mesh rigid origami and operations generating quad-mesh rigid origami from an existing one are provided in Section \ref{section: loop condition}.

A T-hedron (Figure \ref{fig: quad mesh introduction}(e) and \ref{fig: quad mesh introduction}(f)) refers to a quad-mesh rigid origami whose vertices are  orthodiagonal ($\cos \alpha \cos \gamma = \cos \beta \cos \delta$) and every two vertices form an involutive coupling. These terminologies are special geometric requirements on the sector angles \citep[Section 3.1.1]{izmestiev_classification_2017}. A T-hedron can be either developable or non-developable. The name T-hedron is from \citet{sauer_uber_1931}. An anti-T-hedron refers to a quad-mesh rigid origami whose vertices are orthodiagonal and every two vertices form an anti-involutive coupling \citep[Section 3.1.2]{izmestiev_classification_2017}. A $3 \times 3$ anti-T-hedron, with nine quadrilateral panels, was studied in \cite{erofeev_orthodiagonal_2020}, yet there has not been reported progress on a larger mesh.

A more comprehensive introduction of (anti-)V-hedra and T-hedra is provided in Sections \ref{section: V-hedra} and \ref{section: T-hedra}.

\subsection*{Surface approximation}

\begin{figure*}[t]
	\centering
	\includegraphics[width=\linewidth]{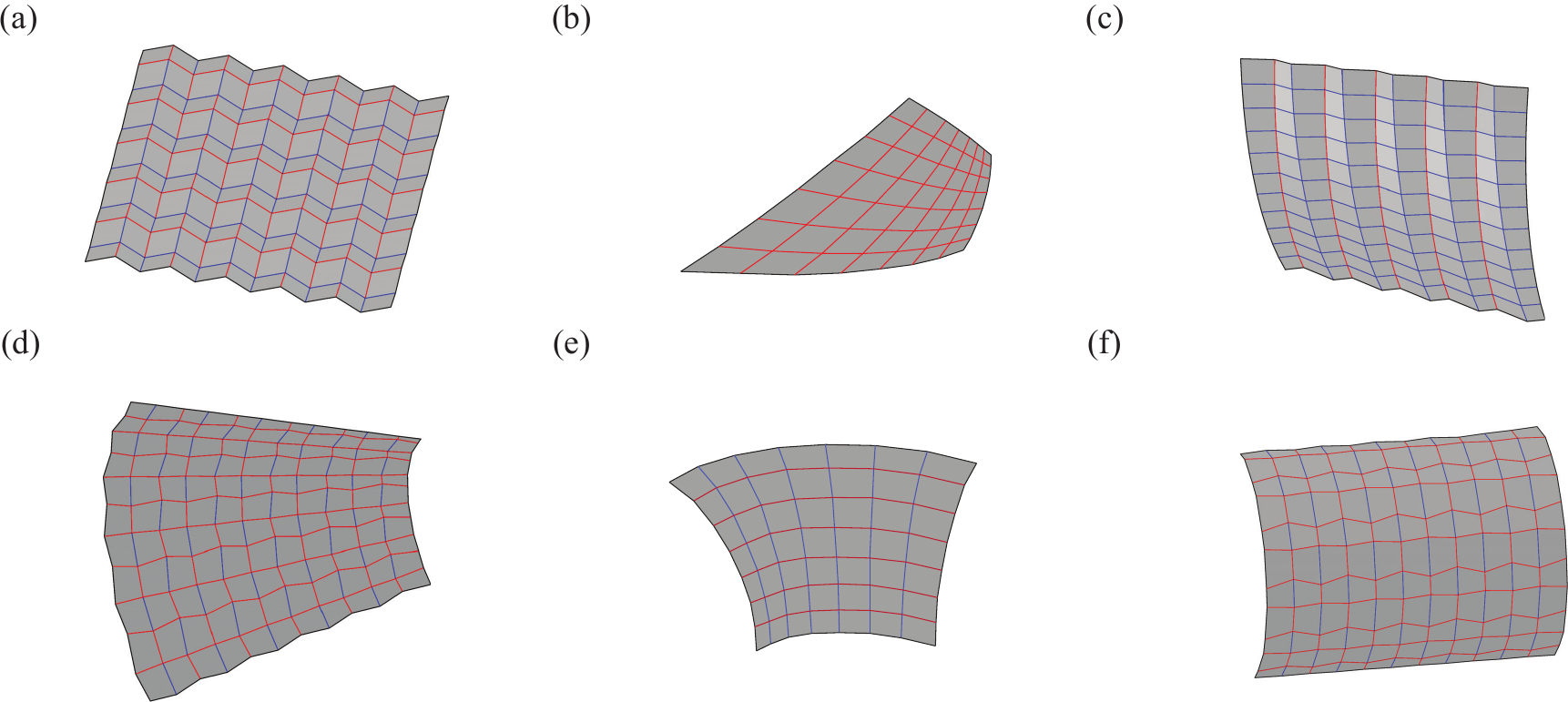}
	\caption{\label{fig: quad mesh introduction}A brief gallery of common quad-mesh rigid origami, including (a) the Miura-ori \citep{miura_method_1985}; (b) a non-developable V-hedron \citep{sauer_differenzengeometrie_1970}; (c) another non-developable V-hedron; (d) a developable anti-V-hedron; (e) a non-developable T-hedron \citep{izmestiev_isometric_2024} ; and (f) a developable T-hedron. Mountain creases are coloured red and valley creases are coloured blue.}
\end{figure*}

In addition to the variety of quad-mesh rigid origami, there has been a continuous effort within the origami research community to explore the surface a quad-mesh rigid origami can approximate. We are further motivated to explore how closely a quad-mesh rigid origami can approximate a smooth surface as the mesh is refined. In other words, for a series of quad-mesh rigid origami following a construction method that allows arbitrary mesh refinement, we aim to investigate the convergence toward a smooth surface in terms of Euclidean distance (detailed in Section \ref{section: ddg convergence}). In the simplest case, under a labelled correspondence between the discrete mesh points and points on the smooth surface, we seek to determine whether the average Euclidean distance between corresponding points can be made arbitrarily small.

The first level of surface approximation happens when a series of quad-mesh rigid origami converge to a smooth surface in distance, and they represent the discrete and smooth forms of the same coordinate net. Consequently, as the mesh is refined, their tangent planes, metric-related and curvature-related properties can become arbitrarily close. Furthermore, a quad-mesh rigid origami admits only a single degree of freedom in its folding motion \citep{schief_integrability_2008}. As the mesh is refined, this discrete motion converges to a one-parameter flex of the limiting smooth surface. As shown in Figure \ref{fig: quad mesh introduction}(b) and \ref{fig: quad mesh introduction}(e), certain V-hedra and T-hedra reach this level of approximation, with the resulting smooth surfaces referred to as V-surfaces \citep{sauer_differenzengeometrie_1970, izmestiev_voss_2025} and T-surfaces \citep{izmestiev_isometric_2024}. Due to this unique relationship, we refer to them as \textit{discrete} and \textit{smooth analogues} of one another.

The second level of surface approximation involves convergence only in terms of distance, without guaranteeing the convergence of tangent planes or properties related to metric and curvature. A limiting smooth surface can be reached with a series of quad-mesh rigid origami, although there is no guarantee that their deformation paths converge to a one-parameter flex of the smooth surface. Some other V-hedra and T-hedra fall into this category, as shown in Figure \ref{fig: quad mesh introduction}(a), (c), (d), and (f). Examples include the Miura-ori and the revolutionary Miura-ori \citep{song_design_2017, hu_design_2019}. In these examples, although we can design this pattern to be close to a plane or a surface of revolution, the origami structure deviates further from these target surfaces as it is folded flat. A common feature for them is they have a `zig-zag' mode --- we will explain this further in Section \ref{section: discussion}.

The third level of surface approximation is frequently employed in origami-based engineering design, such as pavilions, shelters and shells. It would be geometrically sufficient if the origami structure can exhibit desired curvature with limited number of grids. Numerous publications have explored such inverse design employing V-hedra, anti-V-hedra or T-hedra to construct three-dimensional structures. Notably, the number of free variables for an (anti-)V-hedron increases linearly with respected to the number of grids, hence there is sufficient space for shape optimization. The inputs for these inverse design methods include perturbation from an existing pattern \citep{tachi_freeform_2010-1}; an array of folding angles and crease lengths of boundary polylines \citep{lang_rigidly_2018}; `curved creases' \citep{jiang_curve-pleated_2019}; the outline of the mesh and a multivariable objective function \citep{hayakawa_form_2020}; a target surface \citep{dang_inverse_2022}; the discrete normal field (i.e., Gauss map) \citep{montagne_discrete_2022}; and control polylines or vertices \citep{kilian_interactive_2024}. T-hedra have less free variables and are less frequently applied yet, but showed great promise for highly accurate approximation of certain surfaces. The inputs include boundary or control polylines \citep{he_approximating_2018, sharifmoghaddam_using_2020}. Additionally, \cite{he_approximating_2018} demonstrated that a hybrid mesh composed of anti-V-hedra and developable T-hedra can be designed to construct developable quad-mesh rigid origami with a self-locking property — where motion halts at a desired configuration due to panel contact. 

\section{Result}

We introduce the \textit{repetitive stitching} construction method that enables infinite mesh refinement for two newly identified families of quad-mesh rigid origami, composed of linear and equimodular couplings. The method and related terminologies will be explained in Section \ref{section: method}. The accompanying open-source code \citep{he_sector-angle-periodic_2024} includes all relevant data and supports parametric design, mesh refinement, visualization, and user customization. A gallery of examples is presented in Figure~\ref{fig: example}.

\begin{figure*}[tbph]
	\centering
	\includegraphics[width=1\linewidth]{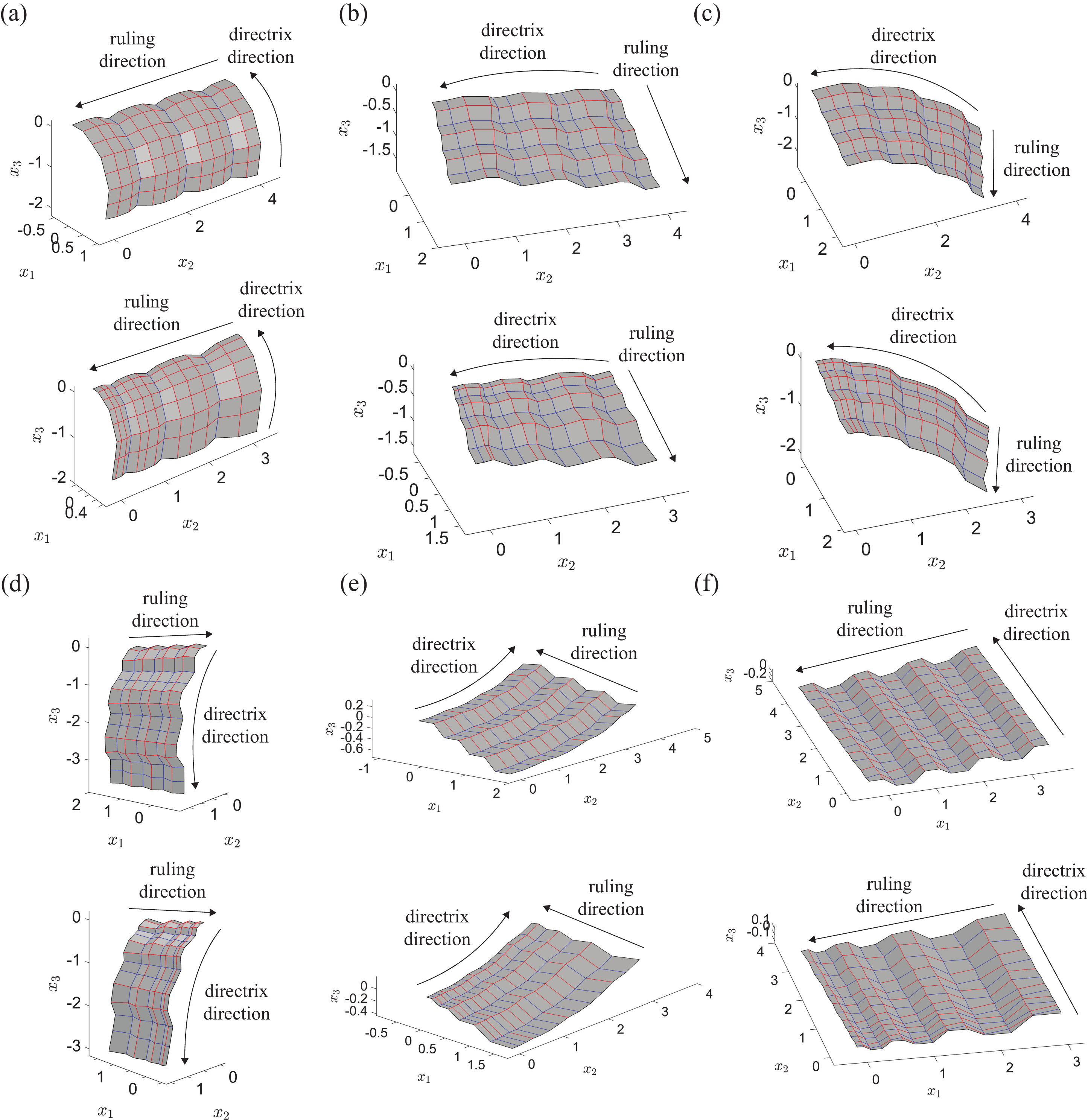}
	\caption{\label{fig: example} A gallery of quad-mesh rigid origami generated using the repetitive stitching method. These examples are symbolically distinguished from the (anti)-V-hedra and (anti)-T-hedra. Mountain folds are coloured red, while valley folds are coloured blue. For each example, we display two designs that share the same sector angles but differ in their crease length distributions, each consisting of nine units (for the concept of unit, see Figure \ref{fig: unit}). More refined meshes can be readily produced using the accompanying script \citep{he_sector-angle-periodic_2024} by increasing the number of units. All structures exhibit a characteristic zig-zag mode, in which the coordinate polylines along the row and column directions oscillate along a ruling line. Both the ruling and directrix directions are indicated for each example.}
\end{figure*} 

We conjecture that a special ruled surface $x(u_1, ~u_2)$ in the form below can be approximated, at the second level of surface approximation, by a series of quad-mesh rigid origami generated from the repetitive stitching method:
\begin{equation} \label{eq: special ruled surface}
	\begin{gathered}
		x(u_1, ~u_2) = \Gamma(u_2) + u_1\Phi(u_2), ~~u_1, ~u_2 \in \mathbb{R}, ~~x \in \mathbb{R}^3 \\
		\Gamma(u_2) = \Gamma(0) + \int \limits_{v = 0}^{u_2} f(v) \begin{bmatrix}
			-a \sin v \\
			a \cos v \\
			b 
		\end{bmatrix} \dif v, ~a>0, ~b \in \mathbb{R}, ~f(u_2) = \dfrac{\left \| \dfrac{\dif \Gamma}{ \dif u_2} (u_2) \right\|}{\left \| \dfrac{\dif \Gamma}{ \dif u_2} (0) \right\|} \\
		\Phi(u_2) \in \mathbb{R}^3,~||\Phi(u_2)|| \equiv 1, ~\dfrac{\dfrac{\dif \Gamma}{\dif u_2} \cdot \Phi}{ f\sqrt{a^2+b^2}} = \mathrm{Const} \in [0, 1)		
	\end{gathered}
\end{equation}
where $f(u_2)$ is a known input crease length distribution function, $\Gamma(u_2)$ is the directrix and $\Phi(u_2)$ is the direction of rulings. The directrix meets the rulings at a constant angle for all $u_2$. Evidence supporting this conjecture is provided in Section \ref{subsection: evidence}. \eqref{eq: special ruled surface} can be utilized to compute the apparent curvature of the origami structure and to formulate optimal inverse design algorithms.

\section{Method} \label{section: method}

\subsection{Repetitive stitching} \label{subsection: stitching} 

A large quad-mesh rigid origami is flexible (i.e. foldable) if and only if all its $3 \times 3$ quadrilaterals (Kokotsakis quadrilaterals) are flexible \citep{schief_integrability_2008}. Thus, by utilizing the classification of flexible Kokotsakis quadrilaterals provided in \cite{izmestiev_classification_2017}, it is possible to construct a large quad-mesh rigid origami by assembling these $3 \times 3$ building blocks. The terminology \textit{Kokotsakis quadrilateral} is named after Antonios Kokotsakis, who studied the flexibility of these polyhedral surfaces in his PhD thesis in 1930s and described several flexible classes \citep{kokotsakis_uber_1933}. At the same time, \cite{sauer_uber_1931} also found several classes. Recent works from \cite{karpenkov_flexibility_2010, stachel_kinematic_2010, nawratil_reducible_2011, nawratil_reducible_2012} made solid contribution to this topic.

Izmestiev describes each type of flexible Kokotsakis quadrilateral by a system of highly nonlinear equations on the sector angles, which is obtained from the calculation conducted in the complexified configuration space. The above limitation necessitates examining: (1) the existence of real solutions to these systems; and (2) the existence of an actual folding motion in $\mathbb{R}^3$. Additionally, to support mesh refinement to infinity, (3) the construction should be `infinitely extendable', rather than restricted in a finite grid. We select two types of quad-mesh rigid origami from \cite{he_rigid_2020} that are entirely distinct from the (anti)-V-hedra and T-hedra and satisfy requirements (1) to (3). 

The infinite mesh refinement is achieved by repetitively stitching \textit{units}. Figure~\ref{fig: unit}(a) shows such a unit of size $3 \times 5$, containing 8 interior vertices and 32 sector angles. The sequence from Figure~\ref{fig: unit}(a) to (c) illustrates how the construction progressively approximates a smooth surface through mesh refinement. Figure~\ref{fig: unit}(d) graphically explains the stitching method: a larger mesh is formed by arranging $m \times n$ `copies' of the unit, where $m, ~n \in \mathbb{Z}_+$ indicate the number of units along the longitudinal and transverse directions. All units share the same sector angles but may differ in crease lengths. For example, placing four $3 \times 5$ units together yields a larger mesh of size $5 \times 9$, in which the sector angles tessellate seamlessly. The resulting mesh contains $mn$ copies of the unit's sector angles. To fully determine the geometry of the entire pattern, one can adjust the crease lengths along a single row and a single column.

\begin{figure*}[tbph]
	\centering
	\includegraphics[width=0.8\linewidth]{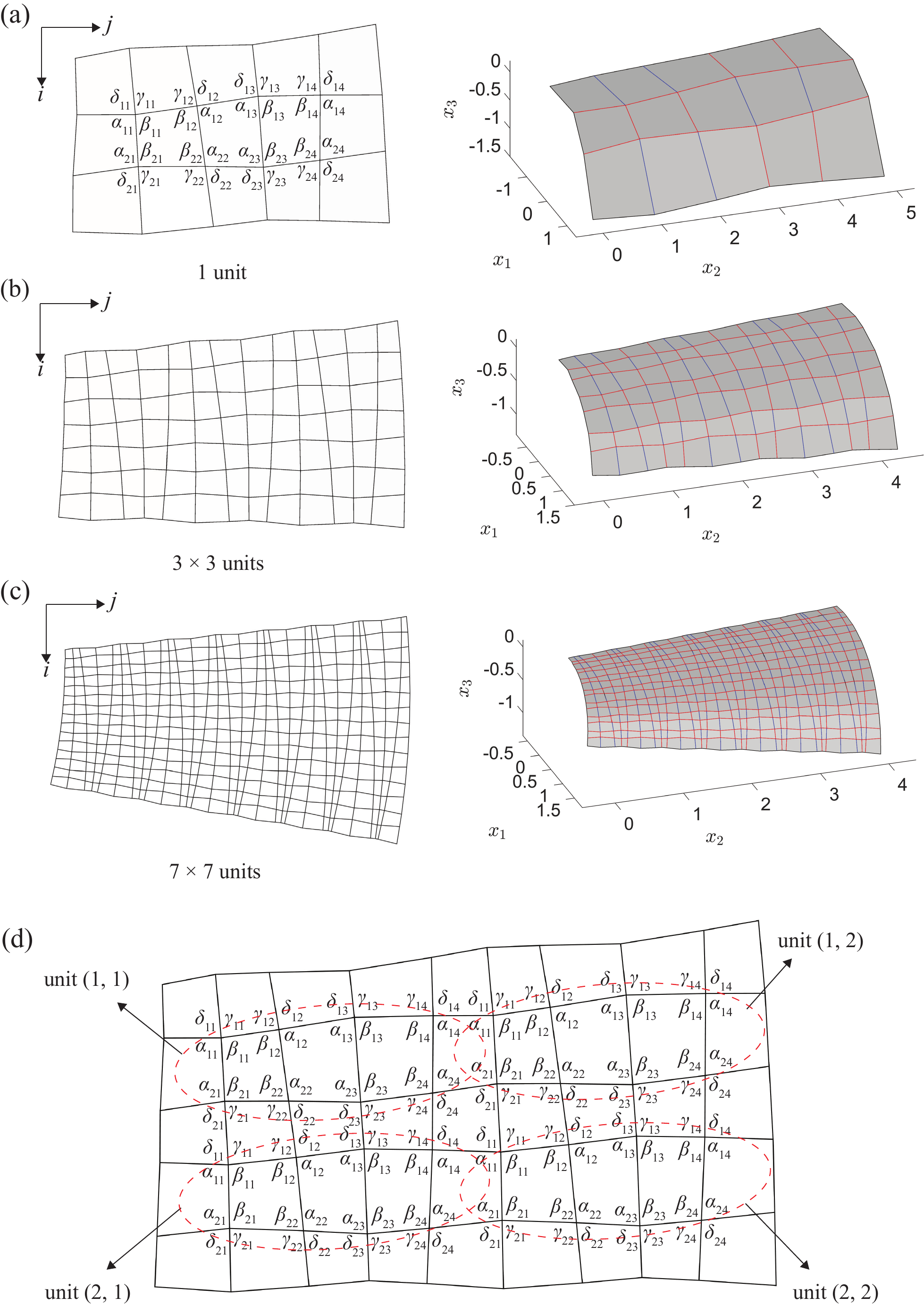}
	\caption{\label{fig: unit} Repetitive stitching of units. (a) shows a unit and our labelling of sector angles. (b) and (c) show the approximation to a smooth surface from refining the mesh. (d) illustrates the stitching process, where sector angles from one unit are repeated and stitched together to form the new pattern with 4 units. One row and one column of crease lengths can be adjusted.}
\end{figure*}

In Figure~\ref{fig: unit}(a), the sector angles $\alpha_{ij}, ~\beta_{ij}, ~\gamma_{ij}, ~\delta_{ij}$, $i, ~j \in \mathbb{Z}_+, ~i \le 2, ~j \le 4$ meet the constraints below, which ensure the flexibility of the entire pattern. There are 30 constraints for 32 sector angles, hence roughly speaking, allowing two independent input sector angles. Details on the derivation of constraints over the sector angles of a unit are provided in Sections \ref{section: linear coupling}, \ref{section: equimodular coupling} and \ref{section: examples}.

Vertex type condition (half are anti-isogram/flat-foldable vertices, half are anti-deltoid II/straight-line vertices):
\begin{equation}
	\begin{dcases}
		\gamma_{11} = \pi - \alpha_{11}, ~\delta_{11} = \pi - \beta_{11}, ~ \gamma_{12} = \pi - \alpha_{12}, ~\delta_{12} = \pi - \beta_{12} \\
		\gamma_{13} = \pi - \alpha_{13}, ~\delta_{13} = \pi - \beta_{13}, ~ \gamma_{14} = \pi - \alpha_{14}, ~\delta_{14} = \pi - \beta_{14} \\
		\gamma_{21} = \pi - \beta_{21}, ~\delta_{21} = \pi -\alpha_{21}, ~\gamma_{22} = \pi - \beta_{22}, ~\delta_{22} = \pi -\alpha_{22} \\
		\gamma_{23} = \pi - \beta_{23}, ~\delta_{23} = \pi -\alpha_{23}, ~\gamma_{24} = \pi - \beta_{24}, ~\delta_{24} = \pi -\alpha_{24} \\
	\end{dcases}
\end{equation} 
Planarity condition of quad panels considering the periodicity of sector angles:
\begin{equation}
	\begin{dcases}
		\beta_{11}+\beta_{21}+\beta_{12}+\beta_{12} = 2\pi, ~\gamma_{11}+\gamma_{21}+\gamma_{12}+\gamma_{12} = 2\pi \\
		\delta_{12}+\delta_{22}+\delta_{13}+\delta_{23} = 2\pi, ~\alpha_{12}+\alpha_{22}+\alpha_{13}+\alpha_{23} = 2\pi \\ \beta_{13}+\beta_{23}+\beta_{14}+\beta_{14} = 2\pi, ~\gamma_{13}+\gamma_{23}+\gamma_{14}+\gamma_{14} = 2\pi \\
		\delta_{14}+\delta_{24}+\delta_{11}+\delta_{21} = 2\pi, ~\alpha_{14}+\alpha_{24}+\alpha_{11}+\alpha_{21} = 2\pi \\
	\end{dcases}
\end{equation}
Condition for being linear couplings:
\begin{equation}
	\begin{dcases}
	\dfrac{\sin \alpha_{21}}{\sin \beta_{21}} = \dfrac{\sin \alpha_{22}}{\sin \beta_{22}} \\
	\dfrac{\sin \alpha_{22}}{\sin \beta_{22}} = \dfrac{\sin \alpha_{23}}{\sin \beta_{23}} \\
	\dfrac{\sin \alpha_{23}}{\sin \beta_{23}} = \dfrac{\sin \alpha_{24}}{\sin \beta_{24}} \\
	\end{dcases}
\end{equation}
Condition on equal ratio for linear couplings:
\begin{equation}
	\begin{gathered}
		\begin{split}
			& \dfrac{\sin \dfrac{\beta_{11} - \gamma_{11}}{2}\sin \dfrac{\beta_{12} + \gamma_{12}}{2}}{\sin \dfrac{\beta_{11} + \gamma_{11}}{2}\sin \dfrac{\beta_{12} - \gamma_{12}}{2}} =  \mathrm{sign}\left(\dfrac{\pi-\beta_{21}-\alpha_{21}}{\pi-\beta_{22}-\alpha_{22}}\right) \sqrt{\dfrac{\sin(\beta_{21}+\alpha_{21})\sin(\beta_{22}-\alpha_{22})}{\sin(\beta_{21}-\alpha_{21})\sin(\beta_{22}+\alpha_{22})}}
		\end{split} \\
		\begin{split}
			& \dfrac{\sin \dfrac{\beta_{12} - \gamma_{12}}{2}\sin \dfrac{\beta_{13} + \gamma_{13}}{2}}{\sin \dfrac{\beta_{12} + \gamma_{12}}{2}\sin \dfrac{\beta_{13} - \gamma_{13}}{2}} =  \mathrm{sign}\left(\dfrac{\pi-\beta_{22}-\alpha_{22}}{\pi-\beta_{23}-\alpha_{23}}\right) \sqrt{\dfrac{\sin(\beta_{22}+\alpha_{22})\sin(\beta_{23}-\alpha_{23})}{\sin(\beta_{22}-\alpha_{22})\sin(\beta_{23}+\alpha_{23})}}
		\end{split} \\
		\begin{split}
			& \dfrac{\sin \dfrac{\beta_{13} - \gamma_{13}}{2}\sin \dfrac{\beta_{14} + \gamma_{14}}{2}}{\sin \dfrac{\beta_{13} + \gamma_{13}}{2}\sin \dfrac{\beta_{14} - \gamma_{14}}{2}} =  \mathrm{sign}\left(\dfrac{\pi-\beta_{23}-\alpha_{23}}{\pi-\beta_{24}-\alpha_{24}}\right) \sqrt{\dfrac{\sin(\beta_{23}+\alpha_{23})\sin(\beta_{24}-\alpha_{24})}{\sin(\beta_{23}-\alpha_{23})\sin(\beta_{24}+\alpha_{24})}} 
		\end{split}
	\end{gathered}
\end{equation}
Note that the two equations below will be implied from the above conditions, which also contributes to the flexibility condition of the entire quad-mesh rigid origami:
\begin{equation}
	\begin{gathered}
		\dfrac{\sin \alpha_{24}}{\sin \beta_{24}} = \dfrac{\sin \alpha_{21}}{\sin \beta_{21}} \\ \dfrac{\sin \dfrac{\beta_{14} - \gamma_{14}}{2}\sin \dfrac{\beta_{11} + \gamma_{11}}{2}}{\sin \dfrac{\beta_{14} + \gamma_{14}}{2}\sin \dfrac{\beta_{11} - \gamma_{11}}{2}} =  \mathrm{sign}\left(\dfrac{\pi-\beta_{24}-\alpha_{24}}{\pi-\beta_{21}-\alpha_{21}}\right) \sqrt{\dfrac{\sin(\beta_{24}+\alpha_{24})\sin(\beta_{21}-\alpha_{21})}{\sin(\beta_{24}-\alpha_{24})\sin(\beta_{21}+\alpha_{21})}}
	\end{gathered}
\end{equation}
A large library of quad-mesh rigid origami can be constructed by repetitively stitching together units composed of linear and equimodular couplings, with variations in vertex types, input sector angles, and crease length distributions, as illustrated in Figure \ref{fig: example}. The diversity of designs can be further expanded by switching some strips, as introduced in Section \ref{sec: introduction}. Notably, switching alternating strips makes all developable vertices non-developable, and vice versa. The sector angles of these examples were solved numerically. Both the sector angle solutions and the resulting folding motion are validated with an accuracy $\sim$ \(10^{-15}\) using double-precision inputs, comparable to those achieved for the Miura-ori, (anti)-V-hedra, and T-hedra. The folding motion is reconstructed from the sector angles and crease lengths, and its validity is further confirmed by measuring the coplanarity of each quadrilateral throughout the motion. Additional details can be found in \citet{he_sector-angle-periodic_2024}.

\subsection{Evidence for the limiting smooth surface} \label{subsection: evidence} 

In this section, we compile and examine the available information concerning the limiting smooth surface associated with a series of arbitrarily refined quad-mesh rigid origami generated from repetitive stitching. To the best of our knowledge, existing methods --- including (1) the standard convergence theorem for conjugate nets \citep[Chapter 5]{bobenko_discrete_2008}, (2) the approach for V-hedra \citep{izmestiev_voss_2025}, (3) the approach for T-hedra \citep{izmestiev_isometric_2024}, and (4) the asymptotic analysis of finite isometries of periodic surfaces \citep[Theorem 2]{nassar_effective_2024} --- are not directly applicable. The inapplicability arises for the following reasons: (1) the newly constructed patterns are not governed by systems of first-order partial difference equations; (2) the sector angles, along with the more sparsely constructed quad-meshes obtained from vertices at identical relative periodic positions across units, do not conform to the definitions of V-hedra or T-hedra; and (3) although one row and one column of crease lengths may be uniformly distributed, the other crease lengths vary across units, rendering the resulting quad-mesh non-periodic. Furthermore, the approach in \citet{nassar_effective_2024} relies on the existence of a supporting plane, which is not available in our case.

To tackle this challenge, we first observed the existence of zig-zag mode along one coordinate direction among all the examples, which is the ruling direction illustrated in Figure \ref{fig: example}. By selecting one row and one column of crease lengths as the \textit{input} crease length, we conjecture that the resulting limiting surface will always be a ruled surface, regardless of the input crease length distribution, as long as it varies smoothly without large fluctuations (Figure~\ref{fig: limiting surface}). This observation is numerically examined from calculating the coefficient of determination $R^2$ of linear regression for each column or row of vertices on the ruling direction. Specifically, $R^2$ is the ratio of the explained variance to the total variance calculated from the steps below. Let $x(i) = [x_1(i);~x_2(i);~x_3(i)] \in \mathbb{R}^3, ~i \in \mathbb{Z}_+, ~i \le n, ~n \in \mathbb{Z}_+$ be a data set of $n$ points in $\mathbb{R}^3$,
\begin{equation*}
	X = \begin{bmatrix}
		x(1) & x(2) & \cdots & x(n)
	\end{bmatrix} \mathrm{~is~a~} 3 \times n \mathrm{~coordinate~matrix}
\end{equation*}
\begin{equation*}
	\overline{X} = \begin{bmatrix}
		\left. \sum \limits_{i = 1}^{n} x_1(i) \middle/ n \right. \\
		\left. \sum \limits_{i = 1}^{n} x_2(i) \middle/ n \right. \\
		\left. \sum \limits_{i = 1}^{n} x_3(i) \middle/ n \right. \\
	\end{bmatrix} \mathrm{~is~the~mean~value~of~coordinates}
\end{equation*}
\begin{equation*}
	\mathrm{d}X = \begin{bmatrix}
		x(1) - \overline{X}  & x(2) - \overline{X} & \cdots & x(n) - \overline{X}
	\end{bmatrix} \mathrm{~is~the~residual}
\end{equation*}
\begin{equation*}
	C = \dfrac{\mathrm{d}X \mathrm{d}X^\mathrm{T}}{n-1} \mathrm{~is~the~} 3 \times 3 \mathrm{~variance}\text{-}\mathrm{covariance~matrix}
\end{equation*}
We apply a spectral decomposition to $C$, writing it as $C = A D A^\mathrm{T}$. The direction of the best fit line is the first column $A(:, ~1)$ of $A$. The line of best fit is hence $\overline{X} + tA(:, ~1), ~t \in \mathbb{R}$. The coefficient of determination is:
\begin{equation*}
	R^2 = \dfrac{D(1, ~1)}{\mathrm{trace}(D)}
\end{equation*}
The closer $R^2$ is to 1, the more closely the data aligns with a perfect linear relation. Below, we report the \textbf{minimum} $R^2$ value across all rulings for the constructions shown in Figure~\ref{fig: example} and \ref{fig: unit} at a representative folded state. The magnitude and trend of $R^2$ at other folded states are similar and can be observed in the accompanying script~\citep{he_sector-angle-periodic_2024}.

\begin{figure}[t]
	\noindent \begin{centering}
		\includegraphics[width=1\linewidth]{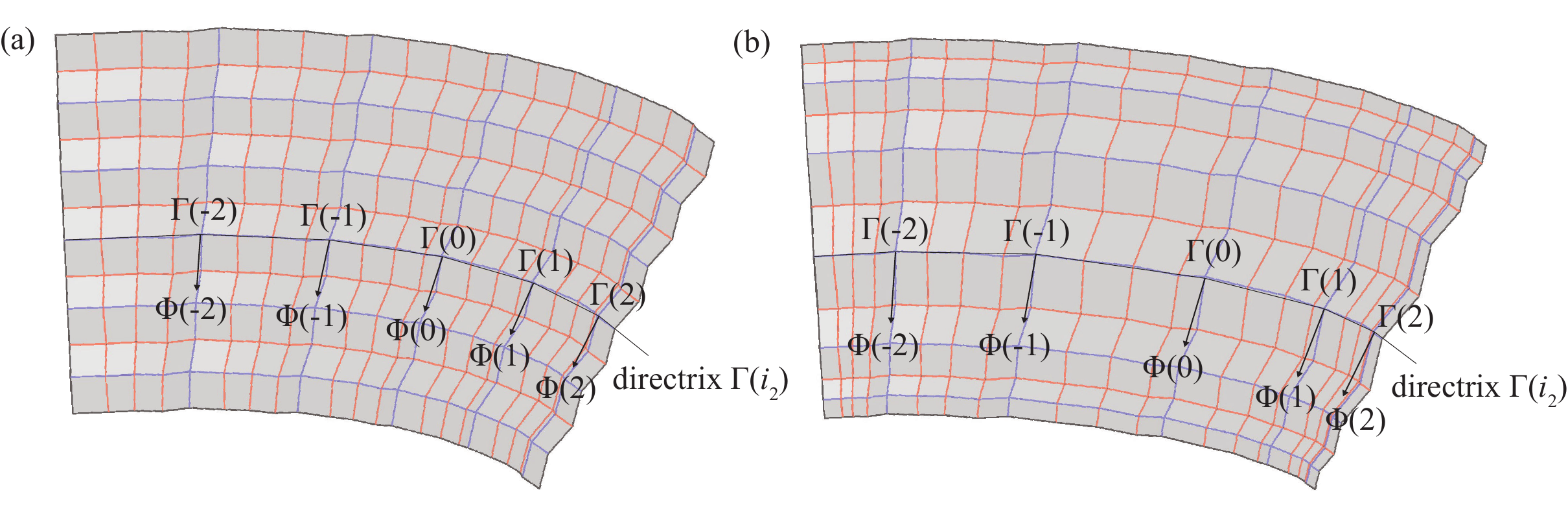}
		\par\end{centering}
	
	\caption{\label{fig: limiting surface} The structure of the discrete surface generated from repetitive stitching. (a) shows a uniform input crease length distribution; (b) shows a quadratic input crease length distribution, both with 25 units. As the number of units increases, the repeated pattern causes the hypothesized directrix $\Gamma$ to approach a pseudo-helix (see \eqref{eq: pesudo-helix}). In the longitudinal direction, the zig-zag configuration produces rulings that maintain a constant angle with the directrix (see \eqref{eq: constant angle}).}
\end{figure}

\begin{description}
	\item[Figure \ref{fig: example}(a)] at folding angle $-60^\circ$ 
	\begin{center}
		\begin{tabular}{ |c|c|c|c|c|c| } 
			\hline
			number of units & 1 & 4 & 9 & 25 & 49 \\ 
			\hline
			uniform input crease length  & 0.99636 & 0.99886 & 0.99945 & 0.99978 & 0.99988 \\ 
			\hline
			quadratic input crease length & 0.99700 & 0.99878 & 0.99939 & 0.99976 & 0.99987 \\ 
			\hline
		\end{tabular}
	\end{center}
	\item[Figure \ref{fig: example}(b)] at folding angle $-30^\circ$ 
	\begin{center}
		\begin{tabular}{ |c|c|c|c|c|c| } 
			\hline
			number of units & 1 & 4 & 9 & 25 & 49 \\ 
			\hline
			uniform input crease length & 0.97546 & 0.99103 & 0.99540 & 0.99813 & 0.99899 \\ 
			\hline
			quadratic input crease length & 0.98050 & 0.99052 & 0.99490 & 0.99788 & 0.99885 \\ 
			\hline
		\end{tabular}
	\end{center}
	\item[Figure \ref{fig: example}(c)] at folding angle $-30^\circ$ 
	\begin{center}
		\begin{tabular}{ |c|c|c|c|c|c| } 
			\hline
			number of units & 1 & 4 & 9 & 25 & 49 \\ 
			\hline
			uniform input crease length & 0.97569 & 0.99111 & 0.99545 & 0.99815 & 0.99900 \\ 
			\hline
			quadratic input crease length & 0.98063 & 0.99061 & 0.99495 & 0.99790 & 0.99886 \\ 
			\hline
		\end{tabular}
	\end{center}
	\item[Figure \ref{fig: example}(d)] at folding angle $-30^\circ$ 
	\begin{center}
		\begin{tabular}{ |c|c|c|c|c|c| } 
			\hline
			number of units & 1 & 4 & 9 & 25 & 49 \\ 
			\hline
			uniform input crease length & 0.97257 & 0.98935 & 0.99423 & 0.99740 & 0.99846 \\ 
			\hline
			quadratic input crease length & 0.97825 & 0.98880 & 0.99361 & 0.99705 & 0.99825 \\ 
			\hline
		\end{tabular}
	\end{center}
	\item[Figure \ref{fig: example}(e)] at folding angle $-30^\circ$ 
	\begin{center}
		\begin{tabular}{ |c|c|c|c|c|c| } 
			\hline
			number of units & 1 & 4 & 9 & 25 & 49 \\ 
			\hline
			uniform input crease length & 0.97576 & 0.99113 & 0.99545 & 0.99815 & 0.99900 \\ 
			\hline
			quadratic input crease length & 0.98068 & 0.99064 & 0.99496 & 0.99790 & 0.99886 \\ 
			\hline
		\end{tabular}
	\end{center}
	\item[Figure \ref{fig: example}(f)] at folding angle $-30^\circ$ 
	\begin{center}
		\begin{tabular}{ |c|c|c|c|c|c| } 
			\hline
			number of units & 1 & 4 & 9 & 25 & 49 \\ 
			\hline
			uniform input crease length & 0.96847 & 0.98892 & 0.99484 & 0.99800 & 0.99895 \\ 
			\hline
			quadratic input crease length & 0.98123 & 0.98913 & 0.99448 & 0.99776 & 0.99881 \\ 
			\hline
		\end{tabular}
	\end{center}
	\item[Figure \ref{fig: unit}] at folding angle $-30^\circ$ 
	\begin{center}
		\begin{tabular}{ |c|c|c|c|c|c| } 
			\hline
			number of units & 1 & 4 & 9 & 25 & 49 \\ 
			\hline
			uniform input crease length & 0.99852 & 0.99956 & 0.99979 & 0.99992 & 0.99996 \\ 
			\hline
			quadratic input crease length & 0.99869 & 0.99949 & 0.99976 & 0.99990 & 0.99995 \\ 
			\hline
		\end{tabular}
	\end{center}
\end{description}

On top of the ruled surface assumption, in Figure \ref{fig: limiting surface}(a), where the input crease length distribution is uniform, we write the discrete surface in coordinate $i = (i_1, ~i_2) \in \mathbb{Z}^2$:
\begin{equation*} 
	x(i_1, ~i_2) = \Gamma(i_2) + i_1\Phi(i_2) 
\end{equation*}
Each $x(i_1, ~i_2)$ is the location of vertex at the same relative position of unit $(i_1, ~i_2)$. $\Gamma(i_2)$ is the directrix. $\Phi(i_2)$ is along the ruling direction. There is no restriction on the choice of such representative vertex $x(i_1, ~i_2)$. Furthermore, let 
\begin{equation*}
	\triangle \Gamma(i_2) = \Gamma(i_2+1) - \Gamma(i_2)
\end{equation*}
The periodicity of sector angles implies the periodicity of folding angles. Since $\triangle \Gamma(i_2)$ is determined by the sector angles, folding angles, and the input crease lengths, it is a spatial vector with constant azimuthal rotation and translational span. Consequently, $\Gamma(i_2)$ forms a helical polyline. We could write $\triangle \Gamma(i_2)$ in the form below, up to a combination of rotation and translation:
\begin{equation} \label{eq: delta i_2}
	\triangle \Gamma(i_2) = \begin{bmatrix}
		a \cos (i_2+1) \theta  \\
		a \sin (i_2+1) \theta \\
		b (i_2+1)
	\end{bmatrix} - \begin{bmatrix}
		a \cos i_2 \theta  \\
		a \sin i_2 \theta \\
		b i_2
	\end{bmatrix} = \begin{bmatrix}
		-2a \sin \left(i_2\theta + \dfrac{\theta}{2}\right) \sin \dfrac{\theta}{2}  \\[8pt]
		2a \cos \left(i_2\theta + \dfrac{\theta}{2}\right) \sin \dfrac{\theta}{2} \\[8pt]
		b 
	\end{bmatrix}
\end{equation}
The parameter $a$ is the radius of the helix, $\theta$ and $b$ are the azimuthal rotation and the translational span of adjacent vertices on the helix. These parameters can be determined by the three-dimensional configuration of the origami structure at any folded state.

Since the ruled surface assumption remains valid even under a generic input crease length distribution, we hypothesize that the periodicity of sector and folding angles has a dominant influence on the overall shape of the limiting surface, compared to the effect of varying crease lengths. Despite the geometric complexity, our goal is to obtain a simple symbolic expression that captures the global features of the limiting surface as faithfully as possible. 

To this end, we assume that for a generic input crease length distribution, each vector segment $\triangle \Gamma(i_2)$ still progresses in a helical fashion, but is modulated by a scaling factor $f(i_2)$ to account for variations in crease length:
\begin{equation} \label{eq: pesudo-helix}
	f(i_2) = \dfrac{\| \triangle \Gamma(i_2) \|}{\| \triangle \Gamma(0) \|}, \quad \triangle \Gamma(i_2) = f(i_2) \begin{bmatrix}
		-2a \sin \left(i_2\theta + \dfrac{\theta}{2}\right) \sin \dfrac{\theta}{2}  \\[8pt]
		2a \cos \left(i_2\theta + \dfrac{\theta}{2}\right) \sin \dfrac{\theta}{2} \\[8pt]
		b 
	\end{bmatrix}
\end{equation}
where $f(i_2) \equiv 1$ corresponds to the case of uniform input crease lengths.

As a further consequence of the dominance of periodic sector and folding angles, we assume that at every point $x(i_1, ~i_2)$, the angle between the directrix and the ruling remains constant. This angle condition greatly simplifies the determination of all $\Phi(i_2)$.
\begin{equation} \label{eq: constant angle}
	\dfrac{\triangle \Gamma(i_2) \cdot \Phi(i_2)}{\|\triangle \Gamma(i_2)\| \|\Phi(i_2)\|} = \mathrm{Const.} \in [0, 1)
\end{equation}

In conclusion, for any input crease length, the discrete surface has the parametrization below:
\begin{equation}
	x(i_1, ~i_2) = \Gamma(0) + \sum \limits_{j = 0}^{i_2} \triangle \Gamma(j) + i_1\Phi(i_2) 
\end{equation}

Now as the mesh is arbitrarily refined, we conjecture that the limiting smooth surface is in the form of \eqref{eq: special ruled surface}, up to a combination of rotation and translation. The directrix $\Gamma(u_2)$ is now a smooth curve, $f(u_2)$ also becomes a smooth function. The directrix meets the rulings at a constant angle for all $u_2$. When $f(v) \equiv 1$, $\Gamma(u_2)$ becomes a helix.

\section{Discussion} \label{section: discussion}

Our results represent an initial step in advancing the form-finding capabilities of quad-mesh rigid origami beyond the commonly explored (anti-)V-hedra and T-hedra. 

\subsection*{Revisiting the levels of surface approximation}

One notable difference between the first and second levels of surface approximation is the zig-zag mode in quad-mesh rigid origami. We claim that there is no smooth analogue for developable quad-mesh rigid origami, such as the Miura-ori, anti-V-hedra and developable T-hedra. To elucidate this, it is helpful to introduce the concept of \textit{mountain-valley assignment}. By assigning an orientation to the discrete surface, we measure the dihedral angle at each crease and subtract it from $\pi$ to determine the folding angle. Creases with negative folding angles are called \textit{mountain creases}, where the paper bends away from the observer from the specified orientation. Conversely, creases with positive folding angles are called \textit{valley creases}, where the paper bends towards the observer from the specified orientation. At every developable vertex, the numbers of mountain and valley creases are 3-to-1 or 1-to-3. At every vertex where the sum of sector angles is less than $2\pi$, the mountain and valley creases can be 4-to-0, 3-to-1, 1-to-3 or 0-to-4 in different folded states. At every vertex where the sum of sector angles is more than $2\pi$, the mountain and valley creases can be 3-to-1, 2-to-2, or 1-to-3 in different folded states. This counting of mountain-valley assignments for a generic degree-4 vertex can be checked symbolically from \cite{he_real_2023}. The coordinate curves and tangent planes around vertices with 3-to-1 or 1-to-3 mountain-valley assignment oscillate when being arbitrarily refined, which is not a feature of a smooth surface. The 3-to-1 and 1-to-3 vertices introduce a zig-zag mode, as illustrated in Figure \ref{fig: quad mesh introduction}(a), \ref{fig: quad mesh introduction}(d) and \ref{fig: quad mesh introduction}(f), though they are not the only cause. In Figure \ref{fig: quad mesh introduction}(c), alternating rows of 2-to-2 and 0-to-4 vertices can also create this zig-zag mode. Patterns with a smooth analogue, as visualized in Figure \ref{fig: quad mesh introduction}(b) and \ref{fig: quad mesh introduction}(e), have a uniform mountain-valley assignment for all interior vertices following 4-to-0, 2-to-2 and 0-to-4 assignments. Our conjecture in \eqref{eq: special ruled surface} attempts to capture the apparent curvature superimposed on the zig-zag mode.

In practical origami-based design, we often aim for the metric- and curvature-related properties of the origami structure to closely approximate those of the target surface — beyond merely achieving closeness in distance. The metric-related properties include 1a) arc lengths of coordinate curves; 2a) arc lengths of geodesics; 3a) surface area. Curvature-related properties include 1b) curvature and torsion of coordinate curves; 2b) curvature and torsion of geodesics 3b) normal vector field; 4b) mean curvature; 5b) Gaussian curvature. In classic differential geometry, there are famous examples such as the `Staircase paradox' and the `Schwarz lantern', showcasing the non-convergence of length and area upon the convergence in distance. From classical mathematical analysis, if a series of discrete curves or surfaces approaches a smooth curve or surface, and all the vertices are exactly on the smooth curve or surface, the tangent plane, metric- and curvature-related properties will converge. We could see that the Staircase and the Schwarz lantern both have a zig-zag mode where the vertices of discrete curves and surfaces are not exactly on the target curve/surface. The examples in Figures \ref{fig: example} and \ref{fig: unit} also exhibit this zig-zag mode, demonstrating non-convergence of the properties listed in 1a) through 3a) and 1b) through 5b). However, this does not imply that all new patterns created through repetitive stitching will exhibit this zig-zag mode.

Several relevant details are worth noting. \citet{morvan_approximation_2004} showed that convergence of the normal vector fields implies convergence of surface area. \citet{hildebrandt_convergence_2006} considerably generalized the result in \citet{morvan_approximation_2004}: upon the convergence in distance, the convergence of metric tensor (the first fundamental form), surface area, normal vector field and mean curvature (the cotangent formula, Laplace-Beltrami operator) are equivalent. Once this convergence is met, it could be further inferred that arclength of coordinate curves/geodesics and Gaussian curvature will converge since they are dependent on the first fundamental form.  \citet{bauer_uniform_2010} showed that discrete principal curvatures computed from a series of curvature line nets uniformly converge to the principal curvatures of the limiting smooth surface. These results are not exhaustive, with much remaining to be explored.

\subsection*{Connection with numerical methods}

How challenging is it to numerically search for a quad-mesh rigid origami with accuracy comparable to the Miura-ori? \citet{jiang_quad_2024} addressed this problem by developing three computational methods capable of finding both planar and spatial flexible quad-meshes and modelling their flexes. Their numerical solutions exhibit precision with errors --- measured as the relative changes in crease lengths --- on the order of \(10^{-5}\) to \(10^{-7}\). These errors represent the relative changes in crease lengths. Notably, they proposed an effective strategy: starting from known symbolic variations of quad-mesh rigid origami, applying small perturbations, and subsequently optimizing for flexibility. Our result would provide a library of new examples that can serve as initial configurations for such numerical methods, potentially enlarging the outcome from computational searches.     

\subsection*{New semi-discrete quad-mesh rigid origami and curved crease origami}

The newly introduced infinitely refinable quad-mesh rigid origami presents strong potential for advancing the design of novel semi-discrete quad-mesh rigid origami and curved crease origami — extending beyond the existing frameworks based on (anti)-V-hedra and T-hedra. For recent advancements within these established frameworks, see \citet{sharifmoghaddam_generalizing_2023, liu_design_2024}.

A semi-discrete quad-mesh rigid origami refines the mesh in only one direction, replacing straight creases in that direction with smooth, non-intersecting curves. The result is a flexible, piecewise-smooth surface segmented by curved creases. More generally, curved crease origami refers to the isometric deformation of such piecewise-smooth surfaces joined along curves — i.e., curved creases. This broader category includes semi-discrete quad-mesh rigid origami but also encompasses configurations with intersecting curved creases.

In rigid-ruling folding of curved crease origami, the crease-ruling pattern preserves (i.e. rulings do not `slide' on the surface) during the folding motion \citep{demaine_conic_2018, mundilova_rigid-ruling_2024}. For example, bending a cylindrical surface to a cone is an isometric deformation but not a rigid-ruling folding. The semi-discrete quad-mesh rigid origami formed by the new infinitely refinable patterns introduced in this article may also potentially form a rigid-ruling folding.

\subsection*{Beyond using repetitive stitching} 

The periodicity of sector angles in our proposed construction not only reduces the number of constraints, making it fewer than the number of sector angles, but also plays a crucial role in defining the limiting smooth surface. However, this represents only the most basic symmetry. There remains substantial potential for exploration beyond periodicity.

\subsection*{Self-intersection of the crease pattern}

Self-intersection occurs when creases intersect at points other than the specified vertices, a scenario that can emerge during mesh refinement. While preventing self-intersection is essential for practical pattern design, allowing it can provide a method to discretize surface with self-intersecting coordinate curves \citep{kilian_interactive_2024}. Resolving this issue requires techniques that lie beyond the scope of this article, and we plan to explore it in future research.

\section*{Acknowledgement}

This work is supported by the Japan Science and Technology Agency – Core Research for Evolutional Science and Technology Grant No.~JPMJCR1911. 

We thank Ivan Izmestiev, Kiumars Sharifmoghaddam, Georg Nawratil, and Hellmuth Stachel for insightful discussions during the Special Semester on Rigidity and Flexibility at the Johann Radon Institute for Computational and Applied Mathematics, Johannes Kepler University Linz, Austria

We thank Hussein Nassar and Frederic Marazzato for insightful discussions during the “Geometry of Materials” Workshop at the Institute for Computational and Experimental Research in Mathematics (ICERM), held in Providence, RI, as part of the semester program on “Geometry of Materials, Packings, and Rigid Frameworks.”

\appendix

\part*{Supplementary Material}

This supplementary material serves as an extensive resource for understanding the mathematical principles underlying the new quad-mesh rigid origami emphasized in the main text. In our setup, the sum of sector angles at every interior vertex is \textbf{not necessarily} $2\pi$, which means our discussion includes but is not restricted to developable origami. We include all the necessary derivations towards common quad-mesh origami variants --- (anti-)V-hedra and T-hedra, and the more generalized variations reported in the main text --- linear couplings and the equimodular couplings. All the derivations are presented in a detailed manner, ensuring accessibility for researchers across diverse disciplines.

The content is divided into two parts: Part I covers foundational concepts around coordinate nets (i.e. surface patches or parametrization) in both differential geometry and discrete differential geometry. Table \ref{tab: notation} lists the pertinent notations used throughout this supplementary material. Section \ref{section: isometry} is a supplement to the information of geodesics and Christoffel symbol provided in \citet{do_carmo_differential_2016}. Section \ref{section: coordinate net} introduces common coordinate nets formed by coordinate curves $u_1 = \mathrm{Const}$ and $u_2 = \mathrm{Const}$. Section \ref{section: differential equation} introduces the well-posed initial condition to obtain these smooth coordinate nets from solving a partial differential equation. Furthermore, in computer graphics and computational mechanics, we are naturally seeking for a `nice' discretization of these coordinate nets. It leads to the introduction on discrete curves and surfaces, together with the matching discrete nets to the aforementioned smooth nets in Section \ref{section: discrete curvature} and Section \ref{section: discrete net}. In parallel, Section \ref{section: difference equation} introduces the well-posed initial condition to construct these discrete nets as the solution of a partial difference equation. After all these preparation, in Section \ref{section: ddg convergence} we discuss the convergence of a series of discrete nets to a smooth net as the mesh is refined. The above information on discrete nets is mainly from \cite{bobenko_discrete_2008}. 

Part II is about the information on quad-mesh rigid origami. We concern the continuous isometric (distance-preserving) deformation of both the quad-mesh rigid origami and its Gauss map. The flexibility of quad-mesh rigid origami is introduced in Section \ref{section: loop condition}. Common variations, including V-hedra and T-hedra, are detailed in Sections \ref{section: V-hedra} and \ref{section: T-hedra}, respectively. For more generalized variations --- linear couplings and equimodular couplings --- all relevant details are presented in Sections \ref{section: linear coupling}, \ref{section: equimodular coupling} and \ref{section: examples}. 

\begin{longtable} {@{}p{0.3\textwidth}@{}p{0.7\textwidth}@{}} 
	\caption{Notations} \label{tab: notation} \\
	\hline
	\textbf{Geometrical objects} & \\
	$X, ~Y$   & a curve or a surface in $\mathbb{R}^3$ \\
	$x,~y$ & arbitrary or fixed points in $X$ or $Y$, dependent on context \\
	$x = (x_1,~x_2,~\dots,~x_n)$ & coordinates of $x~(n \in \mathbb{Z}_+)$ \\
	$I^n$   & an $n$ dimensional open cube in $\mathbb{R}^n~(n \in \mathbb{Z}_+)$. $I$ refers to $(0, ~1)$ by default. \\
	$O$ & an open set \\
	$O(x)$ & a neighbourhood at $x$ \\
	$\Gamma,~\Phi, ~\Psi$ & charts of a curve or a surface. $\Gamma$ is usually used for a curve. \\
	$N$ & a unit normal vector \\
	$\mathrm{I}, ~\mathrm{I}_{11}, ~\mathrm{I}_{12}, ~\mathrm{I}_{22}$ & the first fundamental form and its components \\
	$\mathrm{II}, ~\mathrm{II}_{11}, ~\mathrm{II}_{12}, ~\mathrm{II}_{22}$ & the second fundamental form and its components \\
	$\mathrm{III}, ~\mathrm{III}_{11}, ~\mathrm{III}_{12}, ~\mathrm{III}_{22}$ & the third fundamental form and its components \\
	$\kappa, ~\tau$ & curvature and torsion for a curve  \\
	$\kappa_\mathrm{n}, ~\kappa_\mathrm{g}$ & the normal curvature and geodesic curvature for a curve on a surface \\
	$\kappa_1, ~\kappa_2, ~\kappa_\mathrm{H}, ~\kappa_\mathrm{G}$ & the principal curvatures, mean curvature, and Gaussian curvature for a surface. \\
	\hline
	\textbf{Parameters} & \\
	$i,~j,~k,~l$ & flexible positive integers or free indices \\
	$m,~n$ & fixed positive integers \\
	$a,~b,~c$ & scalar or vector parameters in $\mathbb{R}^n$ $(n \in \mathbb{Z}_+)$ \\ 
	$a = (a_1,~a_2,~\dots,~a_n)$ & coordinates of $a~(n \in \mathbb{Z}_+)$ \\
	$\epsilon,~\delta$ & real numbers in all forms of $\epsilon-\delta$ expressions \\
	$t,~u, ~v$ & parameters for a curve or a surface \\
	$t = (t_1,~t_2,~\dots,~t_m)$ & coordinates of $t~(m \in \mathbb{Z}_+)$ \\
	$s$ & arc length parameter for a curve \\
	\hline
	\label{tab: notation usage}
\end{longtable}

\part{Preliminaries in (discrete) differential geometry}

\section{Geodesic and the Christoffel symbol} \label{section: isometry}

Let $X$ be a surface with local chart $\Phi: u = (u_1, ~u_2) \in I^2 \rightarrow x = (x_1, ~x_2, ~x_3) \in X \subset \mathbb{R}^3$. In the calculations below, we apply the following \textbf{regularity condition} to all curves and surfaces by default: 1) all local charts are analytic, meaning they are locally represented by convergent power series. As a result, these charts are smooth (have arbitrary order of partial derivatives), with both the charts and their partial derivatives being bounded; 2) the Jacobian $\dif x / \dif u$ is of full rank. The normal vector field $N$ on $X$ is:
\begin{equation*}
	N = \dfrac{\dfrac{\partial x}{\partial u_1} \times \dfrac{\partial x}{\partial u_2}}{\left|\left| \dfrac{\partial x}{\partial u_1} \times \dfrac{\partial x}{\partial u_2} \right|\right|}
\end{equation*}
$N: X \rightarrow S^2$ is also called \textit{the Gauss map}, which can be interpreted as translating all the normal vectors of a surface to a unit sphere.

A curve $\Gamma: t \in I \rightarrow x \in X$ on the surface $X$ is a \textit{geodesic} if there is no `lateral acceleration':
\begin{equation*}
	\dfrac{\dif^2 x}{\dif t^2} \cdot \left( N \times \dfrac{\dif x}{\dif t} \right) = 0
\end{equation*}
since $\|\dif x / \dif s\| = 1 \Rightarrow (\dif x / \dif s) \cdot (\dif^2 x / \dif s^2) = 0$. In the arc length parametrization the above condition is:
\begin{equation*}
	\begin{gathered}
		\dfrac{\dif^2 x}{\dif s^2} \times N  = 0 ~~\Leftrightarrow~~ \dfrac{\dif^2 x}{\dif s^2} = \kappa N, ~~\kappa: I \rightarrow \mathbb{R} \\
		\kappa \mathrm{~is~the~curvature~of~a~geodesic~}\Gamma
	\end{gathered}
\end{equation*}
The velocity $\dif x / \dif s$ along a geodesic $\Gamma$, as a vector field, is hence said to be \textit{parallel} on the surface $X$. Clearly, a straight line contained in a surface is a geodesic. Being geodesic is a necessary condition for the shortest path joining two points on a surface.

\begin{figure}[t]
	\noindent \begin{centering}
		\includegraphics[width=0.7\linewidth]{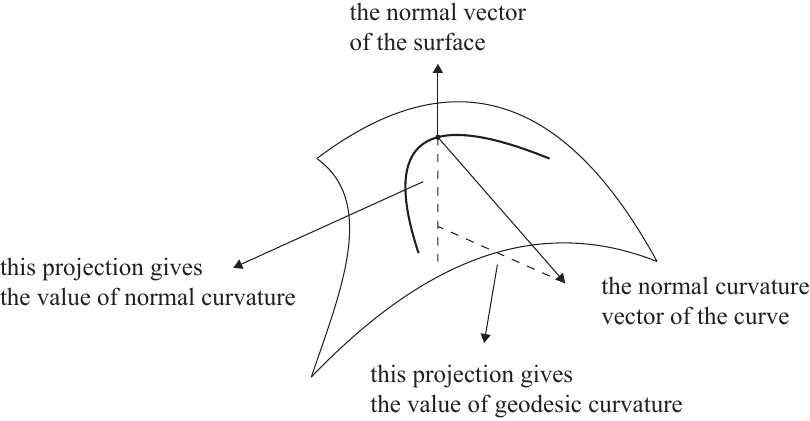}
		\par\end{centering}
	
	\caption{\label{fig: geodesic curvature}Illustration of the geodesic curvature.}
\end{figure}

Furthermore, for any curve $\Gamma \subset X$ we have:
\begin{equation*}
	\dfrac{\dif^2 x}{\dif s^2} = \dfrac{1}{\left\| \dfrac{\dif x}{\dif t}\right\|} \dfrac{\dif}{\dif t} \left( \dfrac{\dfrac{\dif x}{\dif t}}{\left\| \dfrac{\dif x}{\dif t}\right\|} \right) = \dfrac{1}{\left\| \dfrac{\dif x}{\dif t}\right\|^2} \left( \dfrac{\dif^2 x}{\dif t^2} - \dfrac{ \dfrac{\dif x}{\dif t} \cdot \dfrac{\dif^2 x}{\dif t^2}  }{\left\| \dfrac{\dif x}{\dif t}\right\|^2} \dfrac{\dif x}{\dif t} \right) 
\end{equation*}
\begin{equation*}
	\kappa = \left\| \dfrac{\dif^2 x}{\dif s^2}\right\| = \dfrac{ \left\| \dfrac{\dif x}{\dif t} \times \dfrac{\dif^2 x}{\dif t^2}\right\| }{\left\| \dfrac{\dif x}{\dif t}\right\|^3}
\end{equation*}
The (signed) \textit{normal curvature} $\kappa_\mathrm{n}$ is the length of the projection of the acceleration onto the surface normal vector. 
\begin{equation}
	\kappa_\mathrm{n} =  \dfrac{\dif^2 x}{\dif s^2} \cdot N  = \dfrac{ \dfrac{\dif^2 x}{\dif t^2} \cdot N }{\left\| \dfrac{\dif x}{\dif t}\right\|^2} 
\end{equation}
The (signed) \textit{geodesic curvature} $\kappa_\mathrm{g}$ is the length of the projection of the acceleration onto the tangent plane:
\begin{equation} \label{eq: geodesic curvature}
	\kappa_\mathrm{g} = \dfrac{\dif^2 x}{\dif s^2} \cdot \left(N \times \dfrac{\dif x}{\dif s}\right) = \dfrac{ \dfrac{\dif^2 x}{\dif t^2} \cdot \left(N \times \dfrac{\dif x}{\dif t}\right) }{\left\| \dfrac{\dif x}{\dif t}\right\|^3}
\end{equation}
The above definitions lead to
\begin{equation*}
	\kappa^2 = \kappa_\mathrm{n}^2 + \kappa_\mathrm{g}^2
\end{equation*}
Usually we think the curvature or the geodesic curvature is positive if the acceleration $\dif^2 x/\dif s^2$ is rotated counterclockwise from the velocity $\dif x/\dif s$. This aligns with the right-hand coordinate system such that the normal vector pointing outwards the surface is positive. In other words, the geodesic curvature is the curvature measured from the `viewpoint' on the surface', which further explains that a geodesic is the analogue of a line on a plane.
\begin{equation*}
	\Gamma \mathrm{~is~geodesic~} ~\Leftrightarrow~ \kappa_\mathrm{g} = 0 \mathrm{~~along~the~curve~}
\end{equation*}

The \textit{Christoffel symbol} $\Gamma_{11}^i, ~\Gamma_{12}^i, ~\Gamma_{22}^i~(i = 1, ~2)$ denotes how $(\partial^2 x / \partial u_1^2, ~\partial^2 x / \partial u_1 \partial u_2, ~\partial^2 x / \partial u_2^2)$ is linearly represented by the non-orthogonal frame $(\partial x / \partial u_1, ~\partial x / \partial u_2, ~N)$:
\begin{equation} \label{eq: Christoffel symbol}
	\begin{cases}
		\dfrac{\partial N}{\partial u_1} = a_{11}\dfrac{\partial x}{\partial u_1} + a_{21} \dfrac{\partial x}{\partial u_2} \\[8pt]
		\dfrac{\partial N}{\partial u_2} = a_{12}\dfrac{\partial x}{\partial u_1} + a_{22} \dfrac{\partial x}{\partial u_2} \\[8pt]
		\dfrac{\partial^2 x}{\partial u_1^2} = \Gamma^1_{11} \dfrac{\partial x}{\partial u_1} + \Gamma^2_{11} \dfrac{\partial x}{\partial u_2} + b_{11} N \\[8pt] 
		\dfrac{\partial^2 x}{\partial u_1 \partial u_2} = \Gamma^1_{12} \dfrac{\partial x}{\partial u_1} + \Gamma^2_{12} \dfrac{\partial x}{\partial u_2} + b_{12} N \\[8pt] 
		\dfrac{\partial^2 x}{\partial u_2^2} = \Gamma^1_{22} \dfrac{\partial x}{\partial u_1} + \Gamma^2_{22} \dfrac{\partial x}{\partial u_2} + b_{22} N \\[8pt] 
	\end{cases} 
\end{equation}
The first two equations in matrix form is:
\begin{equation*}
	\begin{bmatrix}
		\dfrac{\partial N}{\partial u_1} & \dfrac{\partial N}{\partial u_2}
	\end{bmatrix} = \begin{bmatrix}
		\dfrac{\partial x}{\partial u_1} & \dfrac{\partial x}{\partial u_2}
	\end{bmatrix} \begin{bmatrix}
		a_{11} & a_{12} \\
		a_{21} & a_{22}
	\end{bmatrix}
\end{equation*}
Note that, 
\begin{equation*}
	||N|| \equiv 1 ~~\Rightarrow~~ \dfrac{\dif N}{\dif u} \cdot N \equiv 0 
\end{equation*}
By taking the dot product of $\partial x / \partial u_1$ and $\partial x / \partial u_2$ with both sides of \eqref{eq: Christoffel symbol}, we obtain the equations below using the components of the first and second order fundamental form,
\begin{equation*}
	\begin{dcases}
		a_{11}\mathrm{I}_{11} + a_{21} \mathrm{I}_{12} = \dfrac{\partial x}{\partial u_1} \cdot \dfrac{\partial N}{\partial u_1} = \dfrac{\partial}{\partial u_1} \left( \dfrac{\partial x}{\partial u_1} \cdot N \right) - \mathrm{II}_{11}  = - \mathrm{II}_{11} \\
		a_{11} \mathrm{I}_{12} + a_{21}\mathrm{I}_{22} = \dfrac{\partial x}{\partial u_2} \cdot \dfrac{\partial N}{\partial u_1} = \dfrac{\partial}{\partial u_1} \left( \dfrac{\partial x}{\partial u_2} \cdot N \right) - \mathrm{II}_{12} = - \mathrm{II}_{12}\\
	\end{dcases}
\end{equation*}
\begin{equation*}
	\begin{dcases}
		a_{12}\mathrm{I}_{11} +a_{22} \mathrm{I}_{12} = \dfrac{\partial x}{\partial u_1} \cdot \dfrac{\partial N}{\partial u_2} = \dfrac{\partial}{\partial u_2} \left( \dfrac{\partial x}{\partial u_1} \cdot N \right) - \mathrm{II}_{12} = - \mathrm{II}_{12}\\
		a_{12} \mathrm{I}_{12} + a_{22}\mathrm{I}_{22} = \dfrac{\partial x}{\partial u_2} \cdot \dfrac{\partial N}{\partial u_2} = \dfrac{\partial}{\partial u_2} \left( \dfrac{\partial x}{\partial u_2} \cdot N \right) - \mathrm{II}_{22} = - \mathrm{II}_{22} \\
	\end{dcases}
\end{equation*}
\begin{equation*}
	\begin{dcases}
		\Gamma^1_{11}\mathrm{I}_{11} + \Gamma^2_{11} \mathrm{I}_{12} = \dfrac{\partial x}{\partial u_1} \cdot \dfrac{\partial^2 x}{\partial u_1^2} = \dfrac{1}{2} \dfrac{\partial \mathrm{I}_{11}}{\partial u_1} \\
		\Gamma^1_{11}\mathrm{I}_{12} + \Gamma^2_{11} \mathrm{I}_{22} = \dfrac{\partial x}{\partial u_2} \cdot \dfrac{\partial^2 x}{\partial u_1^2} = \dfrac{\partial \mathrm{I}_{12}}{\partial u_1} - \dfrac{1}{2} \dfrac{\partial \mathrm{I}_{11}}{\partial u_2} \\
	\end{dcases}
\end{equation*}
\begin{equation*}
	\begin{dcases}
		\Gamma^1_{12}\mathrm{I}_{11} +\Gamma^2_{12} \mathrm{I}_{12} = \dfrac{\partial x}{\partial u_1} \cdot \dfrac{\partial^2 x}{\partial u_1 \partial u_2} = \dfrac{1}{2} \dfrac{\partial \mathrm{I}_{11}}{\partial u_2} \\
		\Gamma^1_{12}\mathrm{I}_{12} +\Gamma^2_{12} \mathrm{I}_{22} = \dfrac{\partial x}{\partial u_2} \cdot \dfrac{\partial^2 x}{\partial u_1 \partial u_2} = \dfrac{1}{2} \dfrac{\partial \mathrm{I}_{22}}{\partial u_1} \\
	\end{dcases}
\end{equation*}
\begin{equation*}
	\begin{dcases}
		\Gamma^1_{22}\mathrm{I}_{11} +\Gamma^2_{22} \mathrm{I}_{12} = \dfrac{\partial x}{\partial u_1} \cdot \dfrac{\partial^2 x}{\partial u_2^2} = \dfrac{\partial \mathrm{I}_{12}}{\partial u_2} - \dfrac{1}{2} \dfrac{\partial \mathrm{I}_{22}}{\partial u_1} \\
		\Gamma^1_{22}\mathrm{I}_{12} +\Gamma^2_{22} \mathrm{I}_{22} = \dfrac{\partial x}{\partial u_2} \cdot \dfrac{\partial^2 x}{\partial u_2^2} = \dfrac{1}{2} \dfrac{\partial \mathrm{I}_{22}}{\partial u_2} \\
	\end{dcases}
\end{equation*}
Each group of two linear equations have a unique solution since the first fundamental form is positive-definite. More importantly, the Christoffel symbols are fully determined by the first fundamental form hence invariant under isometry. 

The components of $a$ are (note that $a_{12} \neq a_{21}$ without extra conditions): 
\begin{equation} \label{eq: derivative of the Gauss map}
	\begin{gathered}
		\begin{bmatrix}
			\mathrm{I}_{11} & \mathrm{I}_{12} \\
			\mathrm{I}_{12} & \mathrm{I}_{22}
		\end{bmatrix} \begin{bmatrix}
			a_{11} & a_{12} \\
			a_{21} & a_{22}
		\end{bmatrix} = -\begin{bmatrix}
			\mathrm{II}_{11} & \mathrm{II}_{12} \\
			\mathrm{II}_{12} & \mathrm{II}_{22}
		\end{bmatrix} \\
		\Rightarrow ~~ \begin{bmatrix}
			a_{11} & a_{12} \\
			a_{21} & a_{22}
		\end{bmatrix} = - \begin{bmatrix}
			\mathrm{I}_{11} & \mathrm{I}_{12} \\
			\mathrm{I}_{12} & \mathrm{I}_{22}
		\end{bmatrix}^{-1} \begin{bmatrix}
			\mathrm{II}_{11} & \mathrm{II}_{12} \\
			\mathrm{II}_{12} & \mathrm{II}_{22}
		\end{bmatrix} \\
		= \dfrac{1}{\mathrm{I}_{11}\mathrm{I}_{22}-\mathrm{I}_{12}^2} \begin{bmatrix}
			-\mathrm{I}_{22} & \mathrm{I}_{12} \\
			\mathrm{I}_{12} & -\mathrm{I}_{11}
		\end{bmatrix} \begin{bmatrix}
			\mathrm{II}_{11} & \mathrm{II}_{12} \\
			\mathrm{II}_{12} & \mathrm{II}_{22}
		\end{bmatrix} 
	\end{gathered}
\end{equation}
and we have
\begin{equation} \label{eq: cross product of the Gauss map}
	\begin{aligned}
		\dfrac{\partial N}{\partial u_1} \times \dfrac{\partial N}{\partial u_2} & = (a_{11}a_{22} - a_{12}a_{21}) \dfrac{\partial x}{\partial u_1} \times \dfrac{\partial x}{\partial u_2} \\
		& = \dfrac{\mathrm{II}_{11}\mathrm{II}_{22}-\mathrm{II}_{12}^2}{\mathrm{I}_{11}\mathrm{I}_{22}-\mathrm{I}_{12}^2} \dfrac{\partial x}{\partial u_1} \times \dfrac{\partial x}{\partial u_2} \\
		& = \kappa_\mathrm{G} \dfrac{\partial x}{\partial u_1} \times \dfrac{\partial x}{\partial u_2}
	\end{aligned} 
\end{equation}
The components of $b$ are exactly the components of the second fundamental form. The \textit{third fundamental form}, which is the first fundamental form of the Gauss map, is defined as:
\begin{equation} \label{eq: third fundamental form}
	\begin{bmatrix}
		\mathrm{III}_{11} & \mathrm{III}_{12} \\
		\mathrm{III}_{21} & \mathrm{III}_{22}
	\end{bmatrix} = \begin{bmatrix}
		\dfrac{\partial N}{\partial u_1} \cdot \dfrac{\partial N}{\partial u_1} & \dfrac{\partial N}{\partial u_1} \cdot \dfrac{\partial N}{\partial u_2} \\[12pt]
		\dfrac{\partial N}{\partial u_2} \cdot \dfrac{\partial N}{\partial u_1} & \dfrac{\partial N}{\partial u_2} \cdot \dfrac{\partial N}{\partial u_2}
	\end{bmatrix}
\end{equation}
\begin{equation*}
	\begin{aligned}
		\mathrm{III}_{11} & = a_{11}^2 \mathrm{I}_{11} + 2a_{11}a_{21} \mathrm{I}_{12} + a_{21}^2 \mathrm{I}_{22} \\
		& = \begin{bmatrix}
			a_{11} & a_{21}
		\end{bmatrix} \begin{bmatrix}
			\mathrm{I}_{11} & \mathrm{I}_{12} \\
			\mathrm{I}_{12} & \mathrm{I}_{22}
		\end{bmatrix} \begin{bmatrix}
			a_{11} \\
			a_{21}
		\end{bmatrix}
	\end{aligned}
\end{equation*}
\begin{equation*}
	\begin{aligned}
		\mathrm{III}_{12} & = a_{11}a_{12} \mathrm{I}_{11} + (a_{11}a_{22} + a_{21}a_{12}) \mathrm{I}_{12} + a_{21}a_{22} \mathrm{I}_{22} \\
		& = \begin{bmatrix}
			a_{11} & a_{21}
		\end{bmatrix} \begin{bmatrix}
			\mathrm{I}_{11} & \mathrm{I}_{12} \\
			\mathrm{I}_{12} & \mathrm{I}_{22}
		\end{bmatrix} \begin{bmatrix}
			a_{12} \\
			a_{22}
		\end{bmatrix}
	\end{aligned}
\end{equation*}
\begin{equation*}
	\begin{aligned}
		\mathrm{III}_{22} & = a_{12}^2 \mathrm{I}_{11} + 2a_{12}a_{22} \mathrm{I}_{12} + a_{12}^2 \mathrm{I}_{22} \\
		& = \begin{bmatrix}
			a_{12} & a_{22}
		\end{bmatrix} \begin{bmatrix}
			\mathrm{I}_{11} & \mathrm{I}_{12} \\
			\mathrm{I}_{12} & \mathrm{I}_{22}
		\end{bmatrix} \begin{bmatrix}
			a_{12} \\
			a_{22}
		\end{bmatrix}
	\end{aligned}
\end{equation*}
We could see that
\begin{equation*}
	\begin{aligned}
		\begin{bmatrix}
			\mathrm{III}_{11} & \mathrm{III}_{12} \\
			\mathrm{III}_{12} & \mathrm{III}_{22}
		\end{bmatrix} & = \begin{bmatrix}
			a_{11} & a_{12} \\
			a_{21} & a_{22}
		\end{bmatrix}^\mathrm{T} \begin{bmatrix}
			\mathrm{I}_{11} & \mathrm{I}_{12} \\
			\mathrm{I}_{12} & \mathrm{I}_{22}
		\end{bmatrix} \begin{bmatrix}
			a_{11} & a_{12} \\
			a_{21} & a_{22}
		\end{bmatrix} \\
		& =  \dfrac{1}{\mathrm{I}_{11}\mathrm{I}_{22}-\mathrm{I}_{12}^2}  \begin{bmatrix}
			\mathrm{II}_{11} & \mathrm{II}_{12} \\
			\mathrm{II}_{12} & \mathrm{II}_{22}
		\end{bmatrix} \begin{bmatrix}
			\mathrm{I}_{22} & -\mathrm{I}_{12} \\
			-\mathrm{I}_{12} & \mathrm{I}_{11}
		\end{bmatrix} \begin{bmatrix}
			\mathrm{II}_{11} & \mathrm{II}_{12} \\
			\mathrm{II}_{12} & \mathrm{II}_{22}
		\end{bmatrix} \\
	\end{aligned}
\end{equation*}
\begin{equation*}
	\begin{gathered}
		\mathrm{III}_{11} = \dfrac{\mathrm{II}_{12}^2\mathrm{I}_{11}-2\mathrm{II}_{11}\mathrm{II}_{22}\mathrm{I}_{12}+\mathrm{II}_{11}^2\mathrm{I}_{22}}{\mathrm{I}_{11}\mathrm{I}_{22}-\mathrm{I}_{12}^2}\\
		\mathrm{III}_{12} = 
		\dfrac{\mathrm{II}_{12}\mathrm{II}_{22}\mathrm{I}_{11}-(\mathrm{II}_{11}\mathrm{II}_{22}+\mathrm{II}_{12}^2)\mathrm{I}_{12}+\mathrm{II}_{11}\mathrm{II}_{12}\mathrm{I}_{22}}{\mathrm{I}_{11}\mathrm{I}_{22}-\mathrm{I}_{12}^2}\\
		\mathrm{III}_{22} = 
		\dfrac{\mathrm{II}_{22}^2\mathrm{I}_{11}-2\mathrm{II}_{12}\mathrm{II}_{22}\mathrm{I}_{12}+\mathrm{II}_{12}^2\mathrm{I}_{22}}{\mathrm{I}_{11}\mathrm{I}_{22}-\mathrm{I}_{12}^2}
	\end{gathered}
\end{equation*}
\begin{equation*}
	\begin{bmatrix}
		\mathrm{III}_{11} & \mathrm{III}_{12} \\
		\mathrm{III}_{12} & \mathrm{III}_{22}
	\end{bmatrix} = -\dfrac{\mathrm{II}_{11}\mathrm{II}_{22}-\mathrm{II}_{12}^2}{\mathrm{I}_{11}\mathrm{I}_{22}-\mathrm{I}_{12}^2} 
	\begin{bmatrix}
		\mathrm{I}_{11} & \mathrm{I}_{12} \\
		\mathrm{I}_{12} & \mathrm{I}_{22}
	\end{bmatrix} + \dfrac{\mathrm{II}_{22}\mathrm{I}_{11}-2\mathrm{II}_{12}\mathrm{I}_{12}+\mathrm{II}_{11}\mathrm{I}_{22}}{\mathrm{I}_{11}\mathrm{I}_{22}-\mathrm{I}_{12}^2} 	\begin{bmatrix}
		\mathrm{II}_{11} & \mathrm{II}_{12} \\
		\mathrm{II}_{12} & \mathrm{II}_{22}
	\end{bmatrix}
\end{equation*}
In conclusion:
\begin{equation} \label{eq: fundamental forms relation}
	\begin{bmatrix}
		\mathrm{III}_{11} & \mathrm{III}_{12} \\
		\mathrm{III}_{12} & \mathrm{III}_{22}
	\end{bmatrix} = - \kappa_\mathrm{G} \begin{bmatrix}
		\mathrm{I}_{11} & \mathrm{I}_{12} \\
		\mathrm{I}_{12} & \mathrm{I}_{22}
	\end{bmatrix} + 2 \kappa_\mathrm{H} \begin{bmatrix}
		\mathrm{II}_{11} & \mathrm{II}_{12} \\
		\mathrm{II}_{12} & \mathrm{II}_{22}
	\end{bmatrix}
\end{equation}

The explicit expression of the Christoffel symbol is:
\begin{equation*}
	\begin{gathered}
		\Gamma_{11}^1 = \dfrac{\dfrac{\mathrm{I}_{22}}{2} \dfrac{\partial \mathrm{I}_{11}}{\partial u_1} - \mathrm{I}_{12} \left( \dfrac{\partial \mathrm{I}_{12}}{\partial u_1} - \dfrac{1}{2} \dfrac{\partial \mathrm{I}_{11}}{\partial u_2} \right)}{\mathrm{I}_{11}\mathrm{I}_{22} - \mathrm{I}_{12}^2} \\
		\Gamma_{11}^2 = \dfrac{\mathrm{I}_{11} \left( \dfrac{\partial \mathrm{I}_{12}}{\partial u_1} - \dfrac{1}{2} \dfrac{\partial \mathrm{I}_{11}}{\partial u_2} \right) - \dfrac{\mathrm{I}_{12}}{2} \dfrac{\partial \mathrm{I}_{11}}{\partial u_1} }{\mathrm{I}_{11}\mathrm{I}_{22} - \mathrm{I}_{12}^2}
	\end{gathered}
\end{equation*}
\begin{equation} \label{eq: Christoffel symbol explicit}
	\begin{gathered}
		\Gamma_{12}^1 = \dfrac{\dfrac{\mathrm{I}_{22}}{2} \dfrac{\partial \mathrm{I}_{11}}{\partial u_2} - \dfrac{\mathrm{I}_{12}}{2} \dfrac{\partial \mathrm{I}_{22}}{\partial u_1} }{\mathrm{I}_{11}\mathrm{I}_{22} - \mathrm{I}_{12}^2} \\
		\Gamma_{12}^2 = \dfrac{\dfrac{\mathrm{I}_{11}}{2} \dfrac{\partial \mathrm{I}_{22}}{\partial u_1} - \dfrac{\mathrm{I}_{12}}{2} \dfrac{\partial \mathrm{I}_{11}}{\partial u_2} }{\mathrm{I}_{11}\mathrm{I}_{22} - \mathrm{I}_{12}^2}
	\end{gathered}
\end{equation}
\begin{equation*}
	\begin{gathered}
		\Gamma_{22}^1 = \dfrac{\mathrm{I}_{22} \left(\dfrac{\partial \mathrm{I}_{12}}{\partial u_2} - \dfrac{1}{2} \dfrac{\partial \mathrm{I}_{22}}{\partial u_1}\right) - \dfrac{\mathrm{I}_{12}}{2} \dfrac{\partial \mathrm{I}_{22}}{\partial u_2}}{\mathrm{I}_{11}\mathrm{I}_{22} - \mathrm{I}_{12}^2} \\
		\Gamma_{22}^2 = \dfrac{\dfrac{\mathrm{I}_{11}}{2} \dfrac{\partial \mathrm{I}_{22}}{\partial u_2} - \mathrm{I}_{12} \left(\dfrac{\partial \mathrm{I}_{12}}{\partial u_2} - \dfrac{1}{2} \dfrac{\partial \mathrm{I}_{22}}{\partial u_1}\right) }{\mathrm{I}_{11}\mathrm{I}_{22} - \mathrm{I}_{12}^2}
	\end{gathered}
\end{equation*}

Notably the following \textit{compatibility condition} relates the first and second fundamental forms.
\begin{equation*}
	\begin{cases}
		\dfrac{\partial }{\partial u_2} \left(\dfrac{\partial^2 x}{\partial u_1^2}\right) = \dfrac{\partial }{\partial u_1} \left(\dfrac{\partial^2 x}{\partial u_1 \partial u_2}\right) \\[8pt]
		\dfrac{\partial }{\partial u_1} \left(\dfrac{\partial^2 x}{\partial u_2^2}\right) = \dfrac{\partial }{\partial u_2} \left(\dfrac{\partial^2 x}{\partial u_1 \partial u_2}\right) \\[8pt]
		\dfrac{\partial }{\partial u_2} \left(\dfrac{\partial N}{\partial u_1}\right) = \dfrac{\partial }{\partial u_1} \left(\dfrac{\partial N}{\partial u_2}\right) \\
	\end{cases}
\end{equation*}
by writing everything under the basis $(\partial x / \partial u_1, ~\partial x / \partial u_2, ~N)$ and comparing the coefficient, we could obtain 9 relations among the first and second fundamental forms. It turns out that only 3 of them are independent, called the \textit{compatibility equation of surfaces} or \textit{Gauss-Mainardi-Codazzi Equations}:
\begin{equation} \label{eq: surface compatibility}
	\begin{dcases}
		\dfrac{\partial \Gamma^2_{12}}{\partial u_1} - \dfrac{\partial \Gamma^2_{11}}{\partial u_2} + \Gamma^1_{12}\Gamma^2_{11} + \Gamma^2_{12}\Gamma^2_{12} - \Gamma^2_{11}\Gamma^2_{22} - \Gamma^1_{11}\Gamma^2_{12} = - \mathrm{I}_{11} \kappa_\mathrm{G} \\
		\dfrac{\partial \mathrm{II}_{11}}{\partial u_2} - \dfrac{\partial \mathrm{II}_{12}}{\partial u_1} = \mathrm{II}_{11}\Gamma^1_{12} + \mathrm{II}_{12}(\Gamma^2_{12} - \Gamma^1_{11}) - \mathrm{II}_{22}\Gamma^2_{11} \\
		\dfrac{\partial \mathrm{II}_{12}}{\partial u_2} - \dfrac{\partial \mathrm{II}_{22}}{\partial u_1} = \mathrm{II}_{11}\Gamma^1_{22} + \mathrm{II}_{12}(\Gamma^2_{22} - \Gamma^1_{12}) - \mathrm{II}_{22}\Gamma^2_{12} \\
	\end{dcases}
\end{equation}

\section{Coordinate net} \label{section: coordinate net}

Let $X$ be a parametrized surface with chart $\Phi: u \in I^2 \rightarrow x \in X \subset \mathbb{R}^3$. The \textit{coordinate curves} described by $u_1 = \mathrm{Const}$ and $u_2 = \mathrm{Const}$, forms a \textit{coordinate net} on $X$. The angle $\theta$ between coordinate curves, which can be calculated using
\begin{equation*}
	\cos \theta = \dfrac{\mathrm{I}_{12}}{\sqrt{\mathrm{I}_{11}\mathrm{I}_{22}}}
\end{equation*}
is called the \textbf{Chebyshev angle}. The study on coordinate nets are extremely useful for our interests since it provides a natural discretization to a quad-mesh.

Recall that we simplify the parametrization of a curve by using arc length. If we take a similar operation:
\begin{equation*}
	\begin{gathered}
		s_1 = \int_{0}^{u_1} \sqrt{\mathrm{I}_{11}(v_1, ~v_2)} \dif v_1 \\
		s_2 = \int_{0}^{u_2} \sqrt{\mathrm{I}_{22}(v_1, ~v_2)} \dif v_2 
	\end{gathered}
\end{equation*}
since
\begin{equation*}
	\begin{gathered}
		\dfrac{\partial x}{\partial s_1} = \dfrac{\partial x}{\partial u_1} \dfrac{\partial u_1}{\partial s_1} + \dfrac{\partial x}{\partial u_2} \dfrac{\partial u_2}{\partial s_1} \\
		\dfrac{\partial x}{\partial s_2} = \dfrac{\partial x}{\partial u_1} \dfrac{\partial u_1}{\partial s_2} + \dfrac{\partial x}{\partial u_2} \dfrac{\partial u_2}{\partial s_2} \\
	\end{gathered}
\end{equation*}
the arc length reparametrization does not make any simplification. However, observe that if
\begin{equation*}
	\dfrac{\partial \mathrm{I}_{11}}{\partial u_2} = \dfrac{\partial \mathrm{I}_{22}}{\partial u_1} = 0 
\end{equation*}
which means the lengths of the opposite side of `curved quadrilaterals' formed by the coordinate curves are equal. We can use the above arc length parametrization $s = (s_1, ~s_2)$, called a \textbf{Chebyshev net}, such that
\begin{equation*}
	\begin{bmatrix}
		\mathrm{I}_{11} & \mathrm{I}_{12} \\
		\mathrm{I}_{12} & \mathrm{I}_{22}
	\end{bmatrix} = \begin{bmatrix}
		1 & \cos \theta \\
		\cos \theta & 1
	\end{bmatrix}, ~~\theta \in (0, \pi) \textrm{~is~the~Chebyshev~angle~}
\end{equation*}
Note that the condition for a Chebyshev net is equivalent to
\begin{equation} \label{eq: partial differential Chebyshev}
	\begin{dcases}
		\dfrac{1}{2}\dfrac{\partial \mathrm{I}_{11}}{\partial u_2} = \dfrac{\partial^2 x}{\partial u_1 \partial u_2} \cdot \dfrac{\partial x}{\partial u_1} = 0 \\
		\dfrac{1}{2}\dfrac{\partial \mathrm{I}_{22}}{\partial u_1} = \dfrac{\partial^2 x}{\partial u_1 \partial u_2} \cdot \dfrac{\partial x}{\partial u_2} = 0 \\
	\end{dcases} ~~\Leftrightarrow~~ 
	\begin{dcases}
		\dfrac{\partial^2 x}{\partial u_1 \partial u_2} = \lambda(u_1, ~u_2) \dfrac{\partial x}{\partial u_1} \times \dfrac{\partial x}{\partial u_2} \\ \lambda: I^2 \rightarrow \mathbb{R} 
	\end{dcases}
\end{equation}
which is the differential equation for a Chebyshev net. The ratio $\lambda(u), ~u = (u_1, ~u_2) \in I^2$ is:
\begin{equation} \label{eq: Chebyshev net ratio}
	\lambda(u) = \dfrac{\dfrac{\partial^2 x}{\partial u_1 \partial u_2} \cdot N}{\left \| \dfrac{\partial x}{\partial u_1} \times \dfrac{\partial x}{\partial u_2} \right \|} = \dfrac{\mathrm{I}_{12}}{\sqrt{\mathrm{I}_{11}\mathrm{I}_{22} - \mathrm{I}_{12}^2}}
\end{equation}

The Christoffel symbols and Gaussian curvature from direct calculation over the first fundamental form is:
\begin{equation*}
	\begin{gathered}
		\Gamma_{11}^1 = \dfrac{1}{\tan \theta} \dfrac{\partial \theta}{\partial s_1}, ~~\Gamma_{11}^2 = -\dfrac{1}{\sin \theta} \dfrac{\partial \theta}{\partial s_1} \\
		\Gamma_{12}^1 = \Gamma_{12}^2 = 0 \\
		\Gamma_{22}^1 = -\dfrac{1}{\sin \theta} \dfrac{\partial \theta}{\partial s_2}, ~~\Gamma_{22}^2 = \dfrac{1}{\tan \theta} \dfrac{\partial \theta}{\partial s_2}
	\end{gathered}
\end{equation*}
\begin{equation*}
	\kappa_\mathrm{G} = - \dfrac{1}{\sin \theta} \dfrac{\partial^2 \theta}{\partial s_1 \partial s_2}
\end{equation*}
In particular, an \textbf{orthogonal Chebyshev net} where $\mathrm{I}_{12} = 0$ everywhere infers that $\theta = \pi/2$ identically. The Gaussian curvature $\kappa_\mathrm{G} = 0$ from the above calculation.

Recall that an asymptotic curve on a surface has everywhere zero normal curvature. We say a parametrization forms an \textbf{asymptotic net} if both coordinate curves are asymptotic curves, which means:
\begin{equation} \label{eq: partial differential asymptotic}
	\begin{dcases}
		\begin{bmatrix}
			1 & 0
		\end{bmatrix} \begin{bmatrix}
			\mathrm{II}_{11} & \mathrm{II}_{12} \\
			\mathrm{II}_{12} & \mathrm{II}_{22}
		\end{bmatrix} \begin{bmatrix}
			1 \\
			0 
		\end{bmatrix} = 0 \\ \begin{bmatrix}
			0 & 1
		\end{bmatrix} \begin{bmatrix}
			\mathrm{II}_{11} & \mathrm{II}_{12} \\
			\mathrm{II}_{12} & \mathrm{II}_{22}
		\end{bmatrix} \begin{bmatrix}
			0 \\
			1 
		\end{bmatrix} = 0
	\end{dcases} ~~ \Leftrightarrow ~~ \mathrm{II}_{11} = 0, ~ \mathrm{II}_{22} = 0
\end{equation}
The derivative of the Gauss map of an asymptotic net can be derived from \eqref{eq: derivative of the Gauss map}:
\begin{equation} \label{eq: derivative of an asymptotic net}
	\begin{aligned}
		\begin{bmatrix}
			\dfrac{\partial N}{\partial u_1} &
			\dfrac{\partial N}{\partial u_2}
		\end{bmatrix}
		& = \dfrac{1}{\mathrm{I}_{11}\mathrm{I}_{22}-\mathrm{I}_{12}^2} \begin{bmatrix}
			\dfrac{\partial x}{\partial u_1} &
			\dfrac{\partial x}{\partial u_2}
		\end{bmatrix} \begin{bmatrix}
			-\mathrm{I}_{22} & \mathrm{I}_{12} \\
			\mathrm{I}_{12} & -\mathrm{I}_{11}
		\end{bmatrix} \begin{bmatrix}
			\mathrm{II}_{11} & \mathrm{II}_{12} \\
			\mathrm{II}_{12} & \mathrm{II}_{22}
		\end{bmatrix} \\
		& = \dfrac{\mathrm{II}_{12}}{\mathrm{I}_{11}\mathrm{I}_{22}-\mathrm{I}_{12}^2} \begin{bmatrix}
			\dfrac{\partial x}{\partial u_1} &
			\dfrac{\partial x}{\partial u_2}
		\end{bmatrix} \begin{bmatrix}
			\mathrm{I}_{12} & -\mathrm{I}_{22} \\
			-\mathrm{I}_{11} & \mathrm{I}_{12}
		\end{bmatrix} 
	\end{aligned}
\end{equation}
Since
\begin{equation*}
	N = \dfrac{\dfrac{\partial x}{\partial u_1} \times \dfrac{\partial x}{\partial u_2}}{\left\| \dfrac{\partial x}{\partial u_1} \times \dfrac{\partial x}{\partial u_2}\right\|} = \dfrac{\dfrac{\partial x}{\partial u_1} \times \dfrac{\partial x}{\partial u_2}}{ \sqrt{\mathrm{I}_{11}\mathrm{I}_{22}-\mathrm{I}_{12}^2}}
\end{equation*}
\begin{equation}
	\begin{gathered}
		N \times \dfrac{\partial x}{\partial u_1} = \dfrac{1}{\sqrt{\mathrm{I}_{11}\mathrm{I}_{22}-\mathrm{I}_{12}^2}} \left( \mathrm{I}_{11} \dfrac{\partial x}{\partial u_2} - \mathrm{I}_{12} \dfrac{\partial x}{\partial u_1} \right) \\
		\dfrac{\partial x}{\partial u_2} \times N = \dfrac{1}{\sqrt{\mathrm{I}_{11}\mathrm{I}_{22}-\mathrm{I}_{12}^2}} \left(\mathrm{I}_{22} \dfrac{\partial x}{\partial u_1} - \mathrm{I}_{12} \dfrac{\partial x}{\partial u_2} \right)
	\end{gathered}
\end{equation}
we have (note that $\kappa_\mathrm{G} < 0$ since $\mathrm{II}_{11} = \mathrm{II}_{22} = 0$)
\begin{equation} \label{eq: normal field asymptotic}
	\begin{gathered}
		N \times \dfrac{\partial N}{\partial u_1} = \dfrac{ \mathrm{II}_{12}}{\sqrt{\mathrm{I}_{11}\mathrm{I}_{22}-\mathrm{I}_{12}^2}} \dfrac{\partial x}{\partial u_1} = (-\kappa_\mathrm{G})^{1/2} \dfrac{\partial x}{\partial u_1} \\
		\dfrac{\partial N}{\partial u_2} \times N= \dfrac{ \mathrm{II}_{12}}{\sqrt{\mathrm{I}_{11}\mathrm{I}_{22}-\mathrm{I}_{12}^2}} \dfrac{\partial x}{\partial u_2} = (-\kappa_\mathrm{G})^{1/2} \dfrac{\partial x}{\partial u_2}
	\end{gathered}
\end{equation}
The \textbf{Lelieuvre normal field} $N^\mathrm{L}$ is defined as \citep{blaschke_vorlesungen_1923}:
\begin{equation}
	N^\mathrm{L} = N (-\kappa_\mathrm{G})^{-1/4}
\end{equation}
which satisfies:
\begin{equation} \label{eq: normal field Lelieuvre}
	\begin{gathered}
		N^\mathrm{L} \times \dfrac{\partial N^\mathrm{L}}{\partial u_1} =   \dfrac{\partial x}{\partial u_1} \\
		\dfrac{\partial N^\mathrm{L}}{\partial u_2} \times N^\mathrm{L} = \dfrac{\partial x}{\partial u_2}
	\end{gathered}
\end{equation}
Furthermore,
\begin{equation} \label{eq: derivation of Lelieuvre}
	\begin{gathered}
		\dfrac{1}{2} \left(\dfrac{\partial}{\partial u_1} \left(\dfrac{\partial N^\mathrm{L}}{\partial u_2} \times N^\mathrm{L} \right) - \dfrac{\partial}{\partial u_2} \left(N^\mathrm{L} \times \dfrac{\partial N^\mathrm{L}}{\partial u_1} \right)\right) =  \dfrac{\partial^2 N^\mathrm{L}}{\partial u_1 \partial u_2} \times N^\mathrm{L} = 0 \\
		\dfrac{1}{2} \left( \dfrac{\partial}{\partial u_1} \left(\dfrac{\partial N^\mathrm{L}}{\partial u_2} \times N^\mathrm{L} \right) + \dfrac{\partial}{\partial u_2} \left(N^\mathrm{L} \times \dfrac{\partial N^\mathrm{L}}{\partial u_1} \right) \right) =  \dfrac{\partial N^\mathrm{L}}{\partial u_2} \times \dfrac{\partial N^\mathrm{L}}{\partial u_1} = \dfrac{\partial^2 x}{\partial u_1 \partial u_2} \\
	\end{gathered}
\end{equation}
We say $N^\mathrm{L}$ is \textbf{Lorentz-harmonic} and forms a \textbf{Moutard net} if:
\begin{equation} \label{eq: Moutard}
	\dfrac{\partial^2 N^\mathrm{L}}{\partial u_1 \partial u_2} = \lambda N^\mathrm{L}, ~~\lambda: I^2 \rightarrow \mathbb{R} 
\end{equation}
Given a Moutard net $N^\mathrm{L}$, from \eqref{eq: normal field Lelieuvre} and the integration of $\partial x / \partial s_1$ and $\partial x / \partial s_2$ we could obtain a unique surface $x$, up to a translation. 

Now we consider an \textbf{asymptotic Chebyshev net}, as known as a \textbf{K-surface} in previous literatures. From the compatibility equation of surfaces, \eqref{eq: surface compatibility}, an asymptotic Chebyshev net has the second fundamental form below:
\begin{equation*}
	\begin{bmatrix}
		\mathrm{II}_{11} & \mathrm{II}_{12} \\
		\mathrm{II}_{12} & \mathrm{II}_{22}
	\end{bmatrix} = \begin{bmatrix}
		0 & \sin \theta \\
		\sin \theta & 0
	\end{bmatrix}
\end{equation*} 
and we could obtain
\begin{equation}
	\begin{gathered}
		\kappa_\mathrm{G} = -1 \\
		\dfrac{\partial^2 \theta}{\partial s_1 \partial s_2} = \sin \theta
	\end{gathered}
\end{equation}
To conclude, only a pseudosphere admits an asymptotic Chebyshev net. The latter is the famous \textbf{sine-Gordon equation}. Immediately from \eqref{eq: normal field asymptotic}:
\begin{equation} \label{eq: Lelieuvre normal field}
	\begin{dcases}
		\dfrac{\partial x}{\partial s_1} = N \times \dfrac{\partial N}{\partial s_1} \\
		\dfrac{\partial x}{\partial s_2} = \dfrac{\partial N}{\partial s_2} \times N \\
	\end{dcases}
\end{equation}
then from \eqref{eq: derivation of Lelieuvre} and \eqref{eq: cross product of the Gauss map}:
\begin{equation}
	\dfrac{\partial^2 x}{\partial s_1 \partial s_2} = -\dfrac{\partial N}{\partial s_1} \times \dfrac{\partial N}{\partial s_2} = -\kappa_\mathrm{G} \dfrac{\partial x}{\partial s_1} \times \dfrac{\partial x}{\partial s_2} = -\kappa_\mathrm{G}\sqrt{\mathrm{I}_{11}\mathrm{I}_{22}-\mathrm{I}_{12}^2} N = N \sin \theta
\end{equation}
which means
\begin{equation*}
	\dfrac{\partial^2 x}{\partial s_1 \partial s_2} \times N = 0 ~~\Leftrightarrow~~ \begin{dcases}
		\dfrac{\partial^2 x}{\partial s_1 \partial s_2} \cdot \dfrac{\partial x}{\partial s_1} = \dfrac{1}{2}\dfrac{\partial \mathrm{I}_{11}}{\partial s_2} = 0 \\
		\dfrac{\partial^2 x}{\partial s_1 \partial s_2} \cdot \dfrac{\partial x}{\partial s_2} = \dfrac{1}{2} \dfrac{\partial \mathrm{I}_{22}}{\partial s_1} = 0
	\end{dcases}
\end{equation*}

Now continue from \eqref{eq: derivative of the Gauss map}:
\begin{equation} \label{eq: derivative of an asymptotic Chebyshev net}
	\begin{aligned}
		\begin{bmatrix}
			\dfrac{\partial N}{\partial s_1} &
			\dfrac{\partial N}{\partial s_2}
		\end{bmatrix}
		& = \dfrac{\mathrm{II}_{12}}{\mathrm{I}_{11}\mathrm{I}_{22}-\mathrm{I}_{12}^2} \begin{bmatrix}
			\dfrac{\partial x}{\partial s_1} &
			\dfrac{\partial x}{\partial s_2}
		\end{bmatrix} \begin{bmatrix}
			\mathrm{I}_{12} & -\mathrm{I}_{22} \\
			-\mathrm{I}_{11} & \mathrm{I}_{12}
		\end{bmatrix} \\
		& = \begin{bmatrix}
			\dfrac{\partial x}{\partial s_1} &
			\dfrac{\partial x}{\partial s_2}
		\end{bmatrix} \begin{bmatrix}
			\dfrac{1}{\tan \theta} & -\dfrac{1}{\sin \theta} \\[10pt]
			-\dfrac{1}{\sin \theta} & \dfrac{1}{\tan \theta}
		\end{bmatrix} 
	\end{aligned}
\end{equation}
From direct calculation we could see that:
\begin{equation*}
	\begin{aligned}
		\dfrac{\partial^2 N}{\partial s_1 \partial s_2} = & - \left( 1+ \dfrac{1}{\tan^2 \theta} \right) \dfrac{\partial \theta}{\partial s_2} \dfrac{\partial x}{\partial s_1} + \dfrac{1}{\tan \theta} \dfrac{\partial^2 x}{\partial s_1 \partial s_2} + \dfrac{1}{\sin \theta \tan \theta} \dfrac{\partial \theta}{\partial s_2} \dfrac{\partial x}{\partial s_2} - \dfrac{1}{\sin \theta} \dfrac{\partial^2 x}{\partial s_2^2} \\
		= & - \left( 1+ \dfrac{1}{\tan^2 \theta} \right) \dfrac{\partial \theta}{\partial s_1} \dfrac{\partial x}{\partial s_2} + \dfrac{1}{\tan \theta} \dfrac{\partial^2 x}{\partial s_1 \partial s_2} + \dfrac{1}{\sin \theta \tan \theta} \dfrac{\partial \theta}{\partial s_1} \dfrac{\partial x}{\partial s_1} - \dfrac{1}{\sin \theta} \dfrac{\partial^2 x}{\partial s_1^2}
	\end{aligned}
\end{equation*}
Since
\begin{equation*}
	\dfrac{\partial^2 x}{\partial s_1 \partial s_2} \cdot \dfrac{\partial x}{\partial s_1} = \dfrac{\partial \mathrm{I}_{11}}{\partial s_2} = 0, ~~\dfrac{\partial^2 x}{\partial s_1 \partial s_2} \cdot \dfrac{\partial x}{\partial s_2} = \dfrac{\partial \mathrm{I}_{22}}{\partial s_1} = 0
\end{equation*}
\begin{equation*}
	\dfrac{\partial^2 x}{\partial s_1^2} \cdot \dfrac{\partial x}{\partial s_1} = \dfrac{\partial \mathrm{I}_{11}}{\partial s_1} = 0, ~~\dfrac{\partial^2 x}{\partial s_2^2} \cdot \dfrac{\partial x}{\partial s_2} = \dfrac{\partial \mathrm{I}_{22}}{\partial s_2} = 0
\end{equation*}
\begin{equation*}
	\dfrac{\partial^2 x}{\partial s_1^2} \cdot \dfrac{\partial x}{\partial s_2} = \dfrac{\partial \mathrm{I}_{12}}{\partial s_1} - \dfrac{\partial^2 x}{\partial s_1 \partial s_2} \cdot \dfrac{\partial x}{\partial s_1} = -\sin \theta \dfrac{\partial \theta}{\partial s_1}
\end{equation*}
\begin{equation*}
	\dfrac{\partial^2 x}{\partial s_2^2} \cdot \dfrac{\partial x}{\partial s_1} = \dfrac{\partial \mathrm{I}_{12}}{\partial s_2} - \dfrac{\partial^2 x}{\partial s_1 \partial s_2} \cdot \dfrac{\partial x}{\partial s_2} = -\sin \theta \dfrac{\partial \theta}{\partial s_2}
\end{equation*}
For both expressions of $\partial^2 N / \partial s_1 \partial s_2$, from dot production over $\partial^2 x / \partial s_1$ and $\partial^2 x / \partial s_2$:
\begin{equation*}
	\dfrac{\partial^2 N}{\partial s_1 \partial s_2} \cdot \dfrac{\partial x}{\partial s_1} = 0
\end{equation*}
\begin{equation*}
	\dfrac{\partial^2 N}{\partial s_1 \partial s_2} \cdot \dfrac{\partial x}{\partial s_2} = 0
\end{equation*}
We conclude that the Moutard equation for $N$ is:
\begin{equation} \label{eq: partial differential aymptopic Chebyshev}
	\dfrac{\partial^2 N}{\partial s_1 \partial s_2} = N \cos \theta
\end{equation} 
Furthermore we could see that:
\begin{equation*}
	\begin{gathered}
		\dfrac{\partial^2 N}{\partial s_1 \partial s_2} \cdot \dfrac{\partial N}{\partial s_1} = 0 ~~\Rightarrow~~ \dfrac{\partial}{\partial s_2} \left(\dfrac{\partial N}{\partial s_1} \cdot \dfrac{\partial N}{\partial s_1} \right) = 0\\
		\dfrac{\partial^2 N}{\partial s_1 \partial s_2} \cdot \dfrac{\partial N}{\partial s_2} = 0 ~~\Rightarrow~~ \dfrac{\partial}{\partial s_1} \left(\dfrac{\partial N}{\partial s_2} \cdot \dfrac{\partial N}{\partial s_2} \right) = 0
	\end{gathered}
\end{equation*}
The above derivation leads to the following proposition:
\begin{prop}
	The Gauss map of a K-surface is a Chebyshev net. A K-surface is the only asymptotic net with a Chebyshev Gauss map.
\end{prop}

The counterpart of an asymptotic net is a \textbf{geodesic net}, where the coordinate curves have everywhere zero geodesic curvature. Note that even though the velocity along coordinate curves are constant, it will change along the other direction hence there is no arc length reparametrization similar to the Chebyshev net. The condition for a parametrization $u$ to form a geodesic net is the geodesic curvature is everywhere zero for each coordinate curve. From \eqref{eq: geodesic curvature}
\begin{equation*}
	\begin{dcases}
		\dfrac{\partial^2 x}{\partial u_1^2} \cdot \left(N \times \dfrac{\partial x}{\partial u_1}\right) = 0 \\
		\dfrac{\partial^2 x}{\partial u_2^2} \cdot \left(N \times \dfrac{\partial x}{\partial u_2}\right) = 0 \\
	\end{dcases}
\end{equation*}
It says $\dif^2 x / \dif u_1^2$ can be linearly represented by $N$ and $\dif x / \dif u_1$; $\dif^2 x / \dif u_2^2$ can be linearly represented by $N$ and $\dif x / \dif u_2$. From \eqref{eq: Christoffel symbol}, this condition is equivalent to certain Christoffel symbols are zero:
\begin{equation*}
	\Gamma_{11}^2 = \Gamma_{22}^1 = 0
\end{equation*}
and we can write the condition above in terms of the components of the first fundamental form from \eqref{eq: Christoffel symbol explicit}:
\begin{equation} \label{eq: geodesic net}
	\begin{dcases}
		2 \mathrm{I}_{11}  \dfrac{\partial \mathrm{I}_{12}}{\partial u_1} = \mathrm{I}_{11} \dfrac{\partial \mathrm{I}_{11}}{\partial u_2} + \mathrm{I}_{12} \dfrac{\partial \mathrm{I}_{11}}{\partial u_1} \\
		2\mathrm{I}_{22} \dfrac{\partial \mathrm{I}_{12}}{\partial u_2} = \mathrm{I}_{22} \dfrac{\partial \mathrm{I}_{22}}{\partial u_1} + \mathrm{I}_{12} \dfrac{\partial \mathrm{I}_{22}}{\partial u_2} \\
	\end{dcases}
\end{equation}
\eqref{eq: geodesic net} is the condition for a chart to form a geodesic net. Furthermore, let $\mathrm{I}_{12} = 0$, we obtain $\partial \mathrm{I}_{11} / \partial u_2 = \partial \mathrm{I}_{22} / \partial u_1 = 0$. Therefore an \textbf{orthogonal geodesic net} is equivalent to an orthogonal Chebyshev net. The first fundamental form is an identity matrix and $\kappa_\mathrm{G} = 0$.

Two tangent vectors 
\begin{equation*}
	\begin{gathered}
		\dfrac{\dif x}{\dif u}	\begin{bmatrix}
			a_1 \\ a_2
		\end{bmatrix}, ~~ \dfrac{\dif x}{\dif u} \begin{bmatrix}
			b_1 \\
			b_2 
		\end{bmatrix}, ~~a_1, ~a_2 \in \mathbb{R},~~b_1, ~b_2 \in \mathbb{R}
	\end{gathered}
\end{equation*}
are \textit{conjugate} if: 
\begin{equation*}
	\begin{bmatrix}
		a_1 & a_2
	\end{bmatrix} \begin{bmatrix}
		\mathrm{II}_{11} & \mathrm{II}_{12} \\
		\mathrm{II}_{12} & \mathrm{II}_{22}
	\end{bmatrix} \begin{bmatrix}
		b_1 \\
		b_2 
	\end{bmatrix} = 0
\end{equation*} 
Principal directions are conjugate. An asymptotic direction is conjugate to itself. Coordinate curves of parametrization $u$ forms a \textbf{conjugate net} if 
\begin{equation} \label{eq: partial differential conjugate}
	\begin{bmatrix}
		1 & 0
	\end{bmatrix} \begin{bmatrix}
		\mathrm{II}_{11} & \mathrm{II}_{12} \\
		\mathrm{II}_{12} & \mathrm{II}_{22}
	\end{bmatrix} \begin{bmatrix}
		0 \\
		1 
	\end{bmatrix} = 0 ~~ \Leftrightarrow ~~ \mathrm{II}_{12} = 0
\end{equation}
A special case of a conjugate net is the \textbf{curvature line net}, where the first and second fundamental form are simultaneously diagonalized. Clearly the condition is $\mathrm{I}_{12} = \mathrm{II}_{12} = 0$. A curvature line net is also called an \textbf{orthogonal conjugate net}. From \eqref{eq: derivative of the Gauss map}, the derivative of the Gauss map of a curvature line net is:
\begin{equation} \label{eq: derivative curvature line}
	\begin{aligned}
		\begin{bmatrix}
			\dfrac{\partial N}{\partial u_1} &
			\dfrac{\partial N}{\partial u_2}
		\end{bmatrix}
		& = \dfrac{1}{\mathrm{I}_{11}\mathrm{I}_{22}-\mathrm{I}_{12}^2} \begin{bmatrix}
			\dfrac{\partial x}{\partial u_1} &
			\dfrac{\partial x}{\partial u_2}
		\end{bmatrix} \begin{bmatrix}
			-\mathrm{I}_{22} & \mathrm{I}_{12} \\
			\mathrm{I}_{12} & -\mathrm{I}_{11}
		\end{bmatrix} \begin{bmatrix}
			\mathrm{II}_{11} & \mathrm{II}_{12} \\
			\mathrm{II}_{12} & \mathrm{II}_{22}
		\end{bmatrix}  \\ 
		& =  \begin{bmatrix}
			- \dfrac{\mathrm{II}_{11}}{\mathrm{I}_{11}}\dfrac{\partial x}{\partial u_1} &
			- \dfrac{\mathrm{II}_{22}}{\mathrm{I}_{22}} \dfrac{\partial x}{\partial u_2}
		\end{bmatrix}
	\end{aligned}
\end{equation}

\section{Initial condition for coordinate nets} \label{section: differential equation}

The various smooth coordinate nets introduced in Section \ref{section: coordinate net} are solutions of parametric partial differential systems. When solving a system of parametric partial differential equations, we say this problem is \textit{well-posed} if a given initial condition leads to a unique solution, which smoothly relies on the initial value and parameter. The well-posedness is crucial since in practice the input data can only be measured up to certain level of accuracy.

A hyperbolic first-order system for $x(t)$ is in the form of 
\begin{equation*}
	\dfrac{\dif x}{\dif t} = f(x; ~b) ~~\Leftrightarrow~~ \dfrac{\partial x_i}{\partial t_j} = f_{ij}(x; ~b)
\end{equation*}
and is well-posed \cite[Chapter 5]{bobenko_discrete_2008}. Here $t \in I^m$; $x \in \mathbb{R}^n ~(m, ~n \in \mathbb{Z}_+)$; $f \in \mathbb{R}^{n \times m}$ is a matrix of smooth functions, $b \in \mathbb{R}^p ~(p \in \mathbb{Z}_+)$ are the $p$ parameters for the system. We further require $f$ and all the partial derivatives of $f$ are bounded and possess a global Lipschitz constant. Consequently no blow-ups (value goes to infinity) are possible and hence the well-posedness can be continued to the boundary of $I^m$.

If there are higher order partial derivatives, we could try transferring the system to first-order by adding the number of variables. For example, $\partial^2 x / \partial t_1 \partial t_2 = x$, we could set $y(t) = \partial x / \partial t_1$ and $z(t) = \partial x / \partial t_2$, now $(x, ~y, ~z)$ forms an equivalent first-order system with the compatibility condition $\partial y / \partial t_2 = \partial z / \partial t_1$. 

Index $i ~(i \in \mathbb{Z}_+, ~i \le m)$ is called an \textit{evolution direction} of $x_j~(j \in \mathbb{Z}_+, ~j \le n)$ if $f_{ij} \neq 0$, otherwise the index is called a \textit{stationary direction}. The set of indices for evolution directions is denoted by $I_j$. We refer to $P_j = \{ t_i = 0 ~|~ i \in I_j \}$ as the \textit{coordinate hyperplane} for $I_j$. In our problem setting, the initial value for the system is a smooth function given on:
\begin{equation*}
	x^\mathrm{i} = \{x_j(P_j) \mathrm{~for~all~} j\}
\end{equation*}
In other words, for the $j$-th component of $x$, the initial value includes its value on the coordinate hyperplane over the stationary directions, and we only consider this form of initial value. In the example $\partial^2 x / \partial t_1 \partial t_2 = x$, the initial values are $x(0, ~0), ~y(t_1, ~0), ~z(0, ~t_2)$. 

Specifically for a first-order system, the well-posedness means that: 1) there exists a smooth solution $x(t)$ for initial value $x^\mathrm{i}$ and parameter $b$; 2) the above solution is unique; 3) for a initial value $x^\mathrm{i}$, there exists a neighbourhood $O(x^\mathrm{i})$ such that the family of solution $x(t,~ x^\mathrm{i}; ~b)$ is smooth over $O(x^\mathrm{i})$; 4) for a parameter $b$, there exists a neighbourhood $O(b)$ such that the family of solution $x(t,~ x^\mathrm{i}; ~b)$ is smooth over $O(b)$.

Many of the coordinate nets introduced in Section \ref{section: coordinate net} are hyperbolic first-order linear systems for $x(u_1, ~u_2)$ with constant coefficients, by setting $\partial x / \partial u_1 = y, ~\partial x /\partial u_2 = z$. The initial conditions below are partially mentioned in \cite{bobenko_discrete_2008}. 

\subsection*{Chebyshev net and orthogonal Chebyshev net}

From \eqref{eq: partial differential Chebyshev}, the system for a Chebyshev net is:
\begin{equation*}
	\begin{gathered}
		\dfrac{\partial^2 x}{\partial u_1 \partial u_2} = \lambda(u_1, ~u_2) \dfrac{\partial x}{\partial u_1} \times \dfrac{\partial x}{\partial u_2} ~~\Rightarrow~~ \dfrac{\partial y}{\partial u_2} = \lambda y \times z, ~~\dfrac{\partial z }{ \partial u_1} = \lambda y \times z
	\end{gathered}
\end{equation*}
The initial condition for a Chebyshev net is:
\begin{description}
	\item[Initial value] $x(0, ~0)$, $\dfrac{\partial x}{\partial u_1}(u_1, ~0)$, $\dfrac{\partial x}{\partial u_2}(0, ~u_2)$
	\item[Parameter] $\lambda$ for all $u_1, ~u_2$
\end{description}
From integration along the coordinate curves, the above initial value is equivalent to:
\begin{equation*}
	x(0, ~0), ~\dfrac{\partial x}{\partial u_1}(u_1, ~0), ~\dfrac{\partial x}{\partial u_2}(0, ~u_2)  ~~\Leftrightarrow~~ x(0, ~0), ~x(u_1, ~0), ~x(0, ~u_2)
\end{equation*}

\subsection*{Asymptotic net}

From \eqref{eq: partial differential asymptotic}, the system for an asymptotic net is:
\begin{equation*} 
	\begin{aligned}
		\dfrac{\partial^2 x}{\partial u_1^2} \cdot N = 0 & ~~\Rightarrow~~ \dfrac{\partial^2 x}{\partial u_1^2} = \Gamma^1_{11} \dfrac{\partial x}{\partial u_1} + \Gamma^2_{11} \dfrac{\partial x}{\partial u_2} \\
		\dfrac{\partial^2 x}{\partial u_2^2} \cdot N = 0 & ~~\Rightarrow~~ \dfrac{\partial^2 x}{\partial u_2^2} = \Gamma^1_{22} \dfrac{\partial x}{\partial u_1} + \Gamma^2_{22} \dfrac{\partial x}{\partial u_2}
	\end{aligned}
\end{equation*}
hence:
\begin{equation*}
	\begin{gathered}
		\dfrac{\partial y}{\partial u_1} = \Gamma^1_{11} y + \Gamma^2_{11} z \\
		\dfrac{\partial z}{\partial u_2} = \Gamma^1_{22} y + \Gamma^2_{22} z
	\end{gathered}
\end{equation*}
The initial value for an asymptotic net is supposed to be $x(0, ~0)$, $y(0, ~u_2)$, $z(u_1, ~0)$. Additionally, the initial value for an asymptotic net should meet the compatibility constraint. Since $y(0, ~u_2)$ and $z(u_1, ~0)$ cannot be sorely obtained from differentiating along the coordinate curves $x(u_1, ~0)$ and $x(0, ~u_2)$, we choose to proceed with the Lelieuvre normal field $N^\mathrm{L}$ alternatively. From the derivation for an asymptotic net in Section \ref{section: coordinate net} and the Moutard Equation \eqref{eq: Moutard} $\partial^2 N^\mathrm{L} / \partial u_1 \partial u_2 = \lambda N^\mathrm{L}$, we can solve $N^\mathrm{L}$ first, then calculate $\partial x/\partial u_1$ and $\partial x/\partial u_2$ to obtain surface $x$ by integration. In conclusion, the initial condition for an asymptotic net is:
\begin{description}
	\item[Initial value] $N^\mathrm{L}(0, ~0)$, $\dfrac{\partial N^\mathrm{L}}{\partial u_1}(u_1, ~0)$, $\dfrac{\partial N^\mathrm{L}}{\partial u_2}(0, ~u_2)$
	\item[Parameter] $\lambda$ for all $u_1, ~u_2$ 
\end{description}
From integration along the coordinate curves, the above initial value is equivalent to:
\begin{equation*}
	\begin{gathered}
		N^\mathrm{L}(0, ~0), ~\dfrac{\partial N^\mathrm{L}}{\partial u_1}(u_1, ~0), ~\dfrac{\partial N^\mathrm{L}}{\partial u_2}(0, ~u_2) ~~\Leftrightarrow~~ N^\mathrm{L}(0, ~0), ~N^\mathrm{L}(u_1, ~0),~ N^\mathrm{L}(0, ~u_2)
	\end{gathered}
\end{equation*}
The above condition is also the initial condition for a Moutard net.

\subsection*{Asymptotic Chebyshev net}

From \eqref{eq: partial differential aymptopic Chebyshev}, the system for an asymptotic Chebyshev net is:
\begin{equation*}
	\dfrac{\partial^2 N}{\partial s_1 \partial s_2} = N \dfrac{\partial N}{\partial s_1} \cdot \dfrac{\partial N}{\partial s_2} 
\end{equation*}
and the initial condition for an asymptotic Chebyshev net is:
\begin{description}
	\item[Initial value] $N(0, ~0)$, $\dfrac{\partial N}{\partial s_1}(s_1, ~0)$, $\dfrac{\partial N}{\partial s_2}(0, ~s_2)$
\end{description}
From integration along the coordinate curves, the above initial value is equivalent to:
\begin{equation*}
	\begin{gathered}
		N(0, ~0), ~\dfrac{\partial N}{\partial s_1}(s_1, ~0), ~\dfrac{\partial N}{\partial s_2}(0, ~s_2) ~~\Leftrightarrow~~ N(0, ~0), ~N(s_1, ~0),~ N(0, ~s_2)
	\end{gathered}
\end{equation*}

\subsection*{Geodesic net}

It could be examined that for a geodesic net, \eqref{eq: geodesic net} is not in the form of a first-order system. We will introduce the condition to determine a discrete geodesic net in Section \ref{section: difference equation}.

\subsection*{Conjugate net}

From \eqref{eq: partial differential conjugate}, the system for a conjugate net is:
\begin{equation*} 
	\begin{aligned}
		\dfrac{\partial^2 x}{\partial u_1 \partial u_2} \cdot N = 0 & ~~\Rightarrow~~ \dfrac{\partial^2 x}{\partial u_1 \partial u_2} = \Gamma^1_{12} \dfrac{\partial x}{\partial u_1} + \Gamma^2_{12}   \dfrac{\partial x}{\partial u_2} \\
		& ~~\Rightarrow~~ \dfrac{\partial y}{\partial u_2} = \Gamma^1_{12} y + \Gamma^2_{12} z
	\end{aligned}
\end{equation*}
The initial condition for a conjugate net is:
\begin{description}
	\item[Initial value] $x(0, ~0)$, $\dfrac{\partial x}{\partial u_1}(u_1, ~0)$, $\dfrac{\partial x}{\partial u_2}(0, ~u_2)$
	\item[Parameter] $\Gamma^1_{12}$, $\Gamma^2_{12}$ for all $u_1, ~u_2$
\end{description}
From integration along the coordinate curves, the above initial value is equivalent to:
\begin{equation*}
	\begin{gathered}
		x(0, ~0), ~\dfrac{\partial x}{\partial u_1}(u_1, ~0), ~\dfrac{\partial x}{\partial u_2}(0, ~u_2)  ~~\Leftrightarrow~~ x(0, ~0), ~x(u_1, ~0), ~x(0, ~u_2)
	\end{gathered}
\end{equation*}

\subsection*{Curvature line net}

For a curvature line net, $\mathrm{I}_{12} = \mathrm{II}_{22} = 0$, let $y = av$, $a \in \mathbb{R}_+$ is the norm of $y$, $v \in \mathbb{R}^3$ is the direction vector of $y$, $\|v\| = 1$; $z = bw$, $b \in \mathbb{R}_+$ is the norm of $z$, $w \in \mathbb{R}^3$ is the direction vector of $z$, $\|w\| = 1$. 
\begin{equation*}
	\dfrac{\partial v}{\partial u_2} = \dfrac{\partial }{\partial u_2} \left( \dfrac{\dfrac{\partial x}{\partial u_1}}{\left\| \dfrac{\partial x}{\partial u_1} \right\|} \right) = \dfrac{1}{\left\| \dfrac{\partial x}{\partial u_1} \right\|} \left( \dfrac{\partial^2 x}{\partial u_1 \partial u_2} - \dfrac{ \dfrac{\partial x}{\partial u_1} \cdot \dfrac{\partial^2 x}{\partial u_1 \partial u_2}  }{\left\| \dfrac{\partial x}{\partial u_1} \right\|^2} \dfrac{\partial x}{\partial u_1} \right)
\end{equation*}
we could see that $(\partial v / \partial u_2) \cdot v= 0$ and $(\partial v / \partial u_2) \cdot N = 0$, hence $\partial v / \partial u_2$ is along $w$, and we define $\partial v / \partial u_2 = \beta_2 w$, similarly $\partial w / \partial u_1 = \beta_1 v$. Here $\beta_1, ~\beta_2$ are the \textit{rotational coefficients}. 
\begin{equation*}
	\beta_1 = \dfrac{\dfrac{\partial^2 x}{\partial u_1 \partial u_2} \cdot \dfrac{\partial x}{\partial u_1}}{\left\| \dfrac{\partial x}{\partial u_1} \right\|\left\| \dfrac{\partial x}{\partial u_2} \right\|}, ~~\beta_2 = \dfrac{\dfrac{\partial^2 x}{\partial u_1 \partial u_2} \cdot \dfrac{\partial x}{\partial u_2}}{\left\| \dfrac{\partial x}{\partial u_1} \right\|\left\| \dfrac{\partial x}{\partial u_2} \right\|}
\end{equation*}
The system for a curvature line net is:
\begin{equation*}
	\begin{dcases}
		\dfrac{\partial x}{\partial u_1} = av, ~~\dfrac{\partial x}{\partial u_2} = bw \\	
		\dfrac{\partial v}{\partial u_2} = \beta_2 w, ~~\dfrac{\partial w}{\partial u_1} = \beta_1 v \\
		\dfrac{\partial a}{\partial u_2} = \beta_1 b, ~~\dfrac{\partial b}{\partial u_1} = \beta_2 a
	\end{dcases}
\end{equation*}
The parameters $\beta_1$ and $\beta_2$ are not independent due to the orthogonality. 
\begin{equation*}
	\dfrac{\partial^2 (v \cdot w)}{\partial u_1 \partial u_2} = 0 ~~\Rightarrow~~ \dfrac{\partial \beta_1}{\partial u_1} + \dfrac{\partial \beta_2}{\partial u_2} + \dfrac{\partial v}{\partial u_1} \cdot \dfrac{\partial w}{\partial u_2} = 0
\end{equation*}
\cite[Section 1.4]{bobenko_discrete_2008} indicates that the system can be characterized by 
\begin{equation*}
	\eta = \dfrac{1}{2} \left(\dfrac{\partial \beta_1}{\partial u_1} - \dfrac{\partial \beta_2}{\partial u_2}\right)
\end{equation*}
The initial condition for a curvature line net is:
\begin{description}
	\item[Initial value] $x(0, ~0)$, $\dfrac{\partial x}{\partial u_1}(u_1, ~0)$, $\dfrac{\partial x}{\partial u_2}(0, ~u_2)$
	\item[Parameter] $\eta$ for all $u_1, ~u_2$
\end{description}
From integration along the coordinate curves, the above initial value is equivalent to:
\begin{equation*}
	\begin{gathered}
		x(0, ~0), ~\dfrac{\partial x}{\partial u_1}(u_1, ~0), ~\dfrac{\partial x}{\partial u_2}(0, ~u_2)  ~~\Leftrightarrow~~ x(0, ~0), ~x(u_1, ~0), ~x(0, ~u_2)
	\end{gathered}
\end{equation*}

\section{Discrete curve and surface} \label{section: discrete curvature}

We will see that an $m$-dimensional discrete surface in $\mathbb{R}^n$ ($m,~n \in \mathbb{Z}_+, ~ m \le n$) is a group of scatter points.

\begin{defn}
	An \textit{$m$-dimensional discrete surface} $X$ in $\mathbb{R}^n$ ($m,~n \in \mathbb{Z}_+, ~ m \le n$) is the range of a mapping $\Phi: i \in \mathbb{Z}^m \rightarrow x \in X \subset \mathbb{R}^n$. We say $X$ is a \textit{discrete curve} when $m=1$ and $X$ is a \textit{discrete surface} when $m=2$ in $\mathbb{R}^3$. Here $\mathbb{Z}^m$ is the \textit{parameter domain}.
\end{defn}

Similar to the regularity condition we applied for a chart, we apply the following \textbf{regularity condition} to all the discrete curves and surfaces: the partial difference $\triangle x$ has non-zero components and full rank everywhere:		
\begin{equation*}
	\begin{gathered}
		\triangle x = 
		\begin{bmatrix}
			\triangle_1 x_1 & \triangle_2 x_1 & \cdots & \triangle_m x_1 \\[6pt]
			\triangle_1 x_2 & \triangle_2 x_2 & \cdots & \triangle_m x_2\\[6pt]
			\dots & \dots & \dots & \dots \\[6pt]
			\triangle_1 x_n & \triangle_2 x_n & \cdots & \triangle_m x_n
		\end{bmatrix} \\
		\triangle_{k} x_j(i) = x_j(i_1, ~i_2, ~\cdots, ~i_k+1, ~\cdots, ~i_m) - x_j(i_1, ~i_2, ~\cdots, ~i_k, ~\cdots, ~i_m)  
	\end{gathered}
\end{equation*}

Regarding a discrete curve $X$, an immediate consideration is to introduce a discrete Frenet-Serret frame $(\boldsymbol{x}_\mathrm{t}, ~\boldsymbol{x}_\mathrm{n}, ~\boldsymbol{x}_\mathrm{b})$ attached to every node $x(i) \in X, ~i \in \mathbb{Z}$. Note the use of bold symbols to represent the basis of a vector space, distinct from the coordinates of a point. We define
\begin{equation}
	\begin{gathered}
		\boldsymbol{x}_\mathrm{t}(i) = \dfrac{\triangle x}{\triangle s} =  \dfrac{x(i+1)-x(i)}{\|x(i+1)-x(i)\|} \\
		\boldsymbol{x}_\mathrm{b}(i) = \dfrac{\boldsymbol{x}_\mathrm{t}(i-1) \times \boldsymbol{x}_\mathrm{t}(i)}{ \|\boldsymbol{x}_\mathrm{t}(i-1) \times \boldsymbol{x}_\mathrm{t}(i)\| } \\ 
		\boldsymbol{x}_\mathrm{n}(i) = \boldsymbol{x}_\mathrm{b}(i) \times \boldsymbol{x}_\mathrm{t}(i) = \dfrac{-\boldsymbol{x}_\mathrm{t}(i-1)+(\boldsymbol{x}_\mathrm{t}(i-1) \cdot \boldsymbol{x}_\mathrm{t}(i)) \boldsymbol{x}_\mathrm{t}(i)}{\|-\boldsymbol{x}_\mathrm{t}(i-1)+(\boldsymbol{x}_\mathrm{t}(i-1) \cdot \boldsymbol{x}_\mathrm{t}(i)) \boldsymbol{x}_\mathrm{t}(i)\|} \\
	\end{gathered}
\end{equation}
If $\boldsymbol{x}_\mathrm{t}(i-1)$ is parallel to $\boldsymbol{x}_\mathrm{t}(i)$, the discrete curve $X$ is locally a line at $x(i)$, then $\boldsymbol{x}_\mathrm{b}(i)$ can be determined by other methods. For example, the interpolation of its surrounding values when there is no cluster of zero. The discrete curvature and torsion are calculated from the change rate of these unit vectors:
\begin{equation}
	\begin{gathered}
		\kappa(i) = \dfrac{\|\triangle \boldsymbol{x}_\mathrm{t}(i)\|}{\triangle s} = \dfrac{\|\boldsymbol{x}_\mathrm{t}(i) - \boldsymbol{x}_\mathrm{t}(i-1)\|}{\|x(i)-x(i-1)\|} \\
		\tau(i) = \dfrac{\|\triangle \boldsymbol{x}_\mathrm{b}(i)\|}{\triangle s} = \dfrac{\|\boldsymbol{x}_\mathrm{b}(i+1) - \boldsymbol{x}_\mathrm{b}(i)\|}{\|x(i+1)-x(i)\|} \\
	\end{gathered}
\end{equation}
We need $x(i-1)$, $x(i)$ and $x(i+1)$ to calculate $\kappa(i)$; and $x(i-1)$, $x(i)$, $x(i+1)$ and $x(i+2)$ to calculate $\tau(i)$.

Regarding a discrete surface, the motivation for calculating the discrete mean curvature vector and the discrete Gaussian curvature is from Proposition \ref{prop: geometric interpretation of curvature} below. Let $X$ be a parametrized surface with chart $\Phi: t \in I^2 \rightarrow x \in X \subset \mathbb{R}^3$. Suppose there is a point $x$ where the Gaussian curvature $\kappa_\mathrm{G}(x) \neq 0$. $O_1(x)$ is a neighbourhood where $\kappa_\mathrm{G}$ does not change sign. $O_1(x) \supset O_2(x) \supset \dots \supset O_n(x)$ is a sequence of neighbourhoods at $x$ whose diameter satisfies:
\begin{equation*}
	\lim_{n \rightarrow \infty} \mathrm{diam}(O_n(x)) = 0
\end{equation*}
The diameter of a set refers to the supremum of distances between points within the set. 
\begin{equation*}
	\mathrm{diam}(Y) = \sup (d(x,~y)),~ \mathrm{for~all}~x,~y \in Y
\end{equation*}

\begin{prop} \label{prop: geometric interpretation of curvature}
	The normal vector, mean curvature and Gaussian curvature satisfy the equation below:
	\begin{equation*}
		\begin{gathered}
			2\kappa_\mathrm{H}(x) N(x) = -\lim_{n \rightarrow +\infty} \dfrac{\nabla_x \mathrm{area}(O_n(x))}{\mathrm{area}(O_n(x))} \\
			\kappa_\mathrm{G}(x) = \lim_{n \rightarrow +\infty} \dfrac{\mathrm{area}(N(O_n(x)))}{\mathrm{area}(O_n(x))}
		\end{gathered}
	\end{equation*}
	$\mathrm{area}(O_n(x))$ is the surface area of $O_n(x)$. $\mathrm{area}(N(O_n(x)))$ is the area of $N(O_n(x))$, which is the Gauss map of $O_n(x)$.
\end{prop}

\begin{proof}
	The area of $O_n(x)$ is:
	\begin{equation*}
		\mathrm{area}(O_n(x)) = \bigintsss_{O_n(x)} \sqrt{\mathrm{I}_{11}\mathrm{I}_{22}-\mathrm{I}_{12}^2} \dif u_1 \dif u_2
	\end{equation*}
	We will then consider the \textbf{normal variation} $x \rightarrow x^\epsilon = x+\epsilon h N$ controlled by a distribution $h: O_1(x) \rightarrow \mathbb{R}$, and $\epsilon \in \mathbb{R}_+$ is a scaling factor. The reason for only considering the normal variation is that the limit of area does not change through tangential variation.
	\begin{equation*}
		\mathrm{area}(O_n^\epsilon(x^\epsilon)) = \bigintsss_{O_n(x)} \sqrt{\mathrm{I}^\epsilon_{11}\mathrm{I}^\epsilon_{22}-\left(\mathrm{I}^\epsilon_{12}\right)^2} \dif u_1 \dif u_2
	\end{equation*}
	The gradient of $\mathrm{area}(O_n(x))$ is the integral of the directional derivative along the normal vector:
	\begin{equation*}
		\| \nabla_x \mathrm{area}(O_n(x)) \| = \lim_{\epsilon \rightarrow 0} \bigintsss_{O_n(x)} \dfrac{\sqrt{\mathrm{I}^\epsilon_{11}\mathrm{I}^\epsilon_{22}-\left(\mathrm{I}^\epsilon_{12}\right)^2} - \sqrt{\mathrm{I}_{11}\mathrm{I}_{22}-\mathrm{I}_{12}^2}}{\epsilon h} \dif u_1 \dif u_2
	\end{equation*}
	Continue the calculation:
	\begin{equation*}
		\begin{gathered}
			\dfrac{\partial x^\epsilon}{\partial u_1} = \dfrac{\partial x}{\partial u_1} +\epsilon \dfrac{\partial h}{\partial u_1} N + \epsilon h \dfrac{\partial N}{\partial u_1} \\
			\dfrac{\partial x^\epsilon}{\partial u_2} = \dfrac{\partial x}{\partial u_2} +\epsilon \dfrac{\partial h}{\partial u_2} N + \epsilon h \dfrac{\partial N}{\partial u_2}
		\end{gathered}
	\end{equation*}
	\begin{equation*}
		\begin{gathered}
			\mathrm{I}^\epsilon_{11} = \mathrm{I}_{11} - 2 \epsilon h \mathrm{II}_{11} + \epsilon^2 \dfrac{\partial h}{\partial u_1} \dfrac{\partial h}{\partial u_1} + \epsilon^2h^2\mathrm{III}_{11} \\
			\mathrm{I}^\epsilon_{12} = \mathrm{I}_{12} - 2 \epsilon h \mathrm{II}_{12} + \epsilon^2 \dfrac{\partial h}{\partial u_1} \dfrac{\partial h}{\partial u_2} + \epsilon^2h^2\mathrm{III}_{12} \\
			\mathrm{I}^\epsilon_{22} = \mathrm{I}_{22} - 2 \epsilon h \mathrm{II}_{22} + \epsilon^2 \dfrac{\partial h}{\partial u_2} \dfrac{\partial h}{\partial u_2} + \epsilon^2h^2\mathrm{III}_{22}
		\end{gathered}
	\end{equation*}
	\begin{equation*}
		\begin{aligned}
			\mathrm{I}^\epsilon_{11}\mathrm{I}^\epsilon_{22} - \left(\mathrm{I}^\epsilon_{12}\right)^2 & = \mathrm{I}_{11}\mathrm{I}_{22}-\mathrm{I}_{12}^2 \\
			& -2\epsilon h (\mathrm{I}_{11}\mathrm{II}_{22} - 2\mathrm{I}_{12}\mathrm{II}_{12} + \mathrm{I}_{22}\mathrm{II}_{11}) \\
			& + 4 \epsilon^2 h^2 (\mathrm{II}_{11}\mathrm{II}_{22}-\mathrm{II}_{12}^2) \\
			& + \epsilon^2 h^2 (\mathrm{I}_{11}\mathrm{III}_{22}-2\mathrm{I}_{12}\mathrm{III}_{12}+\mathrm{I}_{22}\mathrm{III}_{11}) \\
			& + \epsilon^2 \left(\mathrm{I}_{11} \dfrac{\partial h}{\partial u_2} \dfrac{\partial h}{\partial u_2} - 2\mathrm{I}_{12} \dfrac{\partial h}{\partial u_1} \dfrac{\partial h}{\partial u_2} + \mathrm{I}_{22} \dfrac{\partial h}{\partial u_1} \dfrac{\partial h}{\partial u_1} \right) \\
			& + o(\epsilon^2)
		\end{aligned}
	\end{equation*}
	$o(\epsilon^2)$ means terms over $\epsilon$ with higher order than 2. From the previous derivation on the third fundamental form, \eqref{eq: fundamental forms relation}, we could see that
	\begin{equation*}
		\begin{aligned}
			& ~ \mathrm{I}_{11}\mathrm{III}_{22}-2\mathrm{I}_{12}\mathrm{III}_{12}+\mathrm{I}_{22}\mathrm{III}_{11} \\
			= & ~ \mathrm{I}_{11}(-\kappa_\mathrm{G}\mathrm{II}_{22}+2\kappa_\mathrm{H}\mathrm{II}_{22})-2\mathrm{I}_{12}(-\kappa_\mathrm{G}\mathrm{II}_{12}+2\kappa_\mathrm{H}\mathrm{II}_{12})+\mathrm{I}_{22}(-\kappa_\mathrm{G}\mathrm{II}_{11}+2\kappa_\mathrm{H}\mathrm{II}_{11}) \\
			= & ~ -2\kappa_\mathrm{G}(\mathrm{I}_{11}\mathrm{I}_{22}-\mathrm{I}_{12}^2) + 2\kappa_\mathrm{H} (\mathrm{I}_{11}\mathrm{II}_{22} - 2\mathrm{I}_{12}\mathrm{II}_{12} + \mathrm{I}_{22}\mathrm{II}_{11})
		\end{aligned}
	\end{equation*} 
	Furthermore, use the definition of the mean and Gaussian curvature
	\begin{equation} \label{eq: curvatures}
		\begin{gathered}
			\kappa_\mathrm{H} =  \dfrac{\mathrm{II}_{22}\mathrm{I}_{11}-2\mathrm{II}_{12}\mathrm{I}_{12}+\mathrm{II}_{11}\mathrm{I}_{22}}{2 \left( \mathrm{I}_{11}\mathrm{I}_{22}-\mathrm{I}_{12}^2 \right)} \\
			\kappa_\mathrm{G} =  \dfrac{\mathrm{II}_{11}\mathrm{II}_{22}-\mathrm{II}_{12}^2}{\mathrm{I}_{11}\mathrm{I}_{22}-\mathrm{I}_{12}^2}
		\end{gathered}
	\end{equation}
	we could obtain:
	\begin{equation*}
		\begin{aligned}
			\dfrac{\mathrm{I}^\epsilon_{11}\mathrm{I}^\epsilon_{22} - \left(\mathrm{I}^\epsilon_{12}\right)^2}{\mathrm{I}_{11}\mathrm{I}_{22}-\mathrm{I}_{12}^2} & = 1 - 4 \kappa_\mathrm{H} \epsilon h + (4 \kappa_\mathrm{H}^2 + 2\kappa_\mathrm{G}) \epsilon^2 h^2  \\
			& + \left(\mathrm{I}_{11} \dfrac{\partial h}{\partial u_2} \dfrac{\partial h}{\partial u_2} - 2\mathrm{I}_{12} \dfrac{\partial h}{\partial u_1} \dfrac{\partial h}{\partial u_2} + \mathrm{I}_{22} \dfrac{\partial h}{\partial u_1} \dfrac{\partial h}{\partial u_1} \right) \epsilon^2 / (\mathrm{I}_{11}\mathrm{I}_{22}-\mathrm{I}_{12}^2) \\
			& + o(\epsilon^2)
		\end{aligned}
	\end{equation*}
	In the calculation of surface gradient, only the first-order term is needed. By applying the Mean Value Theorem for double integral, we could see that:
	\begin{equation*}
		\begin{aligned}
			\dfrac{\| \nabla_x \mathrm{area}(O_n(x)) \|}{\mathrm{area}(O_n(x))} & = \lim_{\epsilon \rightarrow 0} \dfrac{1}{\epsilon h}\left(\sqrt{\dfrac{\mathrm{I}^\epsilon_{11}\mathrm{I}^\epsilon_{22}-\left(\mathrm{I}^\epsilon_{12}\right)^2}{\mathrm{I}_{11}\mathrm{I}_{22}-\mathrm{I}_{12}^2}} - 1\right) \mathrm{~~at~some~} y \in O_n(x) \\
			& = -2 \kappa_\mathrm{H} \mathrm{~~at~some~} y \in O_n(x)
		\end{aligned}
	\end{equation*}
	hence
	\begin{equation*}
		2\kappa_\mathrm{H}(x) N(x) = -\lim_{n \rightarrow +\infty} \dfrac{\nabla \mathrm{area}(O_n(x))}{\mathrm{area}(O_n(x))} 
	\end{equation*}
	Next, apply the derivative of the Gauss map, \eqref{eq: derivative of the Gauss map}, we have:
	\begin{equation*} 
		\begin{aligned}
			\mathrm{area}(N(O_n(x))) & = \bigintsss_{O_n(x)} \left\| \dfrac{\partial N}{\partial u_1} \times \dfrac{\partial N}{\partial u_2}\right\| \dif u_1 \dif u_2 \\
			& = \bigintsss_{O_n(x)} (a_{11}a_{22} - a_{12}a_{21}) \left\| \dfrac{\partial x}{\partial u_1} \times \dfrac{\partial x}{\partial u_2}\right\| \dif u_1 \dif u_2 \\
			& = \bigintsss_{O_n(x)} \kappa_\mathrm{G} \sqrt{\mathrm{I}_{11}\mathrm{I}_{22}-\mathrm{I}_{12}^2} \dif u_1 \dif u_2
		\end{aligned}
	\end{equation*}
	Using the Mean Value Theorem for double integral:
	\begin{equation*}
		\begin{aligned}
			\dfrac{\mathrm{area}(N(O_n(x)))}{\mathrm{area}(O_n(x))} = \kappa_\mathrm{G} \mathrm{~~at~some~} y \in O_n(x) \\
		\end{aligned}
	\end{equation*}
	hence
	\begin{equation*}
		\begin{aligned}
			\lim_{n \rightarrow +\infty} \dfrac{\mathrm{area}(N(O_n(x)))}{\mathrm{area}(O_n(x))} = \kappa_\mathrm{G}(x) \\
		\end{aligned}
	\end{equation*}
\end{proof}

\begin{rem}
	The famous \textit{Steiner formula} considers the \textbf{uniform normal variation} when $h = 1$.
	\begin{equation*}
		\begin{aligned}
			\dfrac{\mathrm{I}^\epsilon_{11}\mathrm{I}^\epsilon_{22} - \left(\mathrm{I}^\epsilon_{12}\right)^2}{\mathrm{I}_{11}\mathrm{I}_{22}-\mathrm{I}_{12}^2} & = 1 - 4 \kappa_\mathrm{H} \epsilon + (4 \kappa_\mathrm{H}^2 + 2\kappa_\mathrm{G}) \epsilon^2 + o(\epsilon^2)
		\end{aligned}
	\end{equation*}
	then
	\begin{equation*}
		\begin{aligned}
			\sqrt{\dfrac{\mathrm{I}^\epsilon_{11}\mathrm{I}^\epsilon_{22} - \left(\mathrm{I}^\epsilon_{12}\right)^2}{\mathrm{I}_{11}\mathrm{I}_{22}-\mathrm{I}_{12}^2}} & = 1 - 2 \kappa_\mathrm{H} \epsilon + \kappa_\mathrm{G} \epsilon^2 + o(\epsilon^2)
		\end{aligned}
	\end{equation*}
	Geometrically,
	\begin{equation*}
		\lim_{n \rightarrow +\infty} \dfrac{\mathrm{area}(O_n^\epsilon(x^\epsilon))}{\mathrm{area}(O_n(x))} = 1 - 2 \kappa_\mathrm{H}(x) \epsilon + \kappa_\mathrm{G}(x) \epsilon^2 + o(\epsilon^2) 
	\end{equation*}
\end{rem}

\begin{figure}[tbph]
	\noindent \begin{centering}
		\includegraphics[width=1\linewidth]{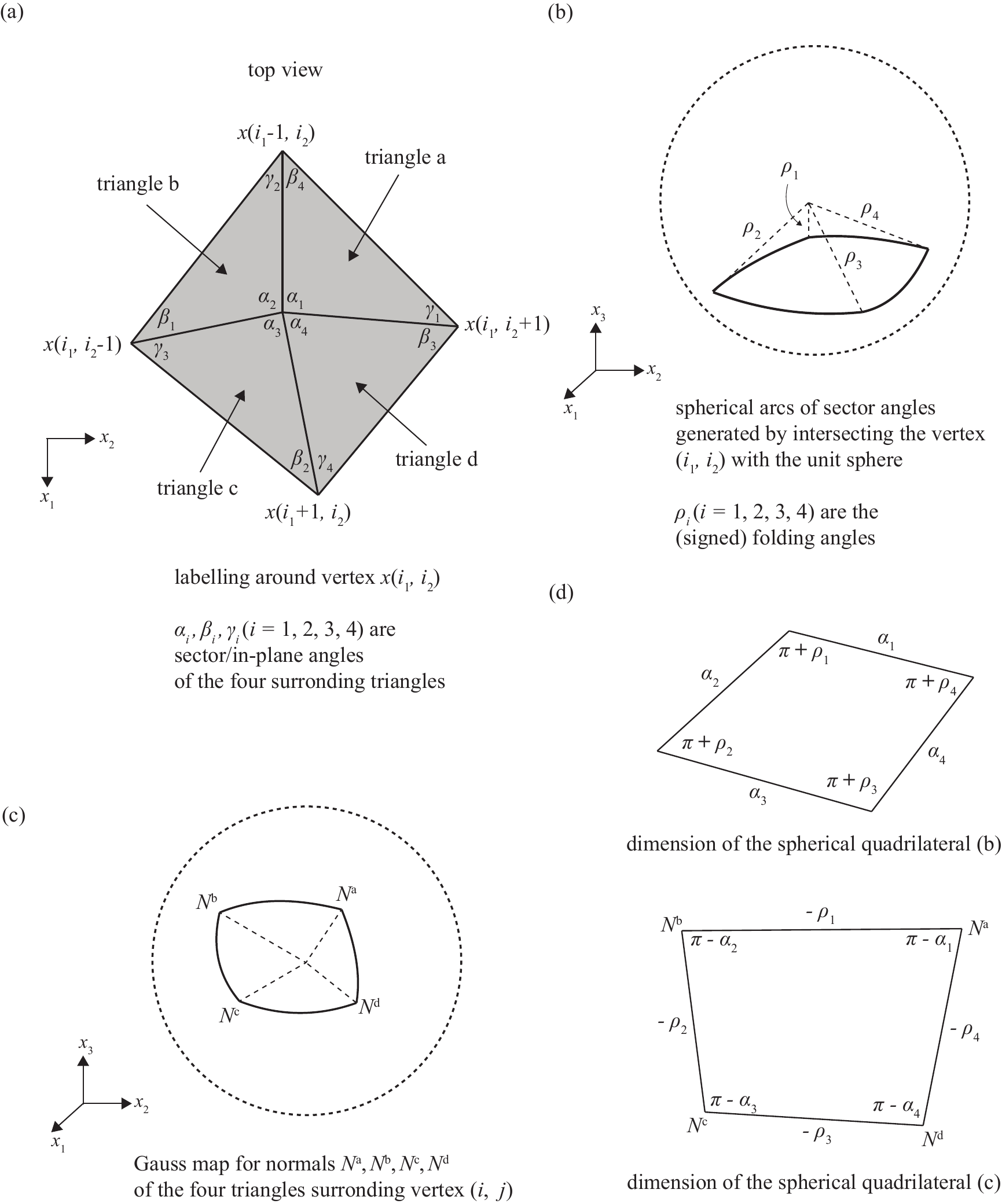}
		\par\end{centering}
	
	\caption{\label{fig: discrete curvature}Labelling around vertex $x(i_1, ~i_2)$ in the calculation of the discrete mean curvature vector and the discrete Gaussian curvature. Note that in (b), at each crease, a signed \textit{folding angle} $\rho_i$ is the angle between the normal vectors of its two adjacent panels. If these two normal vectors meet on the specified side of the paper (here, upwards), $\rho_i \in (0, \pi)$, the crease is called a \textit{valley crease}. If these two normal vectors meet on the opposite side, $\rho_i \in (-\pi, 0)$, the crease is called a \textit{mountain crease}. (b) shows four mountain creases. In (d) we could see that the two spherical quadrilaterals are dual/polar-and-poles to each other.}
\end{figure}

We will show how to use the above formula in Proposition \ref{prop: geometric interpretation of curvature} to calculate the discrete mean curvature vector $\kappa_\mathrm{H} N$ and the discrete Gaussian curvature $\kappa_\mathrm{G}$. For every $x(i), ~i = (i_1, ~i_2) \in \mathbb{Z}^2$ on a discrete surface, we need the information of $x(i_1-1, ~i_2)$, $x(i_1+1, ~i_2)$, $x(i_1, ~i_2-1)$ and $x(i_1, ~i_2+1)$ to calculate the area gradient of the four triangles surrounding $x(i)$. Let
\begin{equation*}
	\begin{gathered}
		a = x(i_1, ~i_2) - x(i_1-1, ~i_2), ~~ b = x(i_1, ~i_2) - x(i_1, ~i_2-1) \\
		c = x(i_1, ~i_2) - x(i_1+1, ~i_2), ~~ d = x(i_1, ~i_2) - x(i_1, ~i_2+1) 
	\end{gathered}
\end{equation*}
and the sum of area of the four triangles is:
\begin{equation*}
	\begin{aligned}
		\mathrm{area}(i) & = \dfrac{1}{2} \left(\|d \times a\| + \|a \times b\| + \|b \times c\| + \|c \times d\| \right) 
	\end{aligned}
\end{equation*}
Note that the $\mathrm{area}(i)$ lives on vertex $i$. The derivative of $\mathrm{area}(i)$ with respect to $x(i)$ can be directly calculated. For example:
\begin{equation*}
	\dfrac{\dif \|d \times a\|}{\dif x} = 	\dfrac{1}{\|d \times a\|} \begin{bmatrix}
		(d \times a) \cdot \dfrac{\partial (d \times a)}{\partial x_1} \\[8pt]
		(d \times a) \cdot \dfrac{\partial (d \times a)}{\partial x_2} \\[8pt]
		(d \times a) \cdot \dfrac{\partial (d \times a)}{\partial x_3}
	\end{bmatrix} = \dfrac{\dif (d \times a)}{\dif x} \cdot \dfrac{d \times a}{\|d \times a\|}
\end{equation*}
and for $a = (a_1; ~a_2; ~a_3)$ and $d = (d_1; ~d_2; ~d_3)$:
\begin{equation*}
	\dfrac{\dif (d \times a)}{\dif x} = \begin{bmatrix}
		0 & a_3-d_3 & d_2-a_2 \\
		d_3-a_3 & 0 & a_1-d_1 \\
		a_2-d_2 & d_1-a_1 & 0
	\end{bmatrix}
\end{equation*}
Using the information above we could obtain the expression below in terms of cross product:
\begin{equation*}
	\begin{aligned}
		\nabla \mathrm{area}(i) = & \dfrac{1}{2} \left((d-a) \times \dfrac{d \times a}{\|d \times a\|} + (a-b) \times \dfrac{a \times b}{\|a \times b\|} \right. \\
		& \left. ~~ + (b-c) \times \dfrac{b \times c}{\|b \times c\|} + (c-d) \times \dfrac{c \times d}{\|c \times d\|} \right)\\
	\end{aligned}	
\end{equation*}
Physically, $\nabla \mathrm{area}(i)$ indicates the steepest direction pulling at vertex $i$ to increase $\mathrm{area}(i)$ of the four triangles. Then apply the formula for triple cross product we will obtain the final expression, as known as the \textbf{cotan formula}, using the angles defined in Figure \ref{fig: discrete curvature}:
\begin{equation} \label{eq: discrete mean curvature 2}
	\begin{aligned}
		\nabla \mathrm{area}(i) = ~ & \dfrac{1}{2} \left( \left(\dfrac{1}{\tan \beta_1}+\dfrac{1}{\tan \gamma_1} \right) a  + \left(\dfrac{1}{\tan \beta_2}+\dfrac{1}{\tan \gamma_2} \right) b \right. \\
		& \left. ~ + \left(\dfrac{1}{\tan \beta_3}+\dfrac{1}{\tan \gamma_3} \right) c +  \left(\dfrac{1}{\tan \beta_4}+\dfrac{1}{\tan \gamma_4} \right) d \right)\\
	\end{aligned}
\end{equation}
Hence the discrete mean curvature vector, i.e., the \textbf{Laplace-Beltrami Operator} is:
\begin{equation} \label{eq: discrete mean curvature}
	\begin{gathered}
		2\kappa_\mathrm{H}(i) N(i) = \dfrac{\nabla \mathrm{area}(i)}{\mathrm{area}(i)}, ~~\|N(i)\| = 1
	\end{gathered}
\end{equation}
Here $\kappa_\mathrm{H}(i)$ and $N(i)$ both live on vertex $i$. Note that if $\nabla \mathrm{area}(i) = 0$, for example when the five points in Figure \ref{fig: discrete curvature} are coplanar, we take $N(i)$ as the average of the surrounding normal vectors. 

Next, we can calculate the normal of the surrounding four triangles, whose spherical view is provided in Figure \ref{fig: discrete curvature}(c)
\begin{equation*}
	\begin{gathered}
		N^\mathrm{a} = \dfrac{d \times a}{\|d\|\|a\| \sin \alpha_1}, ~~N^\mathrm{b} = \dfrac{a \times b}{\|a\|\|b\| \sin \alpha_2} \\
		N^\mathrm{c} = \dfrac{b \times c}{\|b\|\|c\| \sin \alpha_3}, ~~N^\mathrm{d} = \dfrac{c \times d}{\|c\|\|d\| \sin \alpha_4} 
	\end{gathered}
\end{equation*}
\begin{equation*}
	\begin{gathered}
		N^\mathrm{a} \cdot N^\mathrm{b} = \cos \rho_1, ~~ N^\mathrm{b} \cdot N^\mathrm{c} = \cos \rho_2 \\
		N^\mathrm{c} \cdot N^\mathrm{d} = \cos \rho_3, ~~ N^\mathrm{d} \cdot N^\mathrm{a} = \cos \rho_4
	\end{gathered}
\end{equation*}
Note that the Gauss map is an \textit{involution} (a mapping is its inverse) of the direction vectors along $a$, $b$, $c$, $d$. 
\begin{equation*}
	\begin{gathered}
		\dfrac{a}{\|a\|} = -\dfrac{N^\mathrm{a} \times N^\mathrm{b}}{\sin \rho_1}, ~~ \dfrac{b}{\|b\|} = -\dfrac{N^\mathrm{b} \times N^\mathrm{c}}{\sin \rho_2} \\
		\dfrac{c}{\|c\|} = -\dfrac{N^\mathrm{c} \times N^\mathrm{d}}{\sin \rho_3}, ~~ \dfrac{d}{\|d\|} = -\dfrac{N^\mathrm{d} \times N^\mathrm{a}}{\sin \rho_4} 
	\end{gathered}
\end{equation*}
The geometrical reason is $a$ being orthogonal to both $N^\mathrm{a}$ and $N^\mathrm{b}$. The same principle holds for the rest.

An important fact from spherical trigonometry is that the spherical linkage sharing identical motion with the degree-4 vertex shown in Figure \ref{fig: discrete curvature}(b) is the \textbf{polar} quadrilateral of the spherical quadrilateral formed by the Gauss map shown in Figure \ref{fig: discrete curvature}(c). The sector angles and folding angles are therefore related as indicated in Figure \ref{fig: discrete curvature}(d). Further, the area of a spherical quadrilateral is the sum of interior angles minus $2\pi$ (also called the \textbf{angular defect}), which leads to the calculation of discrete Gaussian curvature:
\begin{equation} \label{eq: discrete Gaussian curvature}
	\kappa_\mathrm{G}(i) = 2\pi - \alpha_1 - \alpha_2 - \alpha_3 - \alpha_4
\end{equation}
$\kappa_\mathrm{G}(i)$ also lives on vertex $i$. 

The above calculation method of the discrete mean and Gaussian curvature is just one of the admissible definitions. In practice the calculation might be altered for different discrete surfaces, for example, \cite{meyer_discrete_2003} and the edge-constrained net introduced in \cite{hoffmann_discrete_2017}. 

\section{Discrete nets} \label{section: discrete net}

This section will introduce the discrete analogue of coordinate nets derived in Section \ref{section: coordinate net}. Here `discrete analogue' means the discrete system (usually a partial difference system) defined by a discrete coordinate net is a discretization of a smooth system (usually a partial differential system). We will show the conversion between the smooth and discrete notations by examining multiple discrete nets in line with the smooth nets provided in Section \ref{section: coordinate net} and from the discussion on convergence in Section \ref{section: ddg convergence}. It is worth mentioning that there may be multiple approaches to discretize a smooth net, and the choice of discrete net will depend on specific scenarios and requirements (for example, in the simulation of isometric deformation). The labelling of geometrical quantities on a discrete net is provided in Figure \ref{fig: labelling discrete net}.

\begin{figure}[tbph]
	\noindent \begin{centering}
		\includegraphics[width=1\linewidth]{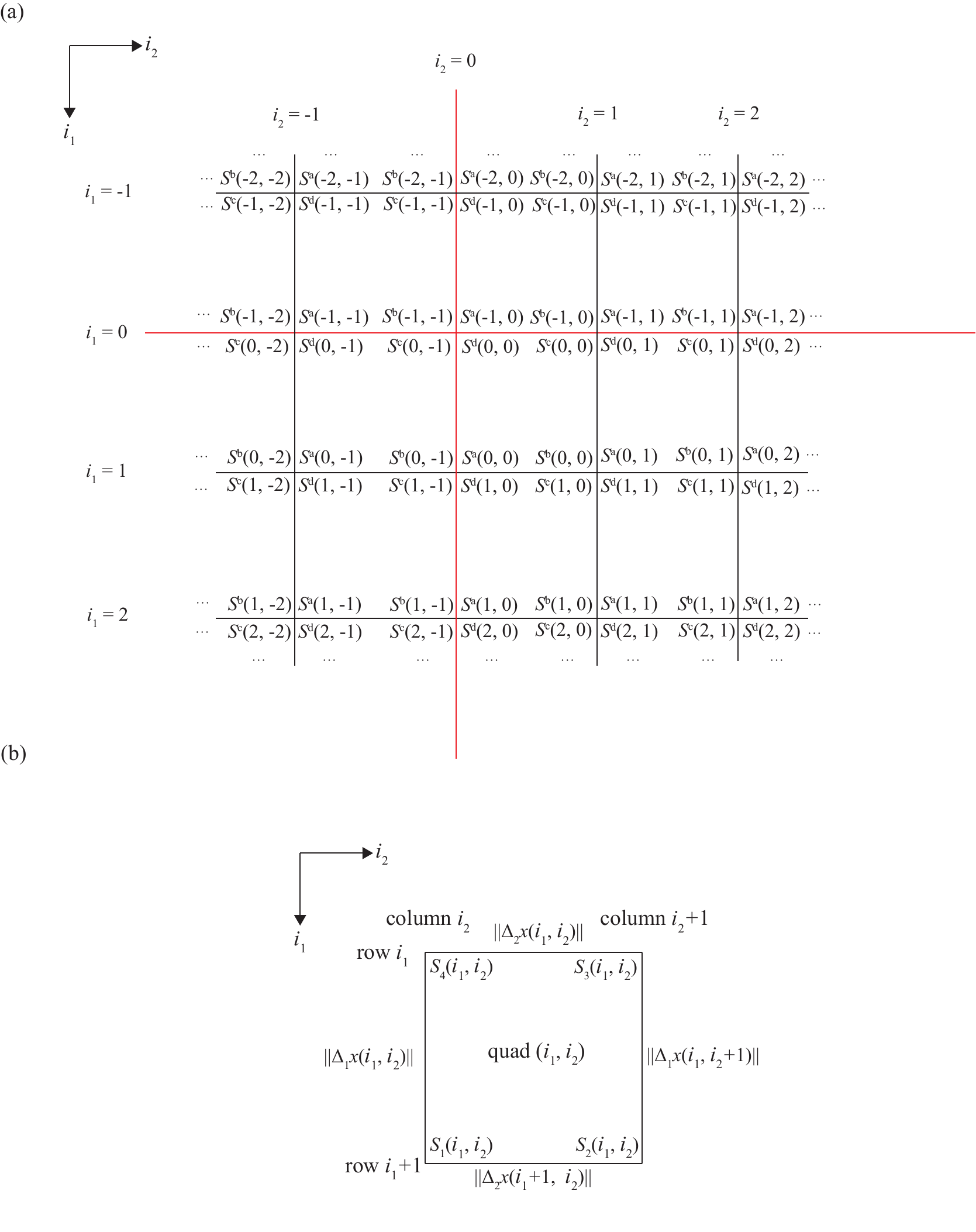}
		\par\end{centering}
	
	\caption{\label{fig: labelling discrete net}Labelling of vertices, lengths and angles of a discrete net. Note that these figures are not a three-dimensional drawing, and $x(i_1, ~i_2), ~x(i_1+1, ~i_2), ~x(i_1+1, ~i_2+1), ~x(i_1, ~i_2+1)$ are \textbf{not necessarily planar}. The sector angles are $S^\mathrm{a}(i), ~S^\mathrm{b}(i), ~S^\mathrm{c}(i), ~S^\mathrm{d}(i), ~i = (i_1, ~i_2) \in \mathbb{Z}^2$. The crease lengths are $\|\triangle_1 x(i)\| = \|x(i_1+1, ~i_2) - x(i_1, ~i_2)\|, ~\|\triangle_2 x(i)\| = \|x(i_1, ~i_2+1) - x(i_1, ~i_2) \|, ~i = (i_1, ~i_2) \in \mathbb{Z}^2$.}
\end{figure}

A discrete surface $X: \mathbb{Z}^2 \rightarrow \mathbb{R}^3$ is called a \textbf{discrete Chebyshev net} if 
\begin{equation} \label{eq: discrete Chebyshev net}
	\triangle_2 \|\triangle_1 x\|^2 = \triangle_1 \|\triangle_2 x\|^2 = 0, ~~\mathrm{for~all~} i = (i_1, ~i_2) \in \mathbb{Z}^2
\end{equation}
The discrete operators are in the form of:
\begin{equation*}
	\begin{gathered}
		\triangle_1 x (i) = x(i_1+1, ~i_2) - x(i_1, ~i_2) \\
		\triangle_2 x (i) = x(i_1, ~i_2+1) - x(i_1, ~i_2) \\
		\triangle_1\triangle_2 x(i) = \triangle_2\triangle_1 x(i) = x(i_1+1, ~i_2+1) - x(i_1, ~i_2+1) - x(i_1+1, ~i_2) + x(i_1, ~i_2)
	\end{gathered}
\end{equation*}
Note that $\triangle_1 x$ lives on grid lines $i_2$, $\triangle_2 x$ lives on grid lines $i_1$, $\triangle_1\triangle_2 x$ lives on quadrilaterals $(i_1, ~i_2)$.
\begin{equation*}
	\begin{aligned}
		\triangle_2 \|\triangle_1 x \|^2 & = \|\triangle_1 x(i_1, ~i_2+1)\|^2 - \|\triangle_1 x(i_1, ~i_2)\|^2 \\
		& = \|x(i_1+1, ~i_2+1) - x(i_1, ~i_2+1)\|^2 - \|x(i_1+1, ~i_2) - x(i_1, ~i_2)\|^2 \\
		& = \triangle_1\triangle_2 x(i) \cdot (x(i_1+1, ~i_2+1) - x(i_1, ~i_2+1) + x(i_1+1, ~i_2) - x(i_1, ~i_2))
	\end{aligned}
\end{equation*}
\begin{equation*}
	\begin{aligned}
		\triangle_1 \|\triangle_2 x \|^2 & = \|\triangle_2 x(i_1+1, ~i_2)\|^2 - \|\triangle_2 x(i_1, ~i_2)\|^2 \\
		& = \|x(i_1+1, ~i_2+1) - x(i_1+1, ~i_2)\|^2 - \|x(i_1, ~i_2+1) - x(i_1, ~i_2)\|^2 \\
		& = \triangle_1\triangle_2 x(i) \cdot (x(i_1+1, ~i_2+1) + x(i_1, ~i_2+1) - x(i_1+1, ~i_2) - x(i_1, ~i_2))
	\end{aligned}
\end{equation*}
Hence the partial difference equation for a discrete Chebyshev net, \eqref{eq: discrete Chebyshev net}, is equivalent to
\begin{equation} \label{eq: partial difference Chebyshev}
	\begin{aligned}
		& \begin{dcases}
			\triangle_1\triangle_2 x(i) \cdot (x(i_1+1, ~i_2+1) - x(i_1, ~i_2)) = 0 \\
			\triangle_1\triangle_2 x(i) \cdot (x(i_1, ~i_2+1) - x(i_1+1, ~i_2)) = 0
		\end{dcases} \\
		~~\Leftrightarrow & \begin{dcases}
			\triangle_1\triangle_2 x(i) = \dfrac{\lambda(i)}{2} (x(i_1+1, ~i_2+1) - x(i_1, ~i_2)) \times (x(i_1, ~i_2+1) - x(i_1+1, ~i_2)) \\ \lambda: \mathbb{Z}^2 \rightarrow \mathbb{R} \mathrm{~living~on~quadrilaterals~} (i_1, ~i_2)
		\end{dcases}
	\end{aligned}
\end{equation}

The reason for choosing $\lambda/2$ is for the consistency with its smooth analogue, \eqref{eq: partial differential Chebyshev}. It can be verified that on the integer grid, if seeing $u_1$ in the direction along $(\sqrt{2}/2, ~\sqrt{2}/2)$, and seeing $u_2$ in the direction along $(-\sqrt{2}/2, ~\sqrt{2}/2)$, we have
\begin{equation*}
	\begin{gathered}
		\triangle_1\triangle_2 x \sim \dfrac{\partial^2 x}{\partial u_1 \partial u_2} \\
		x(i_1+1, ~i_2+1) - x(i_1, ~i_2) \sim \sqrt{2} \dfrac{\partial x}{\partial u_1} \\
		x(i_1, ~i_2+1) - x(i_1+1, ~i_2) \sim \sqrt{2} \dfrac{\partial x}{\partial u_2}
	\end{gathered}
\end{equation*}
hence the amplitude $\lambda$ has the same meaning for both the discrete and smooth case. 

When $\triangle_1\triangle_2 x \neq 0$, see Figure \ref{fig: ChebyShev quad}(a) for a geometric illustration for a Chebyshev quadrilateral. Apply a parallel transport from $BD$ to $B'D'$ such that the intersection $O'$ of $AC$ and $B'D'$ bisects both $AC$ and $B'D'$. The side length condition $AB = CD$, $AD=BC$ implies that both $BB'$ and $DD'$ are perpendicular to the planar parallelogram $AB'CD'$, and we could see that $BB' = DD' = OO' = \triangle_1\triangle_2 x/2$. When $\triangle_1\triangle_2 x = 0$, as shown in Figure \ref{fig: ChebyShev quad}(b), $ABCD$ becomes a planar parallelogram, geometrically flipping $D$ to the other side of plane $AB'CD'$. 

\begin{figure}[t]
	\noindent \begin{centering}
		\includegraphics[width=1\linewidth]{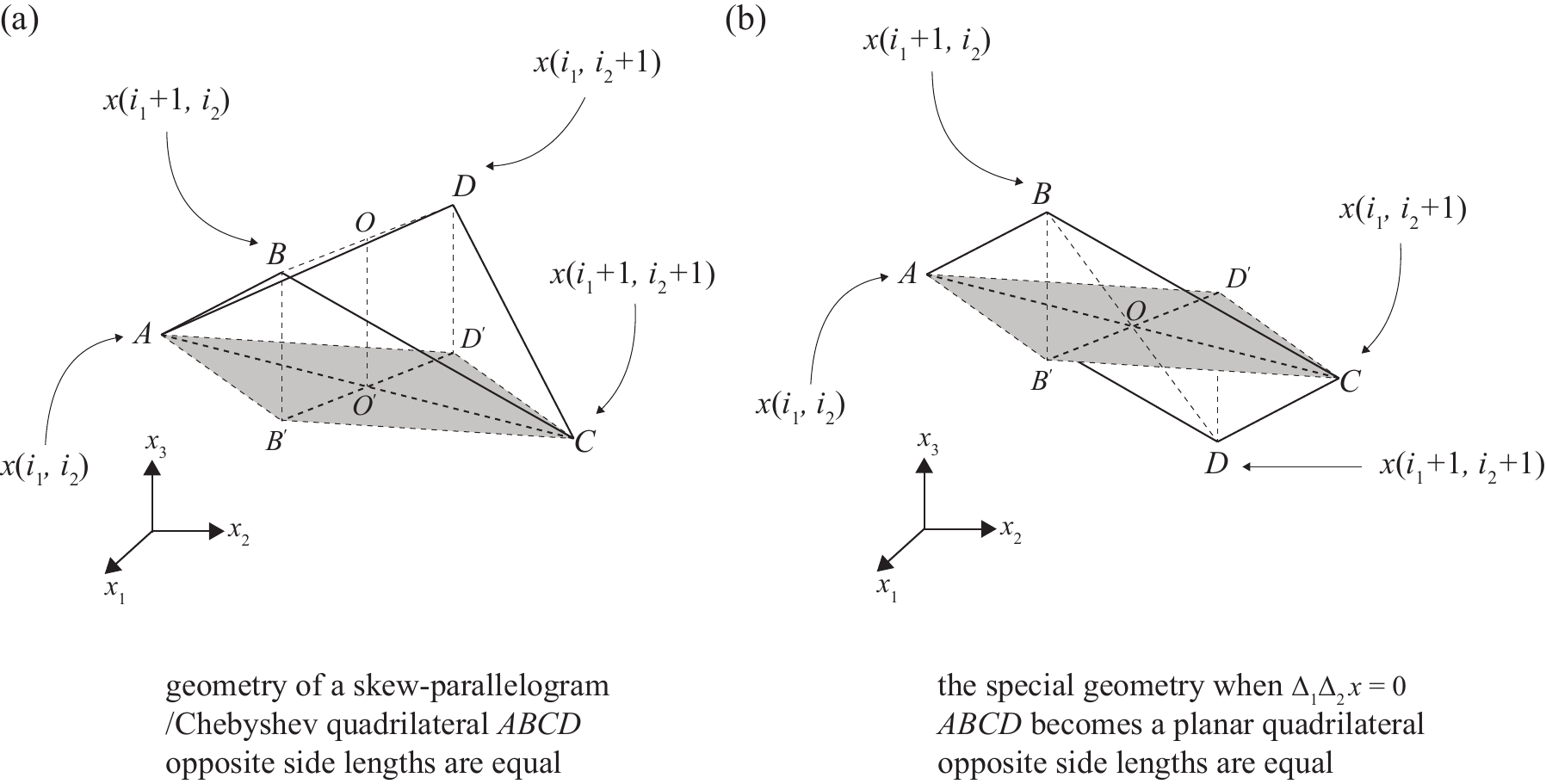}
		\par\end{centering}
	
	\caption{\label{fig: ChebyShev quad}(a) A figure illustrating the geometry of a skew-parallelogram/Chebyshev quadrilateral, where the side lengths $AB = CD$, $AD = BC$. (b) Degeneration to a planar quadrilateral when $\triangle_1\triangle_2 x = 0$. In both figures, $A, ~B, ~C, ~D$ refer to the position of $x(i_1, ~i_2), ~x(i_1+1, ~i_2), ~x(i_1+1, ~i_2+1), ~x(i_1, ~i_2+1)$.}
\end{figure} 

We could see that one reasonable way to define the discrete normal field is to define $N(i), ~i = (i_1, ~i_2) \in \mathbb{Z}^2$ on each quadrilateral $(i_1, ~i_2)$, along the direction of $O'O$:
\begin{equation}
	N(i) = \dfrac{(x(i_1+1, ~i_2+1) - x(i_1, ~i_2)) \times (x(i_1, ~i_2+1) - x(i_1+1, ~i_2))}{\|(x(i_1+1, ~i_2+1) - x(i_1, ~i_2)) \times (x(i_1, ~i_2+1) - x(i_1+1, ~i_2))\|}
\end{equation} 
Additionally, $\lambda(i), ~i = (i_1, ~i_2) \in \mathbb{Z}^2$ for a discrete Chebyshev net shows the `curvature' of a Chebyshev quadrilateral since
\begin{equation*}
	\begin{aligned}
		\lambda(i) & = \dfrac{\triangle_1\triangle_2 x(i) \cdot N(i)}{\|(x(i_1+1, ~i_2+1) - x(i_1, ~i_2)) \times (x(i_1, ~i_2+1) - x(i_1+1, ~i_2))\|} \\
		& = \dfrac{2 \mathrm{length}(OO')}{\mathrm{area}(AB'CD')}
	\end{aligned}
\end{equation*}
Note that the smooth analogue of $\lambda$ for a Chebyshev net is provided in \eqref{eq: Chebyshev net ratio}. The above information for a discrete Chebyshev net is from \cite{schief_discrete_2007}. 

A \textbf{discrete orthogonal Chebyshev net} is a discrete Chebyshev net where $x(i_1+1, ~i_2+1) - x(i_1, ~i_2)$ is perpendicular to $x(i_1, ~i_2+1) - x(i_1+1, ~i_2)$ for all $i \in \mathbb{Z}^2$. It implies that the length of the four sides of the Chebyshev quadrilateral is equal, i.e. $AB = BC = CD = DA$ in Figure \ref{fig: ChebyShev quad}. Such net in fact has a cylindrical shape.

A discrete surface $X: \mathbb{Z}^2 \rightarrow \mathbb{R}^3$ is called a \textbf{discrete asymptotic net} if:
\begin{equation*}
	\begin{dcases}
		\triangle_1^2 x(i) \cdot N(i) = 0 \\
		\triangle_2^2 x(i) \cdot N(i) = 0 
	\end{dcases}, \mathrm{~~for~all~} i = (i_1, ~i_2) \in \mathbb{Z}^2
\end{equation*}
where
\begin{equation*}
	\begin{gathered}
		\triangle_1^2 x(i) = x(i_1-1, ~i_2) + x(i_1+1, ~i_2) - 2x(i_1, ~i_2) \\
		\triangle_2^2 x(i) = x(i_1, ~i_2-1) + x(i_1, ~i_2+1) - 2x(i_1, ~i_2) 
	\end{gathered}
\end{equation*}
From \eqref{eq: discrete mean curvature}:
\begin{equation*}
	\begin{aligned}
		\nabla \mathrm{area}(i) = ~ & \dfrac{1}{2} \left( \left(\dfrac{1}{\tan \beta_1}+\dfrac{1}{\tan \gamma_1} \right) (x(i_1, ~i_2) - x(i_1-1, ~i_2)) \right. \\
		& ~~ + \left(\dfrac{1}{\tan \beta_2}+\dfrac{1}{\tan \gamma_2} \right) (x(i_1, ~i_2) - x(i_1, ~i_2-1))  \\
		& ~~ + \left(\dfrac{1}{\tan \beta_3}+\dfrac{1}{\tan \gamma_3} \right) (x(i_1, ~i_2) - x(i_1+1, ~i_2)) \\
		& ~~ + \left. \left(\dfrac{1}{\tan \beta_4}+\dfrac{1}{\tan \gamma_4} \right) (x(i_1, ~i_2) - x(i_1, ~i_2+1)) \right)\\
	\end{aligned}
\end{equation*}
From direct calculation we could see that the condition for a discrete asymptotic net is that $x(i_1, ~i_2), ~x(i_1, ~i_2-1), ~x(i_1+1, ~i_2), ~x(i_1, ~i_2+1), ~x(i_1-1, ~i_2)$ are coplanar, then $N(i)$ is perpendicular to this plane and hence perpendicular to both $\triangle_1^2 x(i)$ and $\triangle_2^2 x(i)$. The above geometry also indicates that both $N(i_1, ~i_2)$ and $N(i_1+1, ~i_2)$ are perpendicular to $\triangle_1 x$; and both $N(i_1, ~i_2)$ and $N(i_1, ~i_2+1)$ are perpendicular to $\triangle_2 x$. The Lelieuvre normal field for an asymptotic net is defined as $N^\mathrm{L} = N (-\kappa_\mathrm{G})^{-1/4}$ (one option for discrete Gaussian curvature is the angular defect \eqref{eq: discrete Gaussian curvature}), and we could define the \textbf{discrete Lelieuvre normal field} $N^\mathrm{L}(i_1, ~i_2)$ to be a suitable scaling of $N(i_1, ~i_2)$ such that:
\begin{equation} \label{eq: discrete Lelieuvre normal}
	\begin{dcases}
		N^\mathrm{L} \times \triangle_1 N^\mathrm{L} = \triangle_1 x \\
		\triangle_2 N^\mathrm{L} \times N^\mathrm{L} = \triangle_2 x
	\end{dcases}
\end{equation}
From
\begin{equation} \label{eq: discrete Lelieuvre normal specfic}
	\begin{gathered}
		N^\mathrm{L} (i_1, ~i_2) \times N^\mathrm{L} (i_1+1, ~i_2) = x(i_1+1, ~i_2) - x(i_1, ~i_2) \\
		N^\mathrm{L} (i_1, ~i_2+1) \times N^\mathrm{L} (i_1+1, ~i_2+1) = x(i_1+1, ~i_2+1) - x(i_1, ~i_2+1) \\
		N^\mathrm{L} (i_1, ~i_2+1) \times N^\mathrm{L} (i_1, ~i_2) = x(i_1, ~i_2+1) - x(i_1, ~i_2) \\
		N^\mathrm{L} (i_1+1, ~i_2+1) \times N^\mathrm{L} (i_1+1, ~i_2) = x(i_1+1, ~i_2+1) - x(i_1+1, ~i_2) \\
	\end{gathered}
\end{equation}
sum the equations above we could see that:
\begin{equation*}
	(N^\mathrm{L} (i_1, ~i_2) + N^\mathrm{L} (i_1+1, ~i_2+1)) \times (N^\mathrm{L} (i_1+1, ~i_2) + N^\mathrm{L} (i_1, ~i_2+1)) = 0
\end{equation*}
which is equivalent to the \textbf{discrete Moutard Equation} for $N^\mathrm{L}$:
\begin{equation} \label{eq: discrete Moutard}
	\begin{gathered}
		\triangle_1 \triangle_2 N^\mathrm{L} = \dfrac{\lambda(i_1, ~i_2)}{2} (N^\mathrm{L} (i_1+1, ~i_2) + N^\mathrm{L} (i_1, ~i_2+1)), ~~\lambda: \mathbb{Z}^2 \rightarrow \mathbb{R} 
	\end{gathered}
\end{equation}
The reason for choosing $\lambda/2$ is for the consistency with its smooth analogue, \eqref{eq: Moutard}.

A discrete surface $X: \mathbb{Z}^2 \rightarrow \mathbb{R}^3$ is called a \textbf{discrete asymptotic Chebyshev net}, as known as a \textbf{K-hedron}/\textbf{discrete K-surface} in previous literatures if both \eqref{eq: discrete Chebyshev net} and the five points coplanar condition are satisfied. In Figure \ref{fig: ChebyShev quad}(a), set vector $O'B' = a$, vector $O'C = b$, vector $O'O = c$. Here $c$ is perpendicular to both $a$ and $b$. These three vectors determine the shape of a Chebyshev quadrilateral. Since vector $AB = a + b + c$, vector $BC = -a + b - c$, vector $CD = -a - b + c$, vector $DA = a - b - c$, let
\begin{equation*}
	\begin{gathered}
		N^\mathrm{L}(i_1, ~i_2) = \dfrac{a \times (b+c)}{\sqrt{(a \times b) \cdot c}} \\
		N^\mathrm{L}(i_1+1, ~i_2) = \dfrac{b \times (a+c)}{\sqrt{(a \times b) \cdot c}} \\
		N^\mathrm{L}(i_1+1, ~i_2+1) = \dfrac{a \times (b-c)}{\sqrt{(a \times b) \cdot c}} \\
		N^\mathrm{L}(i_1, ~i_2+1) = \dfrac{b \times (a-c)}{\sqrt{(a \times b) \cdot c}} \\
	\end{gathered}
\end{equation*} 
we could examine that $N^\mathrm{L}$ is a discrete Lelieuvre normal field, which agrees with \eqref{eq: discrete Lelieuvre normal specfic}:
\begin{equation*}
	\begin{gathered}
		N^\mathrm{L} (i_1, ~i_2) \times N^\mathrm{L} (i_1+1, ~i_2) = a+b+c \\
		N^\mathrm{L} (i_1, ~i_2+1) \times N^\mathrm{L} (i_1+1, ~i_2+1) = a+b-c \\
		N^\mathrm{L} (i_1, ~i_2+1) \times N^\mathrm{L} (i_1, ~i_2) = -a+b+c \\
		N^\mathrm{L} (i_1+1, ~i_2+1) \times N^\mathrm{L} (i_1+1, ~i_2) = -a+b-c \\
	\end{gathered}
\end{equation*}
It turns out that in the discrete Moutard Equation for a K-hedron, $\lambda(i_1, ~i_2) = -4$,
\begin{equation} \label{eq: discrete Moutard asymptotic Chebyshev}
	\begin{gathered}
		\triangle_1 \triangle_2 N^\mathrm{L} = -2 (N^\mathrm{L} (i_1+1, ~i_2) + N^\mathrm{L} (i_1, ~i_2+1)) 
	\end{gathered}
\end{equation}
In terms of the discrete normal field:
\begin{equation*}
	\begin{gathered}
		N(i_1, ~i_2) = \dfrac{a \times (b+c)}{\|a \times (b+c)\|}, ~~N(i_1+1, ~i_2) = \dfrac{b \times (a+c)}{\|b \times (a+c)\|} \\
		N(i_1+1, ~i_2+1) = \dfrac{a \times (b-c)}{\|a \times (b-c)\|}, ~~N(i_1, ~i_2+1) = \dfrac{b \times (a-c)}{\|b \times (a-c)\|} \\
	\end{gathered}
\end{equation*}
From the geometry illustrated in Figure \ref{fig: ChebyShev quad}(a):
\begin{equation*}
	\|a \times (b+c)\| = \|a \times (b-c)\|, ~~\|b \times (a+c)\| = \|b \times (a-c)\|
\end{equation*}
\begin{equation*}
	\begin{gathered}
		\triangle_1 \triangle_2 N = 2 \left(\dfrac{1}{\|a \times (b+c)\|} + \dfrac{1}{\|b \times (a+c)\|}\right) a \times b \\
		\triangle_1 \triangle_2 N = - \left(\dfrac{\|b \times (a+c)\|}{\|a \times (b+c)\|} + 1 \right) (N (i_1+1, ~i_2) + N (i_1, ~i_2+1))
	\end{gathered}
\end{equation*}
\begin{equation*}
	\begin{gathered}
		N(i_1+1, ~i_2+1) - N(i_1, ~i_2) = \dfrac{2 c \times a}{\|a \times (b+c)\|} \\
		N(i_1, ~i_2+1) - N(i_1+1, ~i_2) = \dfrac{2 c \times b}{\|b \times (a+c)\|}
	\end{gathered}
\end{equation*}
\begin{equation*}
	\begin{gathered}
		\triangle_1 \triangle_2 N \cdot (N(i_1+1, ~i_2+1) - N(i_1, ~i_2)) = 0 \\
		\triangle_1 \triangle_2 N \cdot (N(i_1, ~i_2+1) - N(i_1+1, ~i_2)) = 0
	\end{gathered}
\end{equation*}

The above derivation leads to the following proposition, which is parallel to its smooth analogue:
\begin{prop}
	The Gauss map of a discrete K-surface is a discrete Chebyshev net. A discrete K-surface is the only discrete asymptotic net with a discrete Chebyshev Gauss map.
\end{prop}

\begin{figure}[t]
	\noindent \begin{centering}
		\includegraphics[width=1\linewidth]{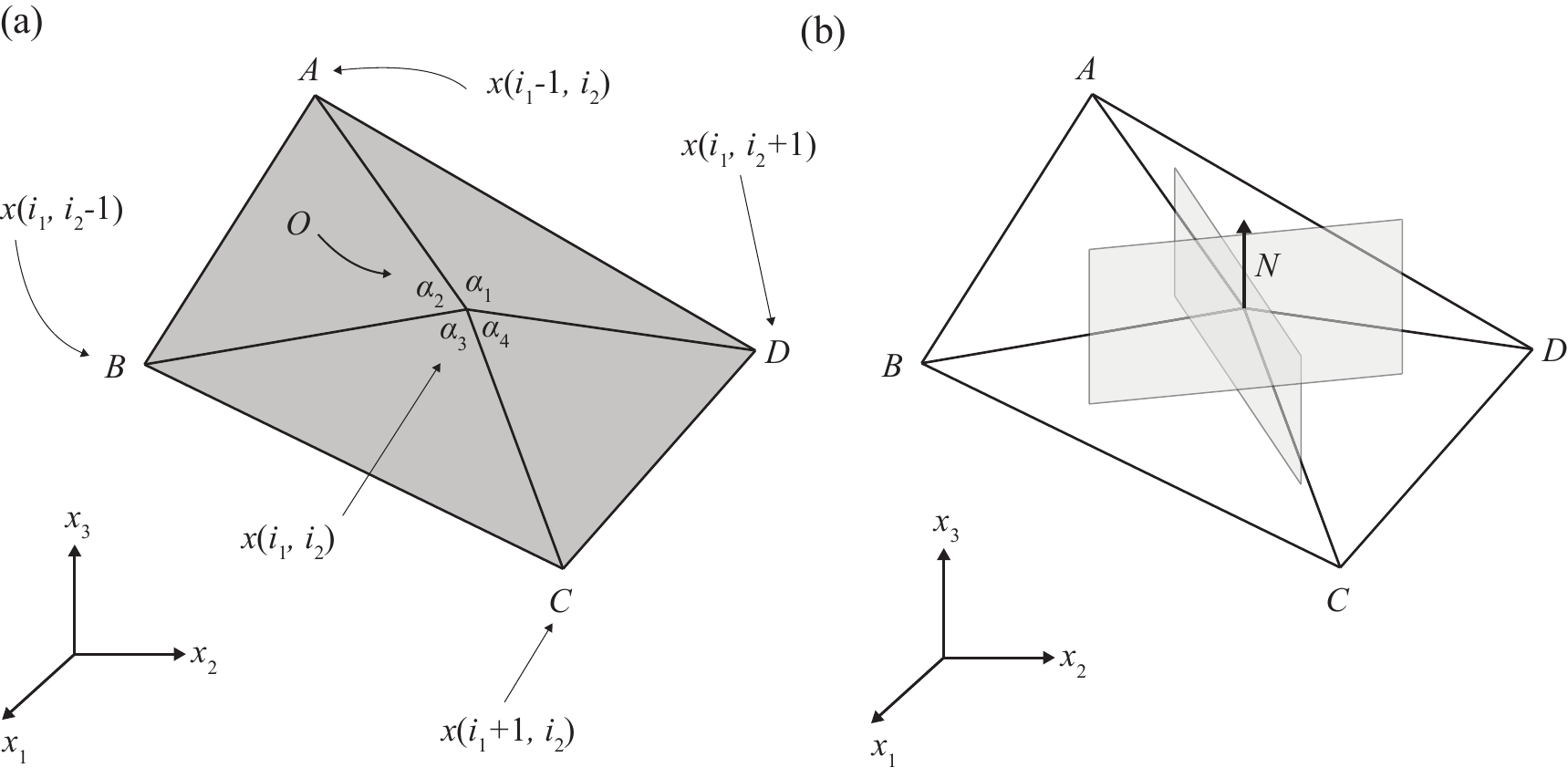}
		\par\end{centering}
	
	\caption{\label{fig: discrete geodesic}A figure illustrating the geometry of a discrete geodesic net, where opposite sector angles at each interior vertex are equal. One important geometrical feature is the coordinate curve normals are identical to the surface normal at every interior vertex.}
\end{figure}

A discrete surface $X: \mathbb{Z}^2 \rightarrow \mathbb{R}^3$ is called a \textbf{discrete geodesic net} if opposite sector angles at every vertex are equal:
\begin{equation*}
	\begin{dcases}
		S_1(i_1, ~i_2+1) = S_3(i_1+1, ~i_2) \\
		S_2(i_1, ~i_2) = S_4(i_1+1, ~i_2+1)
	\end{dcases} \textrm{~~~for~all~} i \in \mathbb{Z}^2
\end{equation*} 
which means $\alpha_1 = \alpha_3, ~\alpha_2 = \alpha_4$ in Figure \ref{fig: discrete geodesic}.

A discrete surface $X: \mathbb{Z}^2 \rightarrow \mathbb{R}^3$ is called a \textbf{discrete orthogonal geodesic net} if all four sector angles at every vertex are equal:
\begin{equation*}
	S_1(i_1, ~i_2+1) = S_2(i_1, ~i_2) = S_3(i_1+1, ~i_2) = S_4(i_1+1, ~i_2+1)
	\textrm{~~~for~all~} i \in \mathbb{Z}^2
\end{equation*}
which means $\alpha_1 = \alpha_2 = \alpha_3 = \alpha_4$ in Figure \ref{fig: discrete geodesic}.

A geodesic curve is `as straight as possible' and has `no lateral acceleration' on a surface. By requiring $\alpha_1 + \alpha_2 = \alpha_3 + \alpha_4$ and $\alpha_2 + \alpha_3 = \alpha_4 + \alpha_1$, the polylines $AOC$ and $BOD$ will divide the angular defect $\kappa_\mathrm{G} = 2 \pi - \alpha_1 - \alpha_2 - \alpha_3 - \alpha_4$ equally. This leads to the angle condition for a discrete geodesic net mentioned above.

Next we will calculate the normal vector defined on $x(i)$, let $a, ~b, ~c, ~d$ be the direction vector of $OA, ~OB, ~OC, ~OD$:
\begin{equation*}
	\begin{gathered}
		a = \dfrac{x(i_1, ~i_2) - x(i_1-1, ~i_2)}{\|x(i_1, ~i_2) - x(i_1-1, ~i_2)\|}, ~~b = \dfrac{x(i_1, ~i_2) - x(i_1, ~i_2-1)}{\|x(i_1, ~i_2) - x(i_1, ~i_2-1)\|} \\
		c = \dfrac{x(i_1, ~i_2) - x(i_1+1, ~i_2)}{\|x(i_1, ~i_2) - x(i_1+1, ~i_2)\|}, ~~d = \dfrac{x(i_1, ~i_2) - x(i_1, ~i_2+1)}{\|x(i_1, ~i_2) - x(i_1, ~i_2+1)\|} 
	\end{gathered}
\end{equation*}

Along the $i_1$ direction, the discrete Frenet-Serret frame is 
\begin{equation*}
	\boldsymbol{x}_\mathrm{t}^1 = -\dfrac{a-c}{\|a-c\|}, ~~ \boldsymbol{x}_\mathrm{n}^1 = -\dfrac{a+c}{\|a+c\|}, ~~ \boldsymbol{x}_\mathrm{b}^1 = \boldsymbol{x}_\mathrm{t}^1 \times \boldsymbol{x}_\mathrm{n}^1
\end{equation*}
Along the $i_2$ direction, the discrete Frenet-Serret frame is 
\begin{equation*}
	\boldsymbol{x}_\mathrm{t}^2 = -\dfrac{b-d}{\|b-d\|}, ~~ \boldsymbol{x}_\mathrm{n}^2 = -\dfrac{b+d}{\|b+d\|}, ~~ \boldsymbol{x}_\mathrm{b}^2 = \boldsymbol{x}_\mathrm{t}^2 \times \boldsymbol{x}_\mathrm{n}^2
\end{equation*}
The normal vector is:
\begin{equation*}
	N(i) = \dfrac{\boldsymbol{x}_\mathrm{t}^1 \times \boldsymbol{x}_\mathrm{t}^2}{\|\boldsymbol{x}_\mathrm{t}^1 \times \boldsymbol{x}_\mathrm{t}^2\|}
\end{equation*}
Since opposite sector angles are equal, $\boldsymbol{x}_\mathrm{n}^1 \cdot \boldsymbol{x}_\mathrm{t}^2 = 0$ and $\boldsymbol{x}_\mathrm{n}^2 \cdot \boldsymbol{x}_\mathrm{t}^1 = 0$. We could see that either $\boldsymbol{x}_\mathrm{n}^1$ or $\boldsymbol{x}_\mathrm{n}^2$ is perpendicular to both $\boldsymbol{x}_\mathrm{t}^1$ and $\boldsymbol{x}_\mathrm{t}^2$, hence
\begin{equation*}
	N(i) = \boldsymbol{x}_\mathrm{n}^1 = \boldsymbol{x}_\mathrm{n}^2
\end{equation*}
The above equality of normal vectors further shows that polylines $AOC$ and $BOD$ are discrete analogue of geodesic curves on a surface. 

For a discrete orthogonal geodesic net, we further have $\boldsymbol{x}_\mathrm{t}^1$ perpendicular to $\boldsymbol{x}_\mathrm{t}^2$, which geometrically explains that the coordinate curves are perpendicular at every interior vertex. The information of discrete geodesic net and discrete orthogonal geodesic net is from \cite{rabinovich_discrete_2018}.

A discrete surface $X: \mathbb{Z}^2 \rightarrow \mathbb{R}^3$ is called a \textbf{discrete conjugate net} if all its elementary quadrilaterals formed by $x(i_1, ~i_2), ~x(i_1+1, ~i_2), ~x(i_1+1, ~i_2+1), ~x(i_1, ~i_2+1)$ are planar for all $i \in \mathbb{Z}$. Here the normal vector $N(i)$ at each vertex is associated with the normal vector of the above planar quadrilateral and hence $\triangle_1\triangle_2x$ is perpendicular to $N$. Using the Christoffel symbol, the planarity condition for a discrete conjugate net is equivalent to:
\begin{equation} \label{eq: discrete conjugate net}
	\triangle_1\triangle_2 x = \Gamma_{12}^1 \triangle_1 x + \Gamma_{12}^2 \triangle_2 x, ~~ \Gamma_{12}^1, ~\Gamma_{12}^2: \mathbb{Z}^2 \rightarrow \mathbb{R}
\end{equation}
Here $\Gamma_{12}^1(i)$, $\Gamma_{12}^2(i)$ can be directly calculated as the coefficients of the above linear combination for the given discrete net $x$. The smooth analogue of $\Gamma_{12}^1(i)$, $\Gamma_{12}^2(i)$ is the corresponding Christoffel symbol $\Gamma_{12}^1, ~\Gamma_{12}^2$ for the corresponding conjugate net.

\begin{figure}[tbph]
	\noindent \begin{centering}
		\includegraphics[width=0.7\linewidth]{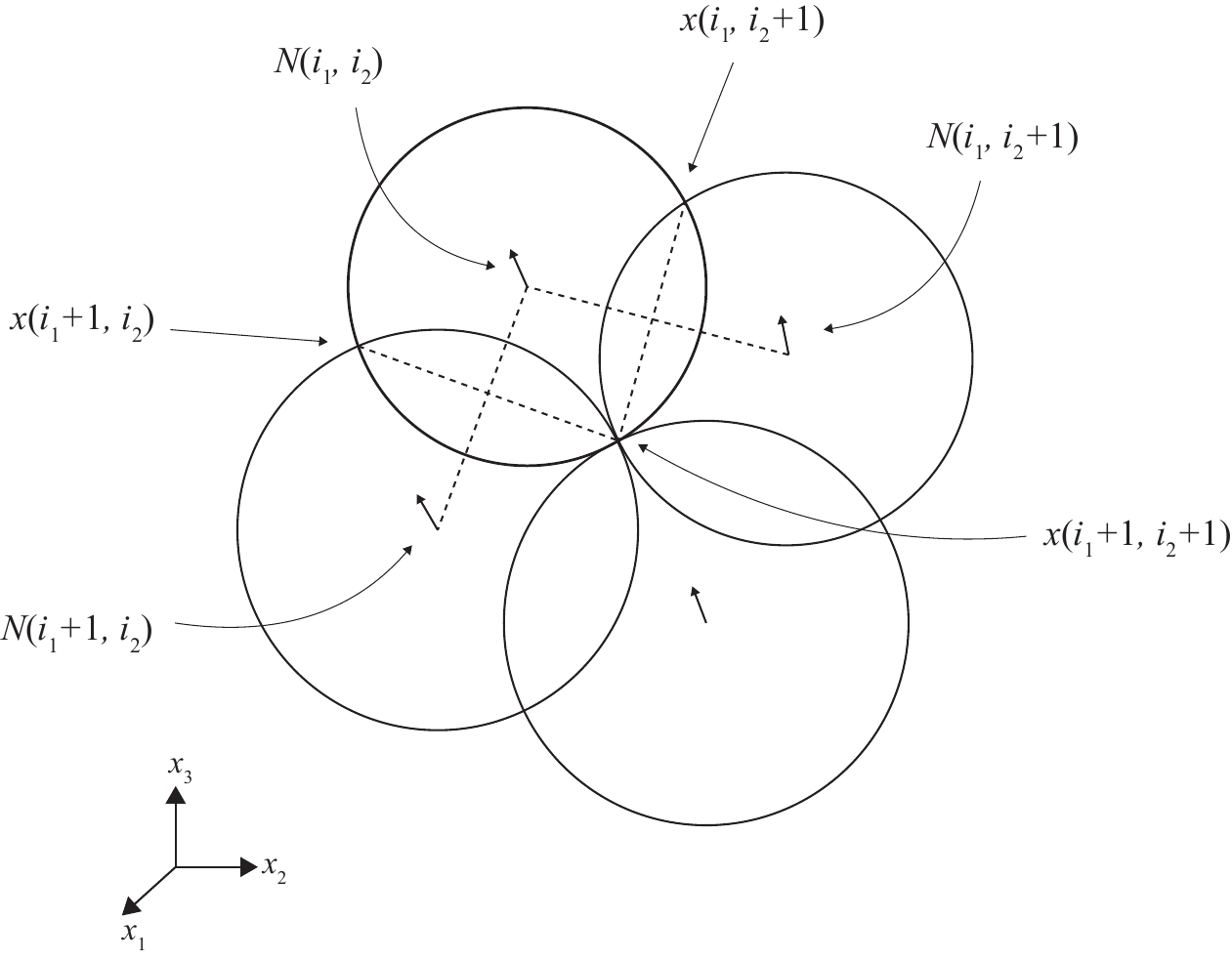}
		\par\end{centering}
	
	\caption{\label{fig: circular net}Illustration of the geometry of a circular net, where $x(i_1, ~i_2), ~x(i_1+1, ~i_2), ~x(i_1+1, ~i_2+1), ~x(i_1, ~i_2+1)$ are concircular for all $i \in \mathbb{Z}^2$. Here we draw a special case where pairs of circles located diagonally from each other are tangent. It happens, for example, on a discrete isothermic net whose cross ratio is limited to -1.}
	
	\noindent \begin{centering}
		\includegraphics[width=1\linewidth]{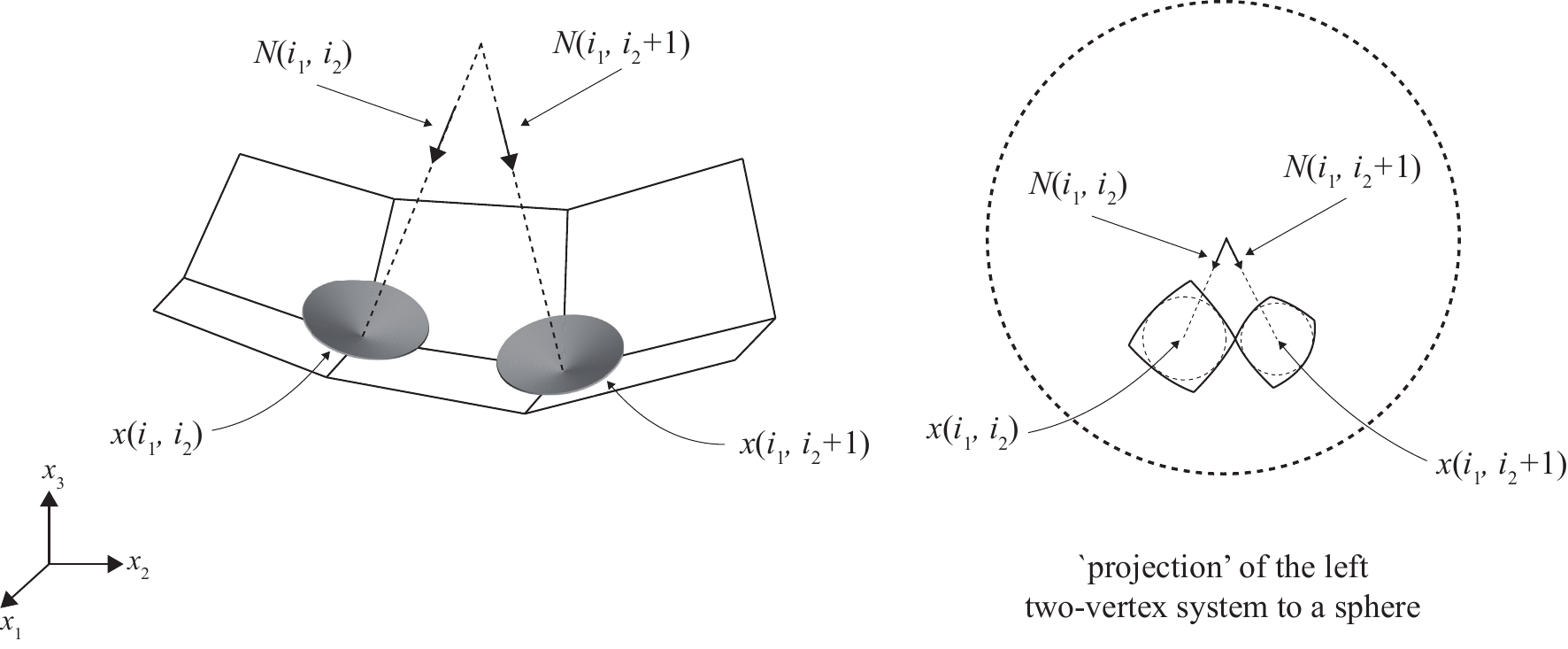}
		\par\end{centering}
	
	\caption{\label{fig: conical net}Illustration of the geometry of a conical net, where the four planar quadrilaterals incident to a vertex are tangent to a common cone whose apex is the vertex. For every pair of adjacent cones, there exists a sphere that touches both at the apexes. The centre of this sphere is the intersection point of the cones' axes.}
\end{figure}

A discrete conjugate net $X: \mathbb{Z}^2 \rightarrow \mathbb{R}^3$ is called a \textbf{circular net} if all its elementary quadrilaterals formed by $x(i_1, ~i_2), ~x(i_1+1, ~i_2), ~x(i_1+1, ~i_2+1), ~x(i_1, ~i_2+1)$ have circumscribed circles, in other words, they are concircular. It is one of the common discrete analogue of the curvature line net. From Figure \ref{fig: circular net}, $\triangle_2 x (i_1+1, ~i_2)$ is perpendicular to both $N(i_1, ~i_2)$ and $N(i_1+1, ~i_2)$, hence perpendicular to $\triangle_1 N(i_1, ~i_2)$; $\triangle_1 x (i_1, ~i_2+1)$ is perpendicular to both $N(i_1, ~i_2)$ and $N(i_1, ~i_2+1)$, hence perpendicular to $\triangle_2 N(i_1, ~i_2)$. This relation forms the discrete analogue of a curvature line net, \eqref{eq: derivative curvature line}.

A discrete conjugate net $X: \mathbb{Z}^2 \rightarrow \mathbb{R}^3$ is called a \textbf{conical net} if the four planar quadrilaterals incident to a vertex are tangent to a common cone whose apex is the vertex. The discrete normal vector $N(i)$ assigned to each vertex $i$ is along the axis of the cone, see Figure \ref{fig: conical net}(a). A conical net is another common discrete analogue of the curvature line net. We can interpret it by drawing the corresponding spherical 4-bar linkages of the degree-4 two-vertex system on a sphere. The cones become inscribed circles of the two spherical quadrilaterals, whose axes intersect at the centre of the sphere. We could see that $\triangle_2 x (i_1, ~i_2)$ is parallel to $\triangle_2 N(i_1, ~i_2)$, similarly, $\triangle_1 x (i_1, ~i_2)$ is parallel to $\triangle_1 N(i_1, ~i_2)$. This relation forms the discrete analogue of a curvature line net, \eqref{eq: derivative curvature line}.

\begin{prop} \label{prop: conical net}
	Properties of a conical net:
	\begin{enumerate} [label={[\arabic*]}]
		\item \citep{wang_angle_2007} The sum of opposite sector angles of each vertex are equal. 
		\item \citep[Section 3.4]{bobenko_discrete_2008} A discrete conjugate net $X$ is a conical net if and only if the Gauss map is a circular net. A discrete conjugate net $X$ is a circular net if and only if the Gauss map is a conical net.
	\end{enumerate}	
\end{prop}

Regarding [1], intuitively, as shown in Figure \ref{fig: conical net}, the sum of the length of opposite spherical arcs are equal if the spherical quadrilateral admits an inscribed circle. Regarding [2], at every vertex, the angles between all four normal vectors and the axis of the cone are equal, therefore the tips of these normal vectors are concircular, and the centre of this circle is on the axis of the cone.

\section{Initial condition for discrete nets} \label{section: difference equation}

The various discrete nets introduced in Section \ref{section: discrete net} are solutions of parametric partial difference equations. In this section we will focus on the well-posedness for such a discrete system. Similarly, we hope a given initial condition leads to a unique solution, which smoothly relies on the initial value and parameter.

In parallel with Section \ref{section: differential equation}, we will focus on a first-order partial difference system for $x(i)$:
\begin{equation*}
	\triangle x = f(x; ~b) ~~\Leftrightarrow~~ \triangle_{k} x_j = f_{jk}(x; ~b)
\end{equation*}
where $i \in \mathbb{Z}^m$; $x \in \mathbb{R}^n ~(m, ~n \in \mathbb{Z}_+)$; $f \in \mathbb{R}^{n \times m}$ is a matrix of smooth functions, $b \in \mathbb{R}^p ~(p \in \mathbb{Z}_+)$ are the $p$ parameters for the system. Similar to Section \ref{section: differential equation}, we further require $f$ and all the partial derivatives of $f$ are bounded and possess a global Lipschitz constant. Consequently no blow-ups (value goes to infinity) are possible and hence the well-posedness can be continued to the boundary of $I^m$. If there are higher order partial differences, we could try transferring the system to first-order by adding the number of variables. For example, when $\triangle_1\triangle_2 x = x, ~x \in \mathbb{R}$, we could set $y(i) = \triangle_1 x$ and $z(i) = \triangle_2 x$, so that $(x, ~y, ~z)$ forms an equivalent first-order system with compatibility condition $\triangle_2 y = \triangle_1 z$.

%

The definitions for the evolution direction, stationary direction, initial value and well-posedness for a first-order partial difference system are verbatim repetition for those defined for a partial differential system in Section \ref{section: differential equation}. For each discrete net introduced in Section \ref{section: discrete net}, we will introduce the construction method leading to a unique configuration from the initial condition. It could be directly examined that solution smoothly relies on the initial value and parameter. 

\subsection*{Discrete Chebyshev net}

\begin{description}
	\item [Initial value] Two discrete coordinate curves $x(i_1, ~0)$ and $x(0, ~i_2)$ intersecting at $x(0, ~0)$. 
	\item [Parameter] The ratio $\lambda$ in \eqref{eq: partial difference Chebyshev} for all quadrilaterals $(i_1, ~i_2)$.
	\item [Step a] In the quadrant $\mathbb{Z}_+^2$, recursively calculate $x(i_1+1, ~i_2+1)$ from $x(i_1, ~i_2)$, $x(i_1+1, ~i_2)$, $x(i_1, ~i_2+1)$ using \eqref{eq: partial difference Chebyshev}. Geometrically in Figure \ref{fig: ChebyShev quad}(a), when points $A, ~B, ~C$ are fixed, the shape of the Chebyshev quadrilateral can be controlled by the length of $BB'$, or equivalently $\lambda$. 
	\item [Step b] Use the same method described in Step a to calculate the other three quadrants to obtain the entire mesh.
	\item [Regularity condition] Every step returns a non-degenerated and bounded result.
\end{description}	

\subsection*{Discrete orthogonal Chebyshev net}

\begin{description}
	\item [Initial value] Two discrete coordinate curves $x(i_1, ~0)$ and $x(0, ~i_2)$ intersecting at $x(0, ~0)$ where $\|\triangle_1 x (i_1, ~0) \| = \|\triangle_1 x (i_1+1, ~0) \| = \|\triangle_2 x (0, ~i_2)\| = \|\triangle_2 x (0, ~i_2+1) \|$ for all $i_1, ~i_2 \in \mathbb{Z}$. 
	\item [Parameter] The ratio $\lambda$ in \eqref{eq: partial difference Chebyshev} for all quadrilaterals $(i_1, ~i_2)$.
	\item [Step a] In the quadrant $\mathbb{Z}_+^2$, recursively calculate $x(i_1+1, ~i_2+1)$ from $x(i_1, ~i_2)$, $x(i_1+1, ~i_2)$, $x(i_1, ~i_2+1)$ using \eqref{eq: partial difference Chebyshev}.  
	\item [Step b] Use the same method described in Step a to calculate the other three quadrants to obtain the entire mesh.
	\item [Regularity condition] Every step returns a non-degenerated and bounded result.
\end{description}	

\subsection*{Discrete asymptotic net}

The first construction is:
\begin{description}
	\item [Initial value 1] Two discrete coordinate curves $x(i_1, ~0)$ and $x(0, ~i_2)$ intersecting at $x(0, ~0)$. The five points $x(0, ~0), ~x(0, ~-1)$, $x(1, ~0), ~x(0, ~1), ~x(-1, ~0)$ are coplanar. The three points $x(i_1-1, ~0), ~x(i_1, ~0), ~x(i_1+1, ~0)$; $x(0, ~i_2-1), ~x(0, ~i_2), ~x(0, ~i_2+1)$ are not collinear. 
	\item [Parameter 1] Cross ratio $q$ for all quadrilaterals $(i_1, ~i_2)$, will be defined below.
	\item [Step 1a] In the quadrant $\mathbb{Z}_+^2$, we say $P(i_1, ~i_2)$ is the plane incident to $x(i_1, ~i_2), ~x(i_1, ~i_2-1), ~x(i_1+1, ~i_2), ~x(i_1, ~i_2+1), ~x(i_1-1, ~i_2)$. We can calculate $P(1, ~0)$ and $P(0, ~1)$ from the initial value.
	\item [Step 1b] $x(1, ~1)$ can be chosen from the intersection of two planes $P(1, ~0)$ and $P(0, ~1)$, which passes through $x(0, ~0)$. Usually we use the \textit{cross-ratio} $q$ defined on each quadrilateral $(i_1, ~i_2)$ to control the position of $x(i_1+1, ~i_2+1)$:
	\begin{equation} \label{eq: cross ratio}
		q(i_1, ~i_2) = \dfrac{\|x(i_1+1, ~i_2+1) - x(i_1, ~i_2)\| \|x(i_1, ~i_2+1) - x(i_1, ~i_2)\|}{\|x(i_1+1, ~i_2+1) - x(i_1+1, ~i_2)\| \|x(i_1, ~i_2+1) - x(i_1+1, ~i_2)\|}
	\end{equation}
	\item [Step 1c] Calculate $P(1, ~1)$ from $x(1, ~0), ~x(0, ~1), ~x(1, ~1)$.
	\item [Step 1d] Recursively do Steps 1a, 1b, 1c to obtain $x$ over the quadrant $\mathbb{Z}_+^2$.
	\item [Step 1e] Use the same method described in Step 1d to calculate the other three quadrants to obtain the entire mesh.
	\item [Regularity condition] Every step returns a non-degenerated and bounded result.
\end{description}	

The second construction is from the discrete Lelieuvre normal field $N^\mathrm{L}$.
\begin{description}
	\item [Initial value 2] $N^\mathrm{L}$ along the two discrete coordinate curves $N^\mathrm{L}(i_1, ~0)$ and $N^\mathrm{L}(0, ~i_2)$. The position of $x(0, ~0)$. 
	\item [Parameter 2] The ratio $\lambda$ in the discrete Moutard Equation \eqref{eq: discrete Moutard} on all the vertices $(i_1, ~i_2)$.
	\item [Step 2a] In the quadrant $\mathbb{Z}_+^2$, recursively calculate $N(i_1+1, ~i_2+1)$ from $N(i_1, ~i_2)$, $N(i_1+1, ~i_2)$, $N(i_1, ~i_2+1)$ using \eqref{eq: discrete Moutard}.
	\item [Step 2b] In the quadrant $\mathbb{Z}_+^2$, use the discrete Lelieuvre normal field, \eqref{eq: discrete Lelieuvre normal}, to calculate all the $\triangle_1 x$ and $\triangle_2 x$, further obtain all the position $x$ based on the initial position $x(0, ~0)$. 
	\item [Step 2c] Use the same method described in Step 2a and Step 2b to calculate the other three quadrants to obtain the entire mesh.
	\item [Regularity condition] Every step returns a non-degenerated and bounded result.
\end{description}	

\subsection*{Discrete asymptotic Chebyshev net}

\begin{description}
	\item [Initial value] $N^\mathrm{L}$ along the two discrete coordinate curves $N^\mathrm{L}(i_1, ~0)$ and $N^\mathrm{L}(0, ~i_2)$. The position of $x(0, ~0)$. 
	\item [Step a] In the quadrant $\mathbb{Z}_+^2$, recursively calculate $N(i_1+1, ~i_2+1)$ from $N(i_1, ~i_2)$, $N(i_1+1, ~i_2)$, $N(i_1, ~i_2+1)$ using \eqref{eq: discrete Moutard asymptotic Chebyshev}.
	\item [Step b] In the quadrant $\mathbb{Z}_+^2$, use the discrete Lelieuvre normal field, \eqref{eq: discrete Lelieuvre normal}, to calculate all the $\triangle_1 x$ and $\triangle_2 x$, further obtain all the position $x$ based on the initial position $x(0, ~0)$. 
	\item [Step c] Use the same method described in Step a and Step b to calculate the other three quadrants to obtain the entire mesh.
	\item [Regularity condition] Every step returns a non-degenerated and bounded result.
\end{description}

\subsection*{Discrete geodesic/orthogonal geodesic net}

The constraint for a discrete geodesic/orthogonal geodesic net is not first-order. From our examination, it is not possible to construct a discrete geodesic/orthogonal geodesic net in a point-by-point procedure as the previous examples. \cite{rabinovich_discrete_2018} generates a discrete geodesic/orthogonal geodesic net from introducing an (global) optimization problem, where the variables are vertex coordinates of the entire mesh, subject to the sector angle constraints.

\subsection*{Discrete conjugate net}

\begin{description}
	\item [Initial value] Two discrete coordinate curves $x(i_1, ~0)$ and $x(0, ~i_2)$ intersecting at $x(0, ~0)$.  
	\item [Parameter] Discrete Christoffel symbol $\Gamma_{12}^1, ~\Gamma_{12}^2$ in \eqref{eq: discrete conjugate net} for all quadrilaterals $(i_1, ~i_2)$.
	\item [Step a] In the quadrant $\mathbb{Z}_+^2$, recursively calculate $x(i_1+1, ~i_2+1)$ from $x(i_1, ~i_2)$, $x(i_1+1, ~i_2)$, $x(i_1, ~i_2+1)$ using \eqref{eq: discrete conjugate net}.
	\item [Step b] Use the same method described in Step a to calculate the other three quadrants to obtain the entire mesh.
	\item [Regularity condition] Every step returns a non-degenerated and bounded result.
\end{description}		

\subsection*{Circular net}

\begin{description}
	\item [Initial value] Two discrete coordinate curves $x(i_1, ~0)$ and $x(0, ~i_2)$ intersecting at $x(0, ~0)$.  
	\item [Parameter] Cross ratio $q$ in \eqref{eq: cross ratio} for all quadrilaterals $(i_1, ~i_2)$ to control the position of $x(i_1+1, ~i_2+1)$.
	\item [Step a] In the quadrant $\mathbb{Z}_+^2$, recursively calculate $x(i_1+1, ~i_2+1)$ on the circle determined by $x(i_1, ~i_2)$, $x(i_1+1, ~i_2)$, $x(i_1, ~i_2+1)$ using \eqref{eq: cross ratio}.
	\item [Step b] Use the same method described in Step a to calculate the other three quadrants to obtain the entire mesh.
	\item [Regularity condition] Every step returns a non-degenerated and bounded result.
\end{description}	

\subsection*{Conical net}

From Proposition \ref{prop: conical net}, a conical net can be uniquely constructed from a circular Gauss map.

\begin{description}
	\item [Initial value] The normal vectors $N(i_1, ~0)$ and $N(0, ~i_2)$ on the two coordinate axes and the position of the planes where the elementary quadrilaterals on the coordinate axes $(i_1, ~0)$ and $(0, ~i_2)$ locate.   
	\item [Parameter] Cross ratio $q$ for all the spherical quadrilaterals of the Gauss map. 
	\item [Step a] In the quadrant $\mathbb{Z}_+^2$, recursively calculate $N(i_1+1, ~i_2+1)$ on the circle determined by $N(i_1, ~i_2)$, $N(i_1+1, ~i_2)$, $N(i_1, ~i_2+1)$ using the cross ratio.
	\item [Step b] Use the same method described in Step a to calculate the other three quadrants to obtain the entire Gauss map.
	\item [Step c] In the quadrant $\mathbb{Z}_+^2$,  recursively calculate the plane where the elementary quadrilateral $(i_1+1, ~i_2+1)$ locate from the position of planes $(i_1, ~i_2)$, $(i_1+1, ~i_2)$, $(i_1, ~i_2+1)$ -- the plane $(i_1+1, ~i_2+1)$ is normal to $N(i_1+1, ~i_2+1)$ and passes through the common intersection $x(i_1+1, ~i_2+1)$ determined by the position of planes $(i_1, ~i_2)$, $(i_1+1, ~i_2)$, $(i_1, ~i_2+1)$.
	\item [Step d] Use the same method described in Step c to calculate the other three quadrants to obtain the entire mesh.
	\item [Regularity condition] Every step returns a non-degenerated and bounded result.
\end{description}		 	 

\section{Convergence} \label{section: ddg convergence}

We have introduced various initial conditions for well-posed solutions of smooth nets in Section \ref{section: differential equation} and discrete nets in Section \ref{section: difference equation}. One natural question is, if the initial conditions and parameters for a discrete net converge to that for a smooth net, will the solution converge? Further, would the geometrical quantities -- such as the distance/area, normal vector field, mean curvature and Gaussian curvature -- also converge? There is still plenty of unexplored space for this question, and some results might be counter-intuitive.  

Let us start with a series of discrete curves $X^\epsilon: I \cap (\epsilon\mathbb{Z}) \rightarrow \mathbb{R}^3, ~\epsilon > 0$ where the number of discrete points in the interval $I$ is controlled by $\epsilon$. When setting $\epsilon = 0$,  $X^0: I \rightarrow \mathbb{R}^3$ is a smooth curve, and all the ratio between discrete operators becomes the corresponding differential operators. 

\begin{defn} \label{defn: curve convergence}
	(Curve convergence in distance) A series of discrete curves $X^\epsilon: I \cap (\epsilon\mathbb{Z}) \rightarrow x^\epsilon \in \mathbb{R}^3$ (uniformly) converges to a smooth curve $X^0: I \rightarrow x^0 \in \mathbb{R}^3$ if for any $\mathrm{error} > 0$, there exists a grid size $\mathrm{gsize} > 0$ such that for all $0 < \epsilon < \mathrm{gsize}$:
	\begin{equation*}
		\sup_{\epsilon i \in I \cap (\epsilon\mathbb{Z})} \| x^\epsilon (\epsilon i)  - x^0(\epsilon i) \| < \mathrm{error}
	\end{equation*}
	That is to say we expect the error uniformly goes to zero as the grid size goes to zero. 
	
	The \textit{convergence rate} of $X^\epsilon$ is $O(f(\epsilon))$ if $\mathrm{error} = \mathrm{Const} \cdot f(\epsilon)$, here $\mathrm{Const}$ is a constant irrelevant to $\epsilon$ when $\mathrm{error} \rightarrow 0$.  
\end{defn}

This `uniform convergence' definition is used in \cite[Section 5.1]{bobenko_discrete_2008}. Regarding the convergence rate, for example, $O(\epsilon)$ means we need to halve the grid size to halve the error in distance when the error is near zero; $O(\epsilon^2)$ means we have a better convergence rate, so that halve the grid size will quarter the error in distance when the error is near zero. For short, we will omit `uniform convergence in distance' and simply call it `convergence in distance'. 

Similarly we provide the definition for surface convergence below:
\begin{defn} \label{defn: surface convergence}
	(Surface convergence in distance) A series of discrete surfaces $X^\epsilon: \epsilon i = (\epsilon i_1, ~\epsilon i_2) \in I^2 \cap (\epsilon \mathbb{Z})^2 \rightarrow x^\epsilon \in \mathbb{R}^3$ uniformly converges to a smooth surface $X^0: u = (u_1, ~u_2) \in I^2 \rightarrow x^0 \in \mathbb{R}^3$ if for any $\mathrm{error} > 0$, there exists a grid size $\mathrm{gsize} > 0$ such that for all $0 < \epsilon < \mathrm{gsize}$:
	\begin{equation*}
		\sup_{\epsilon i \in I^2 \cap (\epsilon \mathbb{Z})^2} \| x^\epsilon \left(\epsilon i\right) - x^0(\epsilon i) \| < \mathrm{error}
	\end{equation*}
	Similarly we expect the error uniformly goes to zero as the grid size goes to zero.
\end{defn}

\begin{thm} \label{thm: first-order convergence}
	\citep{matthes_discrete_2004, bobenko_discrete_2008} The solution of a series of first-order hyperbolic partial difference system (refined over $\epsilon$) converges to the solution of the first-order hyperbolic partial differential system over $I^2$ (globally) upon the convergence to initial condition and parameter. 
\end{thm}

\begin{figure}[t]
	\noindent \begin{centering}
		\includegraphics[width=1\linewidth]{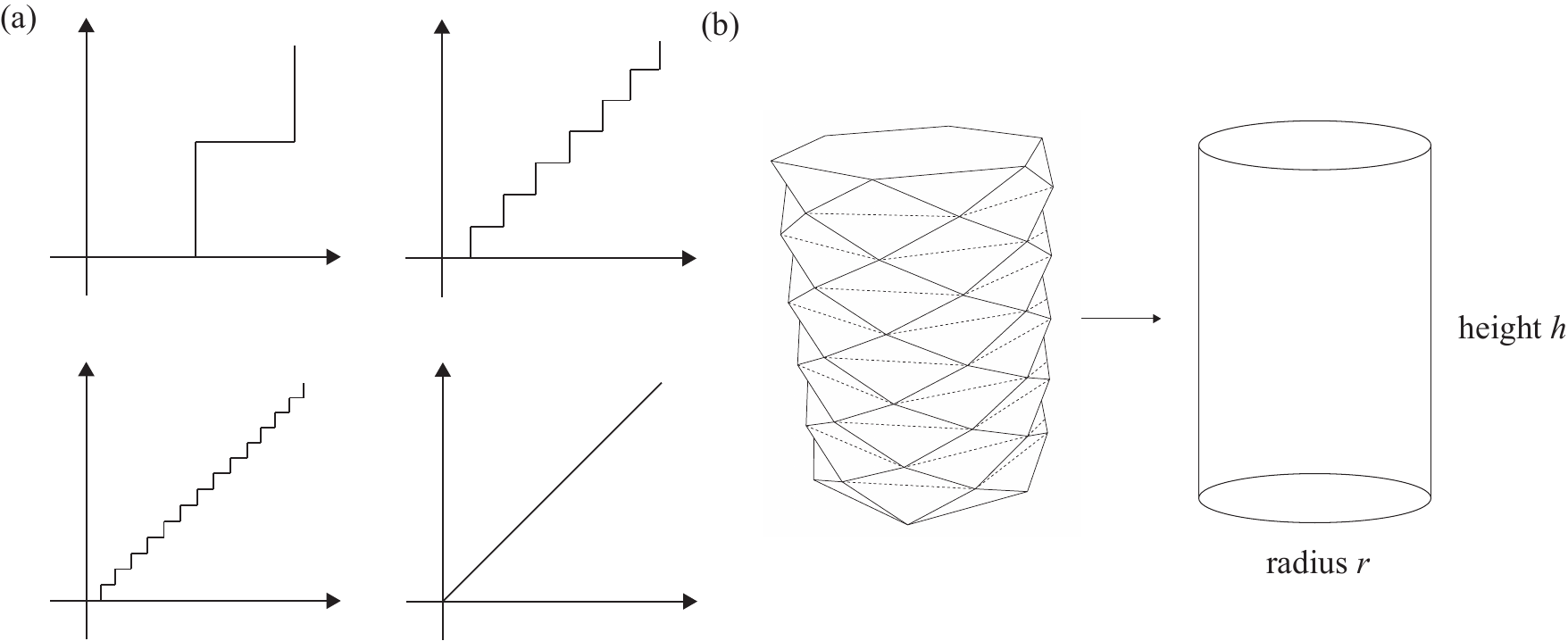}
		\par\end{centering}
	
	\caption{\label{fig: staircase ddg}(a) the `Staircase paradox'. Fractal curves (for example the Koch curve) have similar non-convergence in metric- and curvature-related properties. (b) the `Schwarz Lantern'. In a cylinder of radius $r~(r > 0)$ and height $h~(h > 0)$ we inscribe a polyhedron as follows. Cut the cylinder into $m~(m \in \mathbb{Z}_+)$ equal cylinders each of height $h/m$ by means of horizontal planes. Break each of the $m+1$ circles of intersection (including the upper and lower bases of the original cylinder) into $n ~(n \in \mathbb{Z}_+, ~n \ge 2)$ equal parts so that the points of division on each circle lie beneath the midpoints of the points of division of the circle immediately above. We now take a pair of division points of each circle and the point lying directly above or below the midpoint of the pair of division points. These three points form a triangle, and the set of all such triangles forms a polyhedral surface inscribed in the original cylindrical surface. The area of each elementary triangle is $r \sin (\pi/n) \sqrt{h^2/m^2 + r^2 \left(1- \cos (\pi/n) \right)^2} = r \sin (\pi/n) \sqrt{h^2/m^2 + 4r^2 \sin^4 (\pi/2n)}$, and the total area is $2\left(n\sin (\pi/n)\right)rh\sqrt{1+4m^2r^2/h^2 \sin^4 (\pi/2n)} = 2\pi r h \left(1 + \pi^4 m^2 r^2/8h^2n^4 \right)$, when $m, ~n \rightarrow +\infty$. The area of the polyhedral surface depends on the limit of $m/n^2$, which can even reach infinity.}
\end{figure}

Theorem \ref{thm: first-order convergence} indicates that, for each discrete net listed in Section \ref{section: difference equation}, once the initial value and parameter are bounded and converge to the initial condition for its smooth analogue listed in Section \ref{section: differential equation}, the discrete surface will converge to the corresponding smooth surface. Note that as explained in Section \ref{section: differential equation} and Section \ref{section: difference equation}, for each initial condition problem, the initial value and parameter are bounded, and the result -- $x(u)$, $\partial x / \partial u_1$, $\partial x / \partial u_2$ are bounded and process a global Lipschitz constant.

In the Discussion section of the main text, we mentioned the convergence in distance does not guarantee the convergence of tangent plane, as well as other metric- and curvature- related properties. The zig-zag mode is a common reason for such non-convergence, akin to the Staircase paradox and the Schwarz Lantern illustrated in Figure \ref{fig: staircase ddg}. 

\part{Quad-mesh rigid origami}

\section{The loop condition and the two-vertex system} \label{section: loop condition}

The idea of deriving the flexibility condition of a quad-mesh rigid origami is straightforward, which can be explained by `cutting' through the paper to make the folding of each vertex independently driven (similar to single degree-of-freedom robotic arms), then consider the condition to properly `glue' them together. A graphical illustration of both the original and the cut quad mesh is provided in Figure~\ref{fig: flexibility condtion}(a) and (b), where we denote the tangent of half the folding angles at the labeled creases by $y_{ij}$ and $w_{ij}$, with $i, j \in \mathbb{Z}_+$. We use the tangent of half the folding angles — rather than the angles themselves — because this quantity appears consistently in all subsequent calculations.

\begin{prop} \label{prop: loop condition}	
	(The loop condition, \citep{tachi_generalization_2009}) A quad-mesh rigid origami is flexible if and only if:
	\begin{enumerate} [label={[\arabic*]}]
		\item The cut quad-mesh is flexible. Consequently, there exists a non-constant, analytic one-parameter flex for all $y_{ij}(t)$ and $w_{ij}(t)$ over $t \in [0, 1]$.
		\item For all $i, ~j$, 
		\begin{equation} \label{eq: loop condition}
			y_{i,j}(t) = y_{i,j+1}(t), ~~w_{i,j}(t) = w_{i,j+1}(t), ~~ \mathrm{for~} t \in [0, 1]
		\end{equation}
	\end{enumerate}
\end{prop}

Note that condition [1] above is also essential since the cut quad-mesh might be rigid at special configurations. An example is provided in Figure \ref{fig: flexibility condtion}(c). Further, Proposition \ref{prop: loop condition} infers that:

\begin{prop}
	\citep{schief_integrability_2008} A quad-mesh rigid origami is flexible if and only if all its $3 \times 3$ quad-meshes (Kokotsakis quadrilaterals) are flexible.  
\end{prop}

\begin{figure} [tbph]
	\noindent \begin{centering}
		\includegraphics[width=1\linewidth]{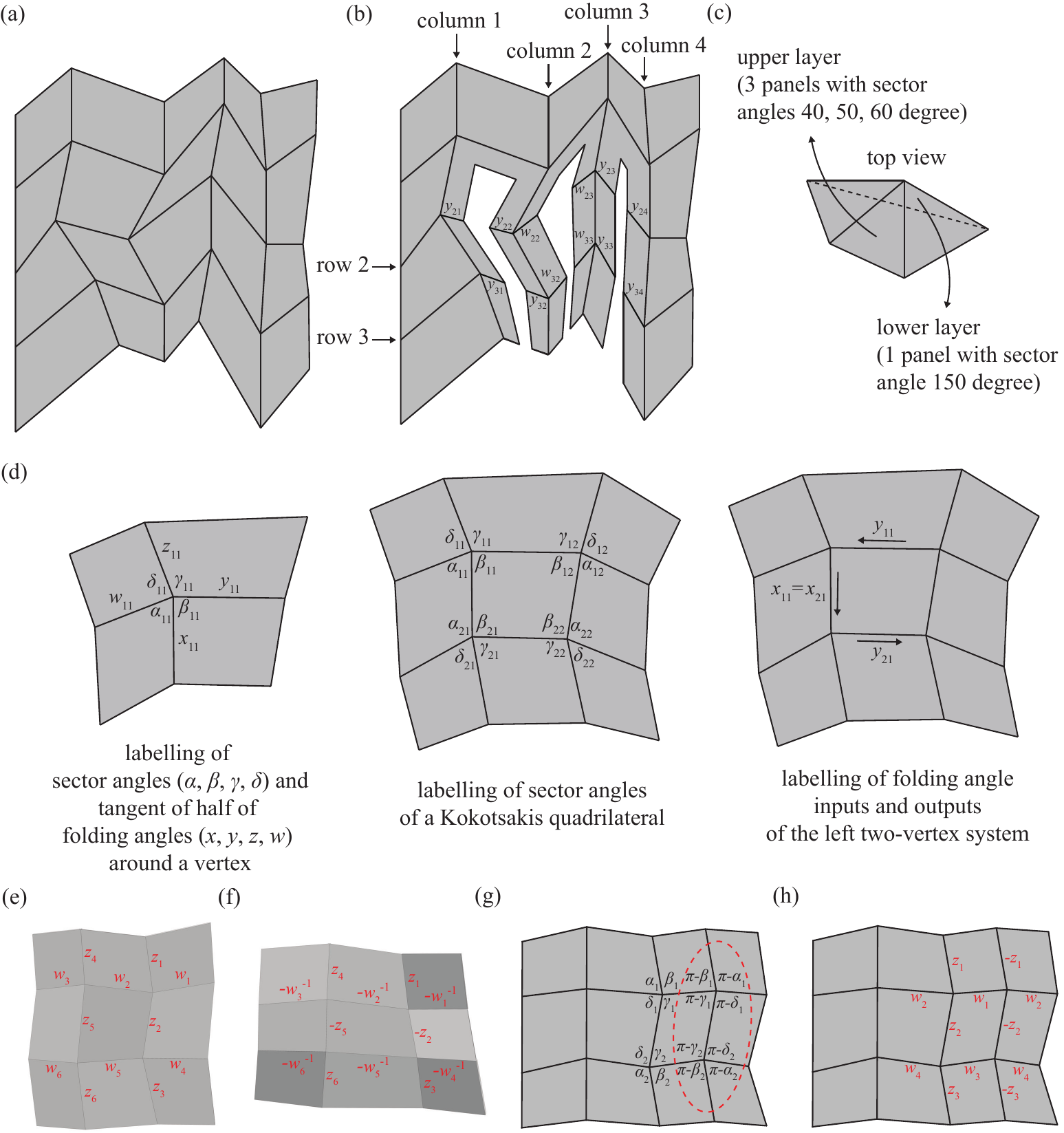}
		\par \end{centering}
	
	\caption{\label{fig: flexibility condtion}(a) and (b) explain how we `cut' a quad-mesh rigid origami to make the interior of crease pattern a tree. When (b) is foldable, the motion of (b) agrees with a quad-mesh rigid origami if and only if all the folding angles $y$ and $w$ are identical as indicated in \eqref{eq: loop condition}. (c) shows a vertex with sector angles $40^\circ, ~50^\circ, ~60^\circ, ~150^\circ$, which forms a double cover of a circular sector and hence remain rigid. (d) Labelling of a single-vertex, a Kokotsakis quadrilateral and a two-vertex system. (e) We use $\{z_1, ~z_2, ~z_3, ~z_4, ~z_5, ~z_6\}$ and $\{w_1, ~w_2, ~w_3, ~w_4, ~w_5, ~w_6\}$ to represent the tangent of half of the folding angles on these labelled interior creases. (f) the sector angles on panels of the middle row are replaced by their complements to $\pi$, other sector angles remain unchanged. (g) shows how a parallel strip is added, where the parallel strip is marked  with a dashed cycle. In (h), the magnitude of new folding angles are labelled.} 
\end{figure} 

The library of flexible Kokotsakis quadrilaterals are derived in the complexified configuration space, where each Kokotsakis quadrilateral is flexible upon a system of constraints on the sector angles --- most of these constraints are highly nonlinear. Our target is to explore all the `stitchings' of Kokotsakis quadrilaterals that can form a quad-mesh rigid origami with the following requirements: 1) we require the construction of rigid origami to be `infinitely extendable', in other words, not constrained in a finite grid. 2) we assume the number of variables is no less than the number of constraints; 3) for the admissible stitchings, we require the existence of valid real solutions from numerical examination; 4) on top of a valid numerical solution, we require the rigid origami to have an actual plotable folding motion. Otherwise, it might locate in the complexified configuration space and the structure will remain rigid even satisfying the flexibility constraint.

In terms of the flexibility condition of a Kokotsakis quadrilateral, it is convenient to consider the `compatibility' of its two two-vertex systems, as shown in Figure \ref{fig: flexibility condtion}(d). Here a Kokotsakis quadrilateral is `divided' to its left and right two-vertex systems. For the left two-vertex system, we start from the input $y_{11}$, going through the left top vertex to obtain the output $x_{11}$. This output $x_{11}$ equals to the input of the left bottom vertex $x_{21}$, with which we can further calculate $y_{21}$. Consequently, the left two-vertex system generates its output $y_{21}$ as a function of its input $y_{11}$. Clearly, if and only if the $y_{21}$ calculated from the left and right two-vertex systems are identical for all $y_{11}$, the Kokotsakis quadrilateral will be flexible. A two-vertex system from $y_{11}$ to $y_{21}$ is clearly a compound function on the relation between adjacent folding angles of two degree-4 single-vertex rigid origami.

In \citet[Section 2]{he_real_2023}, we present a comprehensive list of the various types of a single vertex. Notably, the terms (anti-)isogram, (anti-)deltoid, conic, and elliptic are included in this list.

We will now introduce two operations that can create a new flexible quad-mesh rigid origami from an existing one: these are called 'switching a strip' and 'adding a parallel strip.'

\begin{defn} \label{defn: switching} 
	
\textit{Switching a strip} refers to replacing all the sector angles in a row or column of panels with their complements to $\pi$, while keeping the other sector angles unchanged. A visual representation of this operation can be seen in Figures \ref{fig: flexibility condtion}(e) and \ref{fig: flexibility condtion}(f). \textit{Adding a parallel strip} means introducing an additional row or column of vertices with new interior creases, which are parallel to the creases of the adjacent row or column, as shown in Figures \ref{fig: flexibility condtion}(g) and \ref{fig: flexibility condtion}(h). 

\end{defn}

We will demonstrate that both operations --- switching a strip and adding a parallel strip --- preserve the flexibility of a quad-mesh rigid origami. Let's first examine the case of switching a transverse strip. Consider switching the middle row of panels in Figure \ref{fig: flexibility condtion}(e), where the sector angles are replaced by their complements to $\pi$ as shown in Figure \ref{fig: flexibility condtion}(f).

The tangents of half the folding angles on the labelled interior creases are denoted by ${z_1, ~z_2, ~z_3, ~z_4, ~z_5, ~z_6}$ and ${w_1, ~w_2, ~w_3, ~w_4, ~w_5, ~w_6}$. After switching the strip, the folding angles change as follows:
\begin{equation*}
	\begin{gathered}
		\{z_1,~z_2,~z_3,~z_4,~z_5,~z_6\} \rightarrow \{~z_1,~-z_2,~z_3,~z_4,~-z_5,~z_6\} \\
		\{w_1,~w_2,~w_3,~w_4,~w_5,~w_6\} \rightarrow \{-w_1^{-1},~-w_2^{-1},~-w_3^{-1},~-w_4^{-1},~-w_5^{-1},~-w_6^{-1}\}
	\end{gathered}
\end{equation*}
Further details can be found in \cite[Section 5]{he_real_2023}. According to Proposition \ref{prop: loop condition}, switching a strip preserves the flexibility of a quad-mesh rigid origami. The proof for switching a longitudinal strip follows a similar reasoning. Next, after adding a parallel strip, the new folding angles are shown in Figure \ref{fig: flexibility condtion}(h). As per Proposition \ref{prop: loop condition}, adding a parallel strip also maintains the flexibility of a quad-mesh rigid origami.

\section{V-hedra and V-surface} \label{section: V-hedra}

In this section we will give a comprehensive introduction to a V-hedron, as well as its smooth analogue called a V-surface. A V-hedron only contains linear couplings of isograms. The flexibility condition for a V-hedron is the compatible stitching of linear couplings, or equivalently, the existence of a folded state. The properties of a V-hedron are listed below. The additional \textbf{regularity condition} for a V-hedron is at every vertex $\alpha \in (0, \pi), ~\beta \in (0, \pi), ~\alpha + \beta \neq \pi$, i.e, a V-hedron is not developable. 
\begin{prop} \label{prop: V-hedron properties}
	Features for a V-hedron:
	\begin{enumerate} [label={[\arabic*]}]
		\item An V-hedron has a \textit{flat-folded state} where the folding angles around each vertex are $\{0, ~\pi, ~0, ~\pi\}$, up to any cyclic permutation.
		\item If a V-hedron has a non-flat rigidly folded state, this V-hedron is flexible. 
		\item Folding angles are constant along discrete coordinate curves.
		\item \citep{sauer_differenzengeometrie_1970} The Gauss map of a V-hedron is a discrete Chebyshev net. A V-hedron is the only discrete conjugate net with a discrete Chebyshev Gauss map. See Figure \ref{fig: V-hedron}(c). 
	\end{enumerate}
\end{prop}  

\begin{figure} [t]
	\noindent \begin{centering}
		\includegraphics[width=1\linewidth]{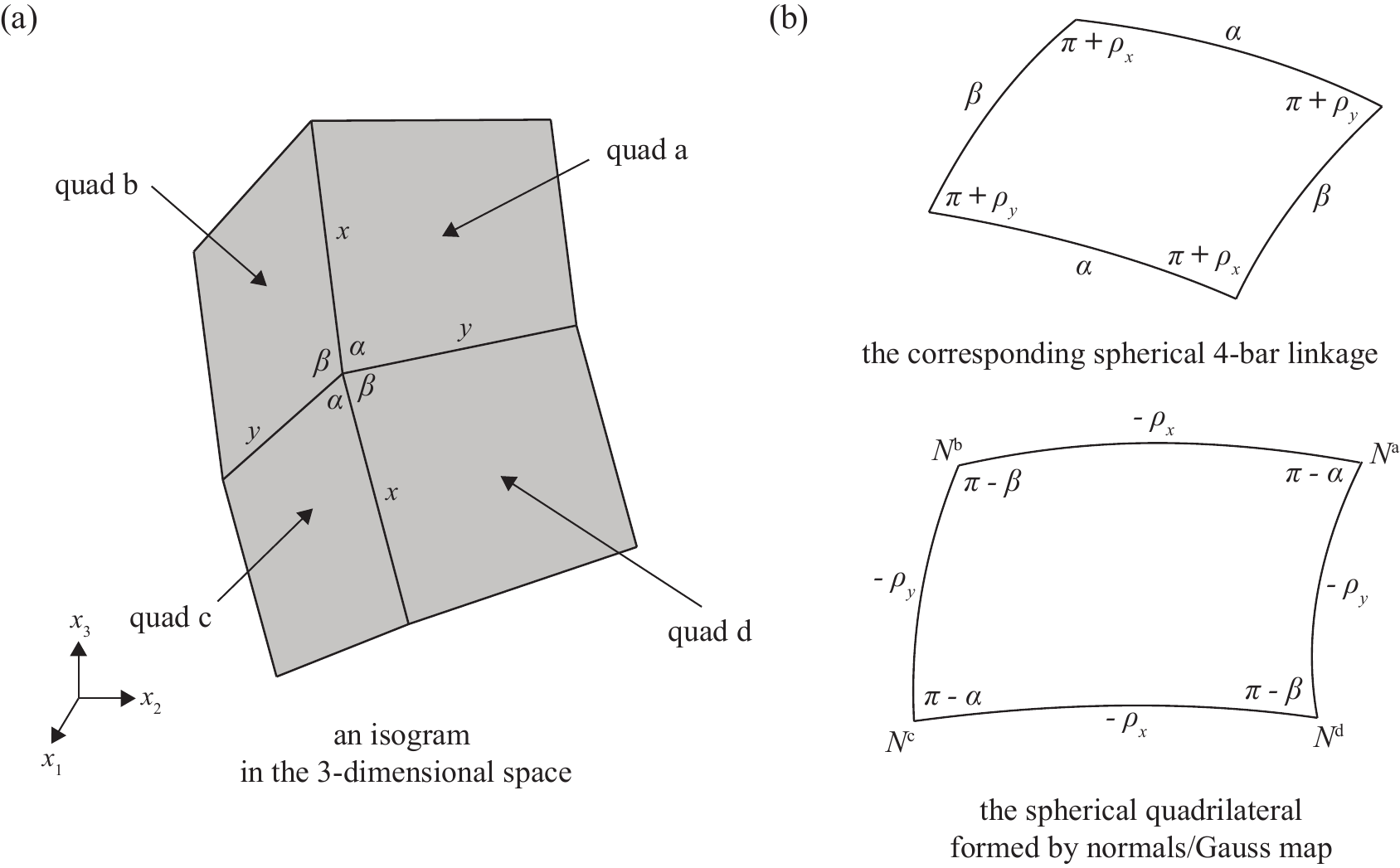}
		\par\end{centering}
	
	\caption{\label{fig: V-hedron}(a) A three-dimensional view of an isogram, where the surrounding quadrilaterals are labelled $\mathrm{a}, ~\mathrm{b}, ~\mathrm{c}, ~\mathrm{d}$. (b)The corresponding spherical 4-bar linkage of this isogram. (c) The spherical parallelogram formed by the Gauss map of surrounding panels.}
\end{figure}

We also say the Gauss map $N$ is a discrete spherical Chebyshev net since it is on the unit sphere. Clearly for any elementary quad of $N$, the opposite spherical arc lengths are also equal.

From the graphical explanation in Figure \ref{fig: reciprocal parallel V-hedron}, a V-hedron is \textit{reciprocal-parallel} related to a K-hedron \citep{sauer_parallelogrammgitter_1950, schief_integrability_2008}. This result is on top of the reciprocal-parallel relation between a discrete asymptotic net and a discrete conjugate net. Let $X(i)$ be a discrete asymptotic net, $Y(i)$ is another discrete surface such that $y(i_1+1, ~i_2) - y(i_1, ~i_2)$ is parallel to $x(i_1, ~i_2) - x(i_1, ~i_2-1)$; $y(i_1+1, ~i_2+1) - y(i_1, ~i_2+1)$ is parallel to $x(i_1, ~i_2+1) - x(i_1, ~i_2)$; $y(i_1, ~i_2+1) - y(i_1, ~i_2)$ is parallel to $x(i_1, ~i_2) - x(i_1-1, ~i_2)$; $y(i_1+1, ~i_2+1) - y(i_1+1, ~i_2)$ is parallel to $x(i_1+1, ~i_2) - x(i_1, ~i_2)$. From Figure \ref{fig: reciprocal parallel}, $X$ is `five points coplanar' if and only if the elementary quadrilateral of $Y$ is planar.

\begin{figure}[p]
	\noindent \begin{centering}
		\includegraphics[width=1\linewidth]{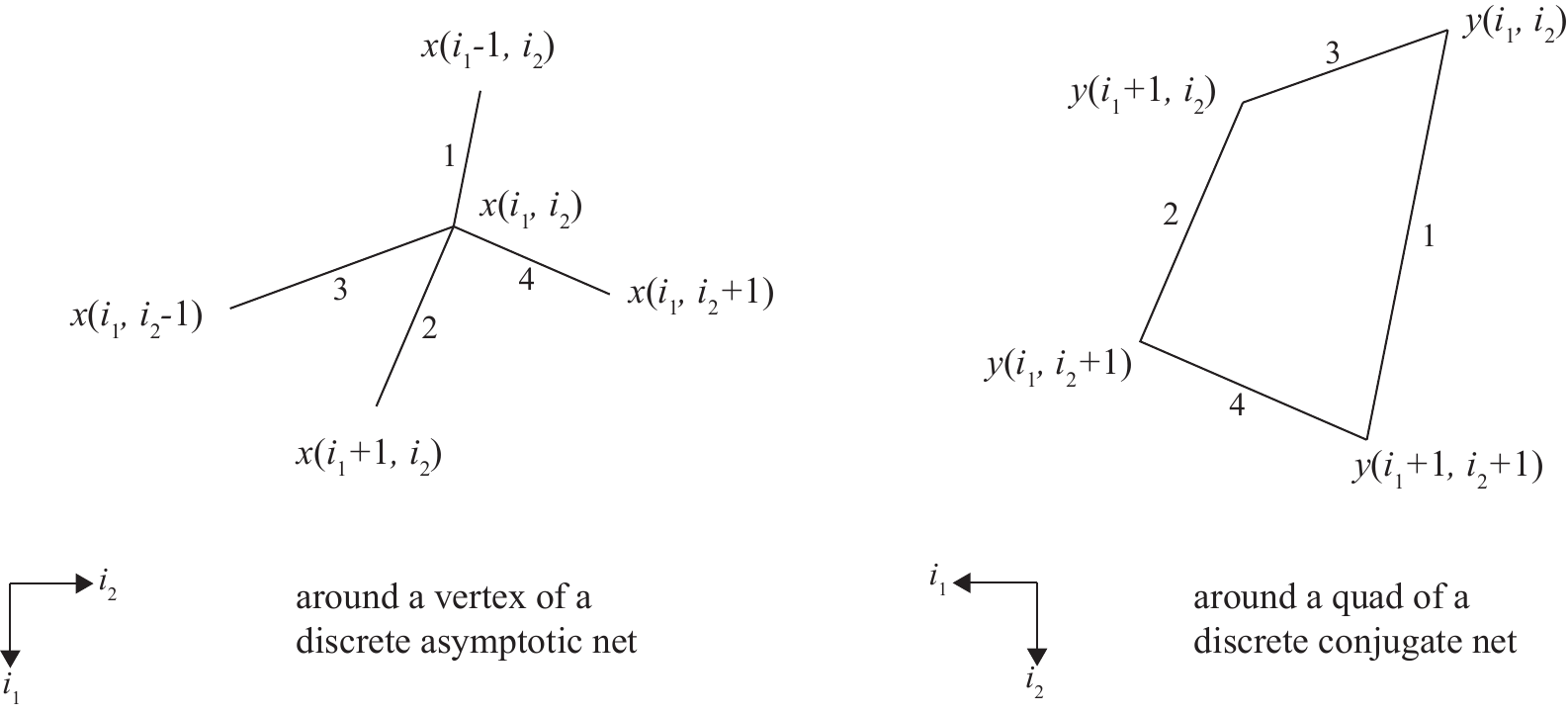}
		\par\end{centering}
	
	\caption{\label{fig: reciprocal parallel}A figure illustrating the reciprocal-parallel relation between an asymptotic net and a conjugate net. Lines labelled with the same numbers are parallel.}
	~\\
	\noindent \begin{centering}
		\includegraphics[width=1\linewidth]{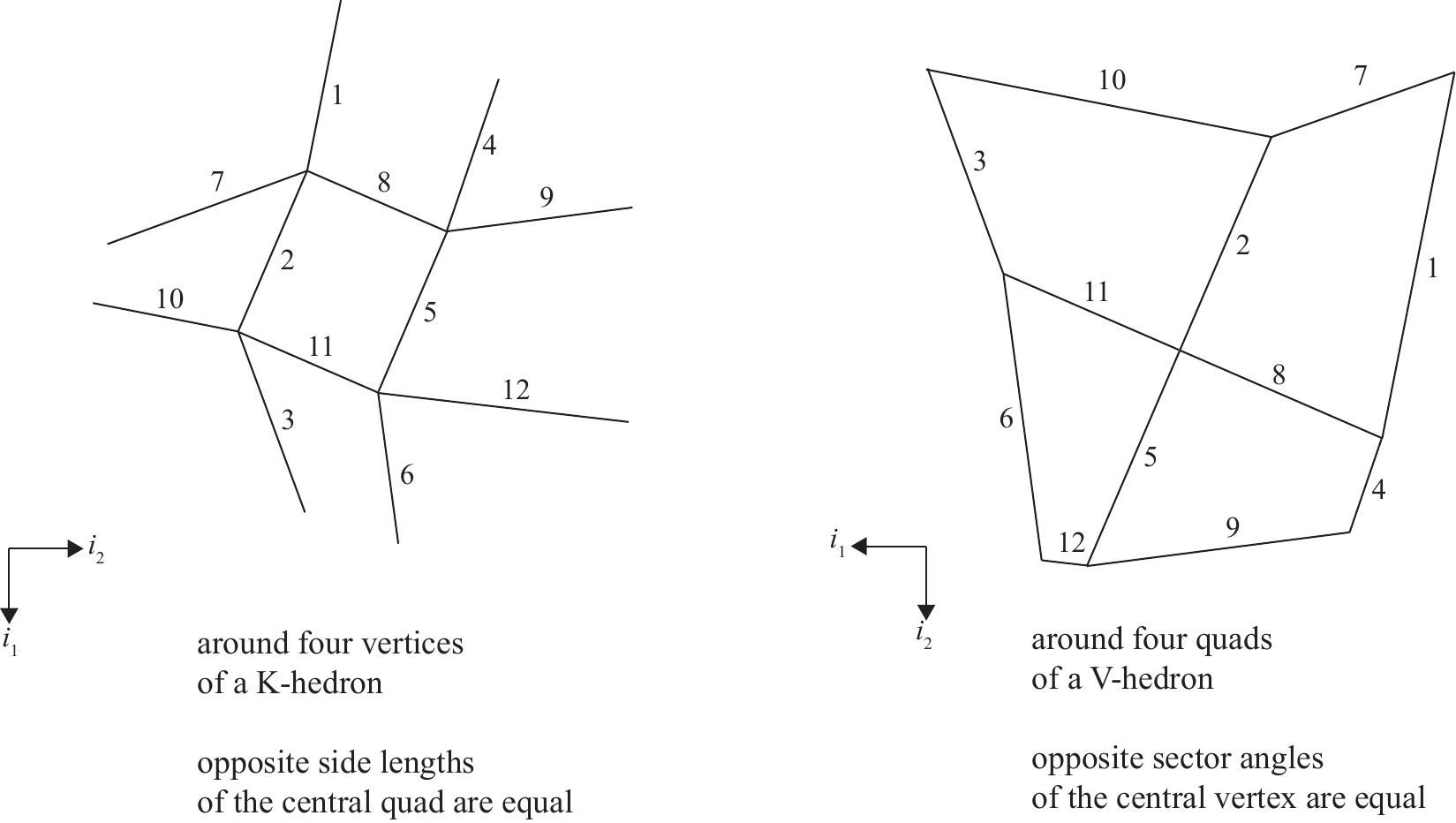}
		\par\end{centering}
	
	\caption{\label{fig: reciprocal parallel V-hedron}A figure illustrating the reciprocal-parallel relation between a V-hedron and a K-hedron. Lines labelled with the same numbers are parallel.}
\end{figure}

Define the non-zero coefficients $a, b: \mathbb{Z}^2 \rightarrow \mathbb{R}, ~a, ~b \neq 0$, which live on the $i_1$ and $i_2$ grid lines, respectively:
\begin{equation*}
	\begin{gathered}
		y(i_1+1, ~i_2) - y(i_1, ~i_2) = -b(i_1, ~i_2)(x(i_1, ~i_2) - x(i_1, ~i_2-1))  \\
		y(i_1, ~i_2+1) - y(i_1, ~i_2) = a(i_1, ~i_2)(x(i_1, ~i_2) - x(i_1-1, ~i_2))
	\end{gathered}
\end{equation*}
For simplicity we further require $ab > 0$ over $\mathbb{Z}^2$, hence $X$ and $Y$ will share the same discrete normal vector field. Given a discrete asymptotic net $X$, $Y$ will be a discrete conjugate net determined upon $a, ~b$ up to a translation. The inverse statement also holds. Given a discrete conjugate net $Y$, $X$ will be a discrete asymptotic net determined upon $a, ~b$ up to a translation.

Next we will do a series of calculation to obtain the discrete Moutard equation, \eqref{eq: discrete Moutard}, for the normal vector field of a V-hedron. Let
\begin{equation*}
	\begin{gathered}
		a = \dfrac{x(i_1, ~i_2) - x(i_1-1, ~i_2)}{\|x(i_1, ~i_2) - x(i_1-1, ~i_2)\|}, ~~b = \dfrac{x(i_1, ~i_2) - x(i_1, ~i_2-1)}{\|x(i_1, ~i_2) - x(i_1, ~i_2-1)\|} \\
		c = \dfrac{x(i_1, ~i_2) - x(i_1+1, ~i_2)}{\|x(i_1, ~i_2) - x(i_1+1, ~i_2)\|}, ~~d = \dfrac{x(i_1, ~i_2) - x(i_1, ~i_2+1)}{\|x(i_1, ~i_2) - x(i_1, ~i_2+1)\|} 
	\end{gathered}
\end{equation*}
Note that $a, ~b, ~c, ~d$ are temporary variables, which is updated from its previous usage. The Gauss map on each quadrilateral around vertex $i$ is
\begin{equation*}
	\begin{gathered}
		N^\mathrm{a} = \dfrac{d \times a}{\sin \alpha_1}, ~~ N^\mathrm{b} = \dfrac{a \times b}{\sin \alpha_2} \\
		N^\mathrm{c} = \dfrac{b \times c}{\sin \alpha_1}, ~~
		N^\mathrm{d} = \dfrac{c \times d}{\sin \alpha_2}
	\end{gathered}
\end{equation*}
since
\begin{equation*}
	\begin{gathered}
		N^\mathrm{a} \cdot N^\mathrm{b} = \cos \rho_1, ~~ N^\mathrm{b} \cdot N^\mathrm{c} = \cos \rho_2 \\
		N^\mathrm{c} \cdot N^\mathrm{d} = \cos \rho_1, ~~ N^\mathrm{d} \cdot N^\mathrm{a} = \cos \rho_2
	\end{gathered}
\end{equation*}
we have
\begin{equation*}
	\begin{gathered}
		(N^\mathrm{a} + N^\mathrm{c}) \cdot (N^\mathrm{b} - N^\mathrm{d}) = 0 \\
		(N^\mathrm{a} - N^\mathrm{c}) \cdot (N^\mathrm{b} + N^\mathrm{d}) = 0 \\
	\end{gathered}
\end{equation*}
Furthermore, 
\begin{equation*}
	\begin{gathered}
		(N^\mathrm{a} + N^\mathrm{c}) \cdot (N^\mathrm{a} - N^\mathrm{c}) = 0 \\
		(N^\mathrm{b} - N^\mathrm{d}) \cdot (N^\mathrm{b} + N^\mathrm{d}) = 0 \\
	\end{gathered}
\end{equation*}
We could see that either $(N^\mathrm{a} + N^\mathrm{c})$ is parallel to $(N^\mathrm{b} + N^\mathrm{d})$, or $(N^\mathrm{a} - N^\mathrm{c})$ is parallel to $(N^\mathrm{b} - N^\mathrm{d})$. It can be verified from the derivation below:
\begin{equation*}
	\begin{gathered}
		N^\mathrm{a} + N^\mathrm{c} = \dfrac{d \times a + b \times c}{\sin \alpha_1} \\
		N^\mathrm{b} + N^\mathrm{d} = \dfrac{a \times b + c \times d}{\sin \alpha_2} \\
		(a \times b + c \times d) - (d \times a + b \times c) = (a + c) \times (b + d) = 0 \\
		(a \times b + c \times d) + (d \times a + b \times c) = (a - c) \times (b - d)
	\end{gathered}
\end{equation*}
which makes use of the special geometry from the spherical quadrilateral Figure \ref{fig: V-hedron}(c):
\begin{equation*}
	\begin{gathered}
		d - a = c - b \\
		d - c = a - b
	\end{gathered} ~~\Rightarrow~~ a + c = b + d
\end{equation*}
It leads to the following relation on the normal vectors:
\begin{equation*}
	\begin{aligned}
		\triangle_1\triangle_2 N & = N^\mathrm{b} + N^\mathrm{d} - N^\mathrm{a} - N^\mathrm{c} \\
		& = \dfrac{\sin \alpha_1 - \sin \alpha_2}{\sin \alpha_1 + \sin \alpha_2} (N^\mathrm{a} + N^\mathrm{b} + N^\mathrm{c} + N^\mathrm{d}) \\
		& = \left( \dfrac{\sin \alpha_1}{\sin \alpha_2} - 1 \right) (N^\mathrm{a} + N^\mathrm{c})
	\end{aligned}
\end{equation*}
Assume there is no self-intersection, $N^\mathrm{a} - N^\mathrm{c}$ will not be parallel to $N^\mathrm{b} - N^\mathrm{d}$ (diagonals of a spherical parallelogram will not be parallel). From the symmetry of the Gauss map:
\begin{equation*}
	\begin{gathered}
		N^\mathrm{a} \cdot \dfrac{N^\mathrm{a} + N^\mathrm{c}}{||N^\mathrm{a} + N^\mathrm{c}||} = \dfrac{||N^\mathrm{a} + N^\mathrm{c}||}{2} \\
		N^\mathrm{b} \cdot \dfrac{N^\mathrm{a} + N^\mathrm{c}}{||N^\mathrm{a} + N^\mathrm{c}||} = \dfrac{||N^\mathrm{b} + N^\mathrm{d}||}{2} \\
	\end{gathered}	 
\end{equation*}
hence
\begin{equation*}
	\begin{gathered}
		N^\mathrm{b} + N^\mathrm{d} =
		\dfrac{N^\mathrm{b} \cdot (N^\mathrm{a} + N^\mathrm{c})}{N^\mathrm{a} \cdot (N^\mathrm{a} + N^\mathrm{c})} (N^\mathrm{a} + N^\mathrm{c}) \\
	\end{gathered}	 
\end{equation*}
which leads to the proposition below:
\begin{prop}
	\citep{bobenko_discrete_1996}
	Let $X: \mathbb{Z}^2 \rightarrow \mathbb{R}^3$ be a V-hedron. The Gauss map $N$ is a discrete Moutard net satisfying the equation below:
	\begin{equation} \label{eq: Moutard for N}
		\begin{gathered}
			\triangle_1 \triangle_2 N =  \lambda(i_1, ~i_2) \left(N(i_1+1, ~i_2) + N(i_1, ~i_2+1) \right) \\
			\lambda(i_1, ~i_2) = \dfrac{N(i_1, ~i_2) \cdot \left(N(i_1+1, ~i_2) + N(i_1, ~i_2+1)\right)}{1+ N(i_1+1, ~i_2) \cdot N(i_1, ~i_2+1)} - 1
		\end{gathered}
	\end{equation}
	Equivalently, 
	\begin{equation*}
		N(i_1+1, ~i_2+1) = -N(i_1, ~i_2) + (\lambda(i_1, ~i_2)+1) \left(N(i_1+1, ~i_2) + N(i_1, ~i_2+1) \right)
	\end{equation*}
\end{prop}

The above proposition infers the following two constructions for a V-hedron. The first construction is based on the position of normal vectors on the coordinate curves. The second construction is based on the sector angles on the coordinate curves, hence the position of the rigid origami is determined up to an orthogonal transformation. The labelling is provided in Figure \ref{fig: labelling discrete net}.
\begin{description}
	\item [Initial Value 1] Two discrete coordinate curves $x(i_2 = 0)$ and $x(i_1 = 0)$ intersecting at $x(0, ~0)$. $S^\mathrm{a}(-1, ~0) = S^\mathrm{c}(0, ~-1)$, $S^\mathrm{b}(-1, ~-1) = S^\mathrm{d}(0, ~0)$. Sector angles $S^\mathrm{a}(i_1 \ge 0, ~0)$; $S^\mathrm{c}(0, ~i_2 \ge 0)$; $S^\mathrm{c}(i_1 \le -1, ~-1)$; $S^\mathrm{a}(-1, ~i_2 \le -1)$.
	
	Note that this input is equivalent to two boundary polylines and the direction vectors along them \citep{sauer_differenzengeometrie_1970}. 
	
	\item [Step 1a] From the above initial value we can immediately calculate $x(1, ~1)$ and $N(0, ~0)$, then from iterative calculation we could obtain $N(i_1 \ge 1, ~i_2 = 0)$ and $N(i_1 = 0, ~i_2 \ge 1)$. 
	\item [Step 1b] Use \eqref{eq: Moutard for N} to calculate $N(i_1 \ge 1, ~i_2 = 1)$ and $N(i_1 = 1, ~i_2 \ge 1)$.
	\item [Step 1c] Use the equations below to locate $x(i_1 \ge 2, ~i_2 = 1)$ and $x(i_1 = 1, ~i_2 \ge 2)$:
	\begin{equation*}
		\begin{dcases*}
			\triangle_1 x(i_1, ~1) \mathrm{~is~parallel~to~} N(i_1, ~1) \times N(i, ~0) \\
			\triangle_2 x(1, ~i_2) \mathrm{~is~parallel~to~} N(1, ~i_2) \times N(0, ~j) 
		\end{dcases*}, ~~i_1, ~i_2 \in \mathbb{Z}_+, ~~i_1, ~i_2 \ge 1
	\end{equation*} 
	\item [Step 1d] In the quadrant $\mathbb{Z}_+^2$, repeat Step 1b and Step 1c to obtain $x(i_1 \ge 0, ~i_2 \ge 0)$ and its Gauss map $N(i_1 \ge 0, ~i_2 \ge 0)$. 
	\item [Step 1e] Use the same method as described in Step 1d to calculate the other three quadrants to obtain the entire mesh.
	\item [Regularity condition] Every step returns a non-degenerated and bounded result.
\end{description}	
Note that Step 1b and Step 1c are equivalent to the condition requiring opposite sector angles equal. The intertwining calculation involving the Gauss map is relatively simple and does not require solving implicit equations.

\begin{description}
	\item[Initial Value 2] Sector angles on the two discrete coordinate curves $S^\mathrm{a}(i_1, ~0)$; $S^\mathrm{a}(-1, ~i_2)$; $S^\mathrm{b}(i_1, ~-1)$; $S^\mathrm{b}(-1, ~i_2)$; $S^\mathrm{c}(i_1, ~-1) = S^\mathrm{a}(i_1, ~0)$; $S^\mathrm{c}(0, ~i_2) = S^\mathrm{a}(-1, ~i_2)$; $S^\mathrm{d}(i_1, ~0) = S^\mathrm{b}(i_1, ~-1)$; $S^\mathrm{d}(0, ~i_2) = S^\mathrm{b}(-1, ~i_2)$. Crease lengths on the two discrete coordinate curves $\|\triangle_1x (i_1, ~0)\|$ and $\|\triangle_2x (0, ~i_2)\|$.
	\item [Step 2a] From the above initial value, as the sum of sector angles on a quadrilateral equals to $2\pi$, and the opposite sector angles at each vertex are equal, we could calculate $S^\mathrm{b}(0, ~0)$ and $S^\mathrm{d}(1, ~1)$, then apply the equality of linear dependence coefficients, we could calculate $S^\mathrm{a}(0, ~1)$ and $S^\mathrm{c}(1, ~0)$. Next, from iterative calculation we will obtain the sector angles on $x(i_2 = 1)$ and $x(i_1 = 1)$.
	\item [Step 2b] In the quadrant $\mathbb{Z}_+^2$, repeat Step 2a to obtain the sector angles of $x(i_1 \ge 0, ~i_2 \ge 0)$.
	\item [Step 2c] Use the same method described in Step 2b to calculate the other three quadrants to obtain the sector angles of the entire mesh.
	\item [Step 2d] After all the sector angles are determined, the crease lengths on the two discrete coordinate curves will fully determine the shape of the quad-mesh, up to an orthogonal transformation.
	\item [Regularity condition] The result of calculating a sector angle always falls in $(0, \pi)$.
\end{description}

Next we explain the smooth analogue of a V-hedron -- called a \textit{V-surface}. A V-hedron is a discrete geodesic conjugate net. From Section \ref{section: discrete net}, the smooth analogue of a V-hedron should be a smooth geodesic conjugate net. The additional \textbf{regularity condition} for a V-surface is non-developable, which means not being an orthogonal net simultaneously. From \eqref{eq: geodesic net}, the condition for a surface parametrization to form a geodesic conjugate net, i.e, to be a V-surface is:
\begin{equation*} 
	\begin{gathered}
		\begin{dcases}
			\dfrac{\partial^2 x}{\partial u_1^2} \mathrm{~is~on~the~plane~spanned ~from~} N \mathrm{~and~} \dfrac{\partial x}{\partial u_1} \\
			\dfrac{\partial^2 x}{\partial u_2^2} \mathrm{~is~on~the~plane~spanned ~from~} N \mathrm{~and~} \dfrac{\partial x}{\partial u_2}
		\end{dcases} ~~\Leftrightarrow~~
		\begin{dcases}
			2 \mathrm{I}_{11}  \dfrac{\partial \mathrm{I}_{12}}{\partial u_1} = \mathrm{I}_{11} \dfrac{\partial \mathrm{I}_{11}}{\partial u_2} + \mathrm{I}_{12} \dfrac{\partial \mathrm{I}_{11}}{\partial u_1} \\
			2\mathrm{I}_{22} \dfrac{\partial \mathrm{I}_{12}}{\partial u_2} = \mathrm{I}_{22} \dfrac{\partial \mathrm{I}_{22}}{\partial u_1} + \mathrm{I}_{12} \dfrac{\partial \mathrm{I}_{22}}{\partial u_2} \\
		\end{dcases} \\
		\mathrm{II}_{12} = 0
	\end{gathered}
\end{equation*}
Recall that the principal curvatures are from the simultaneous diagonalization of the first and second fundamental forms. There exists a $2 \times 2$ matrix $A$ such that:
\begin{equation*}
	\begin{gathered}
		\begin{bmatrix}
			\mathrm{I}_{11} & \mathrm{I}_{12} \\
			\mathrm{I}_{12} & \mathrm{I}_{22}
		\end{bmatrix} =  AA^\mathrm{T}\\
		\begin{bmatrix}
			\mathrm{II}_{11} & \mathrm{II}_{12} \\
			\mathrm{II}_{12} & \mathrm{II}_{22}
		\end{bmatrix} = A \begin{bmatrix}
			\kappa_1 & 0 \\
			0 & \kappa_2
		\end{bmatrix} A^\mathrm{T}
	\end{gathered}
\end{equation*}
We could infer that a sphere is not a V-surface. If so, from the above relation $\mathrm{I}_{12} = \mathrm{II}_{12} = 0$, hence on top of being a geodesic conjugate net, the surface is also an orthogonal Chebyshev net, which has zero Gaussian curvature.

A V-surface has geometric properties parallel to Proposition \ref{prop: V-hedron properties}. 
\begin{prop}
	Features for a V-surface:
	\begin{enumerate} [label={[\arabic*]}]
		\item \citep{bianchi_sopra_1890} A V-surface admits a one-parameter flex (isometric deformation) and preserves to be a V-surface in this flex.
		\item The Gauss map of a V-surface is a Chebyshev net. A V-surface is the only conjugate net with a Chebyshev Gauss map. 
	\end{enumerate}
\end{prop}
In particular we will prove [1] from calculation. The Gauss-Mainardi-Codazzi Equations yield:
\begin{equation*} 
	\begin{dcases}
		\dfrac{\partial \Gamma^2_{12}}{\partial u_1} + \Gamma^2_{12}\Gamma^2_{12} - \Gamma^1_{11}\Gamma^2_{12} = - \mathrm{I}_{11} \kappa_\mathrm{G} \\
		\dfrac{\partial \mathrm{II}_{11}}{\partial u_2} = \mathrm{II}_{11}\Gamma^1_{12}  \\
		\dfrac{\partial \mathrm{II}_{22}}{\partial u_1} =  \mathrm{II}_{22}\Gamma^2_{12} \\
	\end{dcases}
\end{equation*}
The flex parametrized by $t \in I$ where the first fundamental form $\mathrm{I}$ is preserved and 
\begin{equation} \label{eq: flex for a V-hedron}
	\begin{gathered}
		\mathrm{II}_{11}(t) = \lambda(t)\mathrm{II}_{11}(0) \\
		\mathrm{II}_{12}(t) = 0 \\
		\mathrm{II}_{22}(t) = \dfrac{1}{\lambda(t)}\mathrm{II}_{22}(0) \\
	\end{gathered}, ~~ \lambda: I \rightarrow \mathbb{R}
\end{equation}
meets the Gauss-Mainardi-Codazzi Equations. A V-surface is the only conjugate net that has such a flex.

A V-surface is \textit{reciprocal-parallel} related to a K-surface (asymptotic Chebyshev net, Section \ref{section: coordinate net}). This result is on top of the reciprocal-parallel relation between an asymptotic net and a conjugate net. A conjugate net is \textit{reciprocal-parallel} related with an asymptotic net. Let $X(u)$ be an asymptotic net, $Y(u)$ is another surface such that $\partial y / \partial u_1$ is parallel to $\partial x / \partial u_2$ and $\partial y / \partial u_2$ is parallel to $\partial x / \partial u_1$.
\begin{equation*}
	\begin{gathered}
		\dfrac{\partial y}{\partial u_1} = -b \dfrac{\partial x}{\partial u_2}, ~~ \dfrac{\partial y}{\partial u_2} = a \dfrac{\partial x}{\partial u_1} \\
		a, ~b: I^2 \rightarrow \mathbb{R} \mathrm{~are~smooth~functions},~ a, ~b \neq 0
	\end{gathered}
\end{equation*}
For simplicity we further require $ab > 0$ over $I^2$, hence $X$ and $Y$ will share the same normal vector field:
\begin{equation*}
	N^y = \dfrac{\dfrac{\partial y}{\partial u_1} \times \dfrac{\partial y}{\partial u_2}}{\left\| \dfrac{\partial y}{\partial u_1} \times \dfrac{\partial y}{\partial u_2}\right\|} = \dfrac{\dfrac{\partial x}{\partial u_1} \times \dfrac{\partial x}{\partial u_2}}{\left\|\dfrac{\partial x}{\partial u_1} \times \dfrac{\partial x}{\partial u_2}\right\|} = N^x
\end{equation*}
\begin{equation*}
	\begin{dcases}
		\mathrm{II}_{11}^x = 0 \\
		\mathrm{II}_{22}^x = 0 \\
	\end{dcases} ~~\Leftrightarrow~~ \mathrm{II}_{12}^y = 0
\end{equation*}
The coefficients should satisfy the compatibility condition from $\partial (\partial y / \partial u_1)/ \partial u_2 = \partial (\partial y / \partial u_2) / \partial u_1$, using the Christoffel symbols, \eqref{eq: Christoffel symbol}:
\begin{equation} \label{eq: reciprocal parallel compatibility condition}
	\begin{gathered}
		\dfrac{\partial b}{\partial u_2} + a \Gamma_{11}^{2x} + b \Gamma_{22}^{2x} = 0 \\
		\dfrac{\partial a}{\partial u_1} + a \Gamma_{11}^{1x} + b \Gamma_{22}^{1x} = 0 \\
	\end{gathered}
\end{equation}
The compatibility condition is a first-order hyperbolic system for $a, ~b$, which is well-posed (Section \ref{section: differential equation}), hence $a, ~b$ are determined upon the initial value on the stationary directions $a(u_1, ~0)$ and $b(0, ~u_2)$. Given an asymptotic net $X$, $Y$ will be a conjugate net determined upon $a, ~b$ up to a translation. The inverse statement also holds. Given a conjugate net $Y$, $X$ will be an asymptotic net determined upon $a, ~b$ up to a translation.

Now on top of being asymptotic, assume $X$ is a Chebyshev net:
\begin{equation*}
	\begin{dcases}
		\dfrac{\partial^2 x}{\partial u_1 \partial u_2} \cdot \dfrac{\partial x}{\partial u_1 } = 0 \\
		\dfrac{\partial^2 x}{\partial u_1 \partial u_2} \cdot \dfrac{\partial x}{\partial u_2 } = 0 \\
	\end{dcases} ~~\Leftrightarrow~~ \begin{dcases}
		\dfrac{\partial}{\partial u_1} \left(\dfrac{1}{b}\dfrac{\partial y}{\partial u_1} \right) \cdot \dfrac{\partial y}{\partial u_2} = 0 \\
		\dfrac{\partial}{\partial u_1} \left(\dfrac{1}{b}\dfrac{\partial y}{\partial u_1} \right) \cdot \dfrac{\partial y}{\partial u_1} = 0 \\
		\dfrac{\partial}{\partial u_2} \left(\dfrac{1}{a}\dfrac{\partial y}{\partial u_2} \right) \cdot \dfrac{\partial y}{\partial u_2} = 0 \\
		\dfrac{\partial}{\partial u_2} \left(\dfrac{1}{a}\dfrac{\partial y}{\partial u_2} \right) \cdot \dfrac{\partial y}{\partial u_1} = 0 \\
	\end{dcases}
\end{equation*}
which shows that
\begin{equation*}
	\begin{gathered}
		\begin{dcases}
			\dfrac{\partial}{\partial u_1} \left(\dfrac{1}{b}\right) \mathrm{I}_{12}^y + \dfrac{1}{b} \dfrac{\partial^2 y}{\partial u_1^2} \cdot \dfrac{\partial y}{\partial u_2} = 0 \\
			\dfrac{\partial}{\partial u_1} \left(\dfrac{1}{b}\right) \mathrm{I}_{11}^y + \dfrac{1}{b} \dfrac{\partial^2 y}{\partial u_1^2} \cdot \dfrac{\partial y}{\partial u_1} = 0 \\
			\dfrac{\partial}{\partial u_1} \left(\dfrac{1}{a}\right) \mathrm{I}_{22}^y + \dfrac{1}{a} \dfrac{\partial^2 y}{\partial u_2^2} \cdot \dfrac{\partial y}{\partial u_2} = 0 \\
			\dfrac{\partial}{\partial u_1} \left(\dfrac{1}{a}\right) \mathrm{I}_{12}^y + \dfrac{1}{a} \dfrac{\partial^2 y}{\partial u_2^2} \cdot \dfrac{\partial y}{\partial u_1} = 0 \\
		\end{dcases}
		~~\Rightarrow~~
		\begin{dcases}
			\dfrac{\dfrac{\partial^2 y}{\partial u_1^2} \cdot \dfrac{\partial y}{\partial u_1}}{\dfrac{\partial^2 y}{\partial u_1^2} \cdot \dfrac{\partial y}{\partial u_2}} = \dfrac{\mathrm{I}_{11}^y}{\mathrm{I}_{12}^y} \\
			\dfrac{\dfrac{\partial^2 y}{\partial u_2^2} \cdot \dfrac{\partial y}{\partial u_2}}{\dfrac{\partial^2 y}{\partial u_2^2} \cdot \dfrac{\partial y}{\partial u_1}} = \dfrac{\mathrm{I}_{22}^y}{\mathrm{I}_{12}^y} \\
		\end{dcases}
	\end{gathered}
\end{equation*}

Geometrically it means that in the non-orthogonal frame $(\partial y/\partial u_1, ~\partial y/\partial u_2, ~N)$, $\partial^2 y/\partial u_1^2$ has no component along $\partial y/\partial u_1$, and $\partial^2 y/\partial u_2^2$ has no component along $\partial y/\partial u_2$, hence $Y$ is a geodesic net.

\begin{exam} \label{example: K to V helicoid}
	\citep{izmestiev_voss_2025} Consider a K-surface:
	\begin{equation*}
		x(u_1, ~u_2) = \begin{bmatrix}
			\cos (u_1 - u_2) / \cosh (u_1 + u_2) \\
			\sin (u_1 - u_2) / \cosh (u_1 + u_2) \\
			u_1+u_2 - \tanh (u_1+u_2)
		\end{bmatrix}
	\end{equation*}
	The second fundamental form is:
	\begin{equation*}
		\begin{bmatrix}
			\mathrm{II}_{11}^x & \mathrm{II}_{12}^x \\
			\mathrm{II}_{12}^x & \mathrm{II}_{22}^x 
		\end{bmatrix} = 
		\begin{bmatrix}
			0 & -\dfrac{2\tanh (u_1+u_2)}{\cosh (u_1+u_2)} \\[10pt]
			-\dfrac{2\tanh (u_1+u_2)}{\cosh (u_1+u_2)} & 0
		\end{bmatrix}
	\end{equation*}
	The compatibility condition \eqref{eq: reciprocal parallel compatibility condition} has a solution $a = b$:
	\begin{equation*}
		\dfrac{\partial a}{\partial u_1} = \dfrac{\partial a}{\partial u_2} = -\dfrac{2 a}{\cosh (u_1 + u_2) \sinh (u_1 + u_2)}
	\end{equation*}
	and hence:
	\begin{equation*}
		a = b = \mathrm{Const} \cdot  \left(\dfrac{1+\exp(2(u_1+u_2))}{1-\exp(2(u_1+u_2))}\right)^2
	\end{equation*}
	The V-surface is now ready to be obtain by integration. For example if $\mathrm{Const} = 1$, under a suitable sign choice:
	\begin{equation*}
		y(u_1, ~u_2) = \begin{bmatrix}
			\dfrac{2 (\cosh (u_1 + u_2) + \sinh (u_1 + u_2)) \cos(u_1 - u_2)}{ \cosh (2u_1 + 2u_2) + \sinh (2u_1 + 2u_2) +1} \\[10pt]
			\dfrac{2 (\cosh (u_1 + u_2) + \sinh (u_1 + u_2)) \sin(u_1 - u_2)}{ \cosh (2u_1 + 2u_2) + \sinh (2u_1 + 2u_2) +1} \\
			u_1-u_2
		\end{bmatrix}
	\end{equation*}
	We could see that the V-surface is in the form of $(\rho \cos \theta, ~\rho \sin \theta, ~\theta)$, hence is actually a helicoid. Concurrently, we can construct a V-hedron in the shape of a helicoid from a K-hedron in the shape of a pseudosphere using grid values of $a, ~b$ calculated above.
\end{exam}

\section{T-hedra and T-surface} \label{section: T-hedra}

In this section we will focus on the details of a T-hedron, as well as its smooth analogue called a T-surface. The information here is an excerpt from \citet{izmestiev_isometric_2024}. A T-hedron only contains involutive couplings of orthodiagonal vertices. To be specific, consider the grid depicted in Figure \ref{fig: T hedron} as an example, the condition of being orthodiagonal vertices is:
\begin{equation*}
	\begin{gathered}
		\cos \alpha_{11} \cos \gamma_{11} = \cos \beta_{11} \cos \delta_{11}, ~~
		\cos \alpha_{12} \cos \gamma_{12} = \cos \beta_{12} \cos \delta_{12} \\
		\cos \alpha_{21} \cos \gamma_{21} = \cos \beta_{21} \cos \delta_{21}, ~~
		\cos \alpha_{22} \cos \gamma_{22} = \cos \beta_{22} \cos \delta_{22} \\
	\end{gathered} 
\end{equation*} 
The condition of the involution factor being equal for all the four pairs in a Kokotsakis quadrilateral is:
\begin{equation*} 
	\begin{gathered}
		\dfrac{\tan \alpha_{11}}{\tan \beta_{11}} = \dfrac{\tan \alpha_{21}}{\tan \beta_{21}}, ~~ \dfrac{\tan \alpha_{12}}{\tan \beta_{12}} = \dfrac{\tan \alpha_{22}}{\tan \beta_{22}} \\
		\dfrac{\tan \gamma_{11}}{\tan \beta_{11}} = \dfrac{\tan \gamma_{21}}{\tan \beta_{21}}, ~~ \dfrac{\tan \gamma_{12}}{\tan \beta_{12}} = \dfrac{\tan \gamma_{22}}{\tan \beta_{22}}
	\end{gathered}	
\end{equation*}
The condition of the amplitude being equal is:
\begin{equation*}
	\dfrac{\cos \beta_{11}}{\cos \beta_{21}} = \dfrac{\cos \beta_{12}}{\cos \beta_{22}}
\end{equation*}
As $\beta_{11}+\beta_{21}+\beta_{12}+\beta_{22} = 2\pi$, we could see that either $\beta_{11}+\beta_{21} = \beta_{12} + \beta_{22} = \pi$ or $\beta_{11}+\beta_{12} = \beta_{21} + \beta_{22} = \pi$, i.e., every elementary quadrilateral is a trapezoid. Further, if $\beta_{11}+\beta_{21} = \beta_{12} + \beta_{22} = \pi$, the nearby Kokotsakis quadrilaterals must have $\beta_{21}+\beta_{31} = \beta_{22} + \beta_{32} = \pi$ and $\beta_{12}+\beta_{22} = \beta_{13} + \beta_{23} = \pi$, etc. In combination with each vertex being orthodiagonal and every pair of vertices forms an involutive coupling, we list the properties of a T-hedron below, which is also visualized in Figure \ref{fig: T hedron}.

\begin{figure}[tbph]
	\noindent \begin{centering}
		\includegraphics[width=0.9\linewidth]{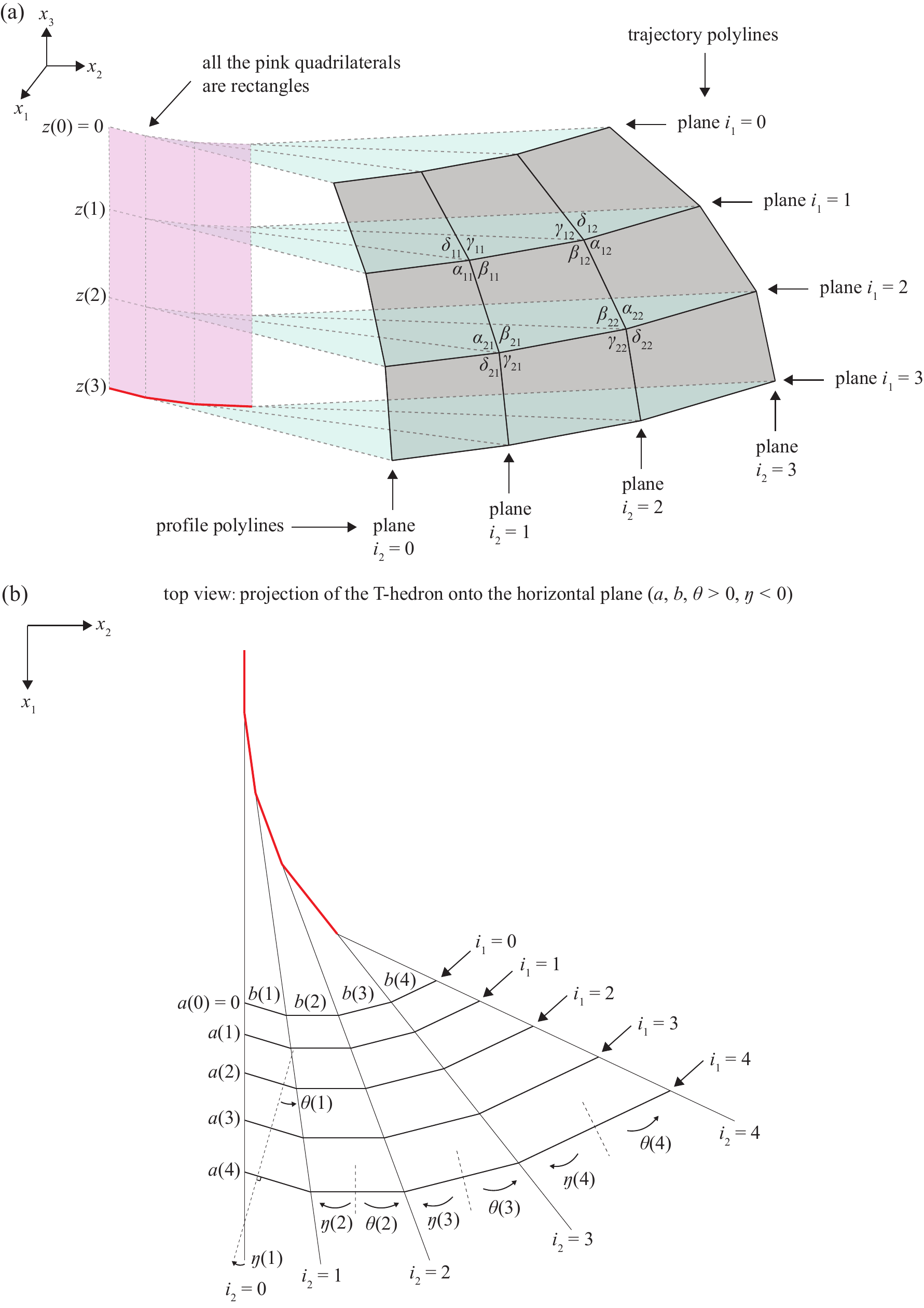}
		\par\end{centering}
	
	\caption{\label{fig: T hedron}(a) shows a T-hedron (coloured grey) and its associated geometry. The sector angles are labelled $\alpha_{i_1, ~i_2}, ~\beta_{i_1, ~i_2}, ~\gamma_{i_1, ~i_2}, ~\delta_{i_1, ~i_2},~i_1, ~i_2 \in \mathbb{Z}$. The horizontal planes ($i_1 = j, ~j \in \mathbb{Z}$) are coloured green, and all the intersections of the horizontal planes and the longitudinal ($i_2 = k, ~k \in \mathbb{Z}$) planes are drawn with dashed lines. These dashed lines intersect with each other consecutively, forming rectangles (coloured pink). (b) is a top view of the projection of a T-hedron to the horizontal plane, which also graphically explains coordinates $a(i_1)$, signed crease lengths $b(i_2)$, signed angles $\eta(i_2)$, $\theta(i_2)$.}
\end{figure}

\begin{prop} \label{prop: T-hedron properties}
	Features for a T-hedron:
	\begin{enumerate} [label={[\arabic*]}]
		\item Every elementary quadrilateral is a trapezoid, the parallel sides of all the trapezoids are all horizontal or all longitudinal.
		\item Every row of vertices ($i_1 = j, ~j \in \mathbb{Z}$) is coplanar. Every column of vertices ($i_2 = k, ~k \in \mathbb{Z}$) is coplanar.
		\item Plane $i_1 = j, ~j \in \mathbb{Z}$ is orthogonal to plane $i_2 = k, ~k \in \mathbb{Z}$.
		\item Either all the horizontal planes $i_1 = j, ~j \in \mathbb{Z}$ are parallel to each other (Figure \ref{fig: T hedron}), or all the longitudinal planes are parallel to each other. 
	\end{enumerate}
\end{prop} 

Statement [4] holds since if not all horizontal planes are parallel to each other, we could take two intersecting horizontal planes, all the longitudinal planes will be perpendicular to this intersection, hence all the longitudinal planes are parallel to each other.

To reach an analytical description of a T-hedron, we will use the quantities graphically defined in Figure \ref{fig: T hedron}(b) to write the coordinate of every vertex of the T-hedron: $\eta(1) \in (-\pi/2, \pi/2)$ is the rotation from the projection of plane $i_2 = 0$ to the line perpendicular to the parallel edges of trapezoid on column 0; $\theta(1) \in (-\pi/2, \pi/2)$ is the rotation from the aforementioned line to the projection of plane $i_2 = 1$, $\eta(i_2) \in (-\pi/2, \pi/2)$ and $\theta(i_2) \in (-\pi/2, \pi/2)$ are defined in a similar way. $a(i_1) \in \mathbb{R}$ are the coordinates along the $x_1$ axis of the projection of row $i_1$, column $0$ of the T-hedron; $b(i_2) \in \mathbb{R}$ are the signed crease lengths of the projection of row $0$, column $i_2$ of the T-hedron. In Figure \ref{fig: T hedron} all the $a(i_1 \neq 0)$, $b(i_2)$ are positive. The projection of each row $i_1 = j$ of the T-hedron is called a \textbf{trajectory polyline}. The projection of each column $i_2 = k$ of the T-hedron is called a \textbf{profile polyline}. In addition to being a discrete surface, we further apply the regularity condition to the data mentioned above:

\begin{description}
	\item [Additional regularity condition] $b(i_2) \neq 0$ for all $i_2$, $z(0) \neq z(1) \neq \cdots \neq z(i_1)$ for all $i_1$.
\end{description}

Let
\begin{equation*}
	\phi(i_2) = \sum \limits_{k=1}^{i_2} \eta(k) + \sum \limits_{k=1}^{i_2} \theta(k), ~~\psi(i_2) = \sum \limits_{k=1}^{i_2} \eta(k) + \sum \limits_{k=1}^{i_2-1} \theta(k)
\end{equation*}
The vertices on row $i_1$ are on the same horizontal plane:
\begin{equation} \label{eq: T-hedra 1}
	x(i_1, ~i_2) = \begin{bmatrix}
		x_1(i_1, ~i_2) \\
		x_2(i_1, ~i_2) \\ 
		z(i_1)
	\end{bmatrix}
\end{equation}
the coordinates of vertices on the first row $i_1 = 0$ is:
\begin{equation*}
	\begin{bmatrix}
		x_1(0, ~i_2) \\
		x_2(0, ~i_2) 
	\end{bmatrix} = \sum \limits_{k=1}^{i_2} b(k) \begin{bmatrix}
		- \sin \psi(k) \\
		\cos \psi(k)
	\end{bmatrix}
\end{equation*}
The signed distance between $[x_1(i_1, ~1); ~x_2(i_1, ~1)]$ and $[x_1(i_1, ~0); ~x_2(i_1, ~0)]$ on the horizontal plane is:
\begin{equation*}
	a(i_1) \dfrac{\cos \eta(1)}{\cos \theta(1)}
\end{equation*}
Similarly, the signed distance between $[x_1(i_1, ~i_2); ~x_2(i_1, ~i_2)]$ and $[x_1(i_1, ~0); ~x_2(i_1, ~0)]$ on the horizontal plane is:
\begin{equation*}
	\begin{gathered}
		a(i_1) \dfrac{\cos \eta(1) \cos \eta(2) \cdots \cos \eta(i_2)}{\cos \theta(1) \cos \theta(2) \cdots \cos \theta(i_2)} = a(i_1) \prod \limits_{k = 1}^{i_2} \dfrac{ \cos \eta(k)}{ \cos \theta(k)} = a(i_1) c(i_2) \\ \mathrm{~if~set~} c(i_2) = \prod \limits_{k = 1}^{i_2} \dfrac{ \cos \eta(k) }{ \cos \theta(k) }
	\end{gathered}
\end{equation*}
then we could calculate the coordinate on each column:
\begin{equation} \label{eq: T-hedra 2}
	\begin{aligned}
		\begin{bmatrix}
			x_1(i_1, ~i_2) \\
			x_2(i_1, ~i_2) 
		\end{bmatrix} & = 	\begin{bmatrix}
			x_1(0, ~i_2) \\
			x_2(0, ~i_2) 
		\end{bmatrix} + a(i_1) c(i_2) \begin{bmatrix}
			\cos \phi(i_2) \\
			\sin \phi(i_2) 
		\end{bmatrix} \\
		& = \sum \limits_{k=1}^{i_2} b(k) \begin{bmatrix}
			- \sin \psi(k) \\
			\cos \psi(k)
		\end{bmatrix} + a(i_1) c(i_2) \begin{bmatrix}
			\cos \phi(i_2) \\
			\sin \phi(i_2) 
		\end{bmatrix}
	\end{aligned} 
\end{equation}
To summarize, the dataset $\eta(i_2), ~\theta(i_2), ~a(i_1), ~b(i_2), ~z(i_1)$, or equivalently $\phi(i_2), ~\psi(i_2), ~a(i_1), ~b(i_2), ~z(i_1)$ uniquely determines a T-hedron upon the regularity condition.

\cite{izmestiev_isometric_2024} also provides several special T-hedra with graphical illustration, including 1) the molding surface: $\eta(i_2) = \theta(i_2)$. Here every trapezoid is isosceles, consequently, every trapezoid have same sector angles; 2) the axial surface: the trajectory polyline at $i_1 = 0$ degenerates to a single point; 3) surface of revolution: being both a molding surface and an axial surface; 4) translational surface: the trajectory polyline at $i_1 = 0$ degenerates to a single point at infinity. Here every trapezoid is a parallelogram.

Next we will calculate the coordinates of all the vertices $x(i; ~t), ~t \in [-\epsilon, \epsilon], ~\epsilon \in (0, 1)$ in its one-parameter folding motion. By applying a proper rotation and translation, all the green planes of $x(i; ~t)$ in Figure \ref{fig: T hedron}(a) could be set horizontal, with $z(0;~t) = 0$ for all $t \in [-\epsilon, \epsilon]$. The deformed T-hedron has the parametrization below:
\begin{equation*}
	x(i_1, ~i_2; ~t) = \begin{bmatrix}
		x_1(i_1, ~i_2; ~t) \\
		x_2(i_1, ~i_2; ~t) \\ 
		z(i_1; ~t)
	\end{bmatrix} 
\end{equation*}
\begin{equation*}
	\begin{bmatrix}
		x_1(i_1, ~i_2; ~t) \\
		x_2(i_1, ~i_2; ~t) 
	\end{bmatrix} = \sum \limits_{k=1}^{i_2} b(k) \begin{bmatrix}
		- \sin \psi(k; ~t) \\
		\cos \psi(k; ~t)
	\end{bmatrix} + a(i_1; ~t)c(i_2; ~t) \begin{bmatrix}
		\cos \phi(i_2; ~t) \\
		\sin \phi(i_2; ~t) 
	\end{bmatrix}
\end{equation*}
\begin{equation*}
	\phi(i_2; ~t) = \sum \limits_{k=1}^{i_2} \eta(k; ~t) + \sum \limits_{k=1}^{i_2} \theta(k; ~t), ~~\psi(i_2; ~t) = \sum \limits_{k=1}^{i_2} \eta(k; ~t) + \sum \limits_{k=1}^{i_2-1} \theta(k; ~t)
\end{equation*}
Note that $b(i_2)$ is irrelevant of $k$. We will see it from the analysis below. 

We will analyse how a T-hedron deforms from its projection on the horizontal planes. In Figure \ref{fig: T hedron deformation}(b) we extract four trapezoids $(i_1, ~i_2) = (0, ~0)$, $(0, ~1)$, $(1, ~0)$, $(1, ~1)$ (coloured black) from Figure \ref{fig: T hedron} and consider its deformed state (coloured red). Take quadrilateral $ABED$ as example, the key geometrical constraint is $DD'E'E$ preserves to be a rectangle. Further, $BCFE$ will deform to $BCF'E'$ under this constraint: $EE'F'F$ preserves to be a rectangle after a possible rotation and translation of $BCF'E'$. Here, $BCHG$ is rotated from $BCF'E'$, and $EFHG$ preserves to be a rectangle. From how the deformed state is generated, we could see that, $b(i_2)$ preserves in the motion for all $i_2$. Now consider the first row of quadrilaterals, we could list $i_2+1$ equations from the crease lengths of the T-hedron being preserved for all $t \in [-\epsilon, \epsilon]$. Note that $a(0; ~t) = 0$ for all $t \in [-\epsilon, \epsilon]$, but we still write it for consistency with the other rows of quadrilaterals.
\begin{equation*}
	\begin{dcases}
		\begin{split}
			& \quad (a(1; ~0)-a(0; ~0))^2 + (z(1; ~0)-z(0; ~0))^2 \\ & = (a(1; ~t)-a(0; ~t))^2 + (z(1;~t)-z(0;~t))^2 
		\end{split} \\
		\begin{split}
			& \quad (a(1; ~0)-a(0; ~0))^2c^2(1; ~0) + (z(1; ~0)-z(0; ~0))^2 \\ &= (a(1; ~t)-a(0; ~t))^2c^2(1; ~t) + (z(1;~t)-z(0;~t))^2 
		\end{split} \\
		\cdots \\
		\begin{split}
			& \quad (a(1; ~0)-a(0; ~0))^2c^2(i_2; ~0) + (z(1; ~0)-z(0; ~0))^2 \\ &= (a(1; ~t)-a(0; ~t))^2c^2(i_2; ~t) + (z(1;~t)-z(0;~t))^2
		\end{split}
	\end{dcases}
\end{equation*}

\begin{figure}[tbph]
	\noindent \begin{centering}
		\includegraphics[width=1\linewidth]{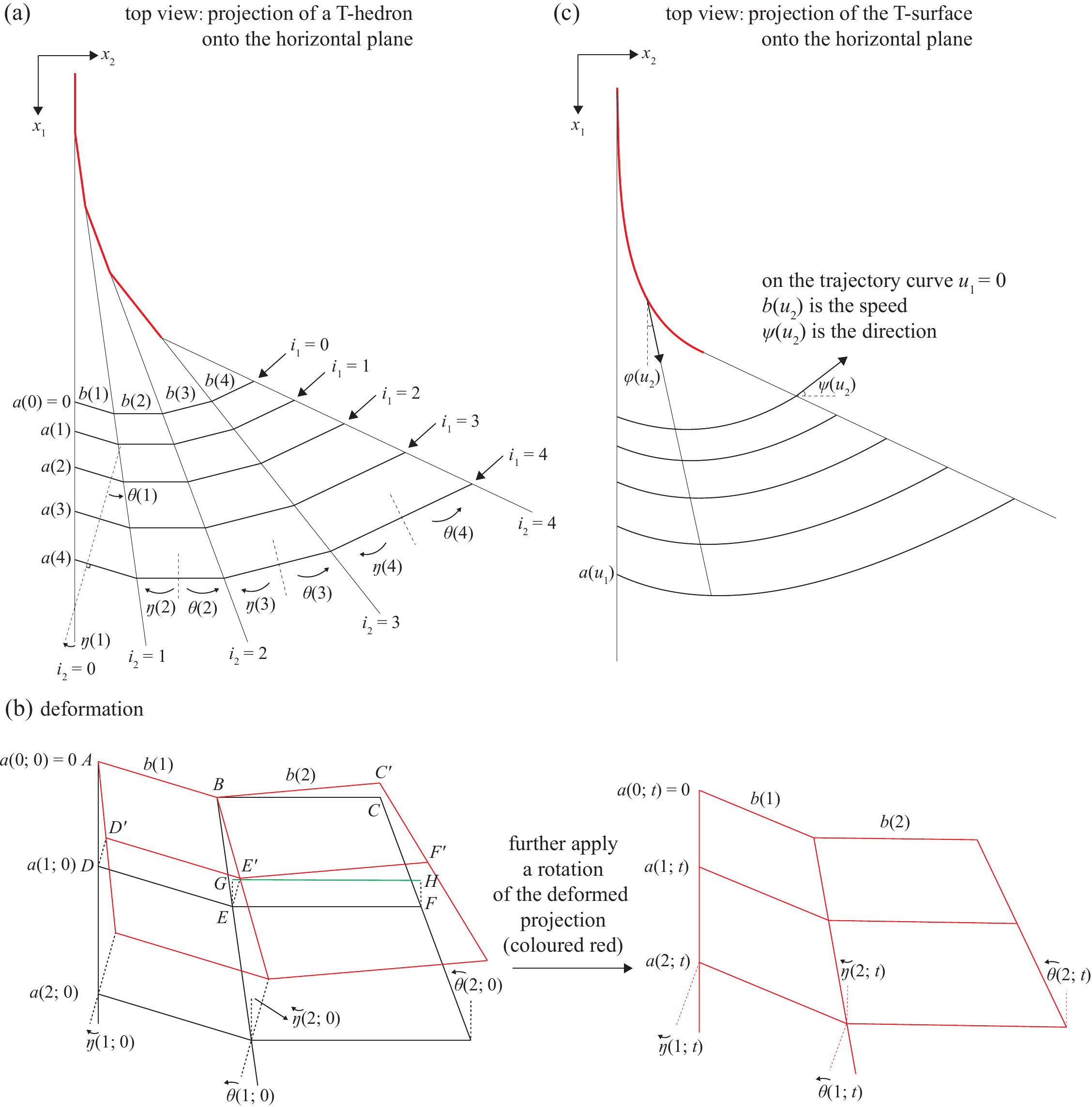}
		\par\end{centering}
	
	\caption{\label{fig: T hedron deformation} An illustration to the deformation of a T-hedron from its projection onto the horizontal plane. We let the horizontal plane remain horizontal and the vertex $x(0, ~0)$ fixed -- consequently $a(0;~t) = 0$ for all $t$. (b) shows how the trapezoid projection deforms from (a). (c) depicts the projection of the corresponding T-surface onto the horizontal plane.}
\end{figure}

From these equations, we could draw a parametrization starting with:
\begin{equation}
	a(1;~t) = a(1;~0) \sqrt{1 + t}
\end{equation}
Consequently,
\begin{equation}
	c^2(i_2; ~t) = \dfrac{ c^2(i_2;~0) + t}{1 + t}, ~~\mathrm{for~all~}i_2
\end{equation}
Further from Figure \ref{fig: T hedron deformation}(b):
\begin{equation*}
	\begin{gathered}
		a(1; ~0) c(i_2-1;~0) \sin \eta(i_2; ~0) = a(1; ~t) c(i_2-1;~t) \sin \eta(i_2;~t) \\
		a(1; ~0) c(i_2;~0) \sin \theta(i_2;~0) = a(1;~t) c(i_2;~t) \sin \theta(i_2;~t)
	\end{gathered}
\end{equation*}
that is to say
\begin{equation*}
	\begin{gathered}
		\dfrac{\sin \eta(i_2; ~t)}{\sin \eta(i_2;~0)} = \dfrac{a(1; ~0) c(i_2-1;~0)}{a(1; ~t) c(i_2-1;~t)} = \dfrac{c(i_2-1;~0)}{\sqrt{c^2(i_2-1;~0) + t}} \\
		\dfrac{\sin \theta(i_2;~t)}{\sin \theta(i_2;~0)} = \dfrac{a(1; ~0) c(i_2;~0)}{a(1; ~t) c(i_2;~t)} = \dfrac{c(i_2;~0)}{\sqrt{c^2(i_2;~0) + t}} \\
	\end{gathered}
\end{equation*}
Next we will calculate $z(i_1; ~t)$, from the second row of quadrilateral:
\begin{equation*}
	\begin{dcases}
		\begin{split}
			& \quad (a(i_1; ~0)-a(i_1-1; ~0))^2 + (z(i_1; ~0)-z(i_1-1; ~0))^2 \\ & = (a(i_1; ~t)-a(i_1-1; ~t))^2 + (z(i_1;~t)-z(i_1-1;~t))^2
		\end{split} \\
		\begin{split}
			& \quad (a(i_1; ~0)-a(i_1-1; ~0))^2c^2(1; ~0) + (z(i_1; ~0)-z(i_1-1; ~0))^2 \\ & = (a(i_1; ~t)-a(i_1-1; ~t))^2c^2(1; ~t) + (z(i_1;~t)-z(i_1-1;~t))^2 \\
		\end{split} \\
		\cdots \\
		\begin{split}
			& \quad	(a(i_1; ~0)-a(i_1-1; ~0))^2c^2(i_2; ~0) + (z(i_1; ~0)-z(i_1-1; ~0))^2 \\ & = (a(i_1; ~t)-a(i_1-1; ~t))^2c^2(i_2; ~t) + (z(i_1;~t)-z(i_1-1;~t))^2 
		\end{split}
	\end{dcases}
\end{equation*}
The result is an iterative expression for $z(i_1;~t)$:
\begin{equation}
	a(i_1;~t) = a(i_1;~0) \sqrt{1+t} \mathrm{~~for~all~} i_1 
\end{equation}
\begin{equation*}
	\begin{gathered}
		z^2(1;~t) = z^2(1; ~0) - t (a(1;~0)-a(0;~0))^2 \\
		(z(2;~t)-z(1;~t))^2 = (z(2;~0)-z(1;~0))^2 - t (a(2;~0)-a(1;~0))^2 \\
		(z(3;~t)-z(2;~t))^2 = (z(3;~0)-z(2;~0))^2 - t (a(3;~0)-a(2;~0))^2 \\
		\cdots \\
		(z(i_1;~t)-z(i_1-1;~t))^2 = (z(i_1;~0)-z(i_1-1;~0))^2 - t (a(i_1;~0)-a(i_1-1;~0))^2 \\
	\end{gathered} 
\end{equation*}
Calculate the square root of every equation above and sum them up, we could obtain:
\begin{equation}
	\begin{gathered}
		z(i_1;~t) =  \sum \limits_{j = 1}^{i_1} \mathrm{sign} \left( \triangle z(j;~0) \right) \sqrt{(\triangle z(j;~0))^2 - t (\triangle a(j;~0))^2} \\
		\triangle z(j;~0) = z(j;~0)-z(j-1;~0), ~~\triangle a(j;~0) = a(j;~0)-a(j-1;~0)
	\end{gathered}
\end{equation}

Now we are in a good position to discuss T-surface. A \textit{T-surface} is a conjugate net where 1) each coordinate curve $u_1 = \mathrm{Const}, ~u_2 = \mathrm{Const}$ is coplanar; 2) each plane $u_1 = \mathrm{Const}$ is orthogonal to $u_2 = \mathrm{Const}$. Immediately, from statement [4] of Proposition \ref{prop: T-hedron properties} we could know that either all planes $u_1 = \mathrm{Const}$ are parallel to each other, or all planes $u_2 = \mathrm{Const}$ are parallel to each other.

As depicted in Figure \ref{fig: T hedron deformation}(c), a T-surface has the following parametrization, which corresponds well with a T-hedron in Figure \ref{fig: T hedron deformation}(a). Corresponding to \eqref{eq: T-hedra 1} and \eqref{eq: T-hedra 2}: 
\begin{equation} \label{eq: T-surface 1}
	x(u_1, ~u_2) = \begin{bmatrix}
		x_1(u_1, ~u_2) \\
		x_2(u_1, ~u_2) \\ 
		x_3(u_1)
	\end{bmatrix} 
\end{equation}
\begin{equation} \label{eq: T-surface 2}
	\begin{aligned}
		\begin{bmatrix}
			x_1(u_1, ~u_2) \\
			x_2(u_1, ~u_2) 
		\end{bmatrix} & = 	\begin{bmatrix}
			x_1(0, ~u_2) \\
			x_2(0, ~u_2) 
		\end{bmatrix} + a(u_1) c(u_2) \begin{bmatrix}
			\cos \phi(u_2) \\
			\sin \phi(u_2) 
		\end{bmatrix} \\
		& = \int \limits_{0}^{u_2} b(v) \begin{bmatrix}
			- \sin \psi(v) \\
			\cos \psi(v)
		\end{bmatrix} + a(u_1) c(u_2) \begin{bmatrix}
			\cos \phi(u_2) \\
			\sin \phi(u_2) 
		\end{bmatrix}
	\end{aligned} 
\end{equation}
Here the dataset $\phi(u_2), ~\psi(u_2), ~a(u_1), ~b(u_2), ~z(u_1)$ are exactly the smooth analogue of $\phi(i_2), ~\psi(i_2), ~a(i_1), ~b(i_2), ~z(i_1)$. Specifically, consider the projection onto the horizontal plane, $\phi(u_2)$ is the direction of line $[x_1(u_1; ~u_2);~ x_2(u_1; ~u_2)]$, $\psi(u_2)$ and $b(u_2)$ are the direction and length of the tangent vector along the trajectory curve $[x_1(0; ~u_2); ~x_2(0; ~u_2)]$, $a(u_1) = x_1(u_1; ~0)$, $c(u_2)$ is the length expansion ratio at position $u_2$ compared to the $x_1$ axis. Further, $\eta(u_2) = \theta(u_2) = \phi(u_2)-\psi(u_2)$.

It is straightforward to examine the geometrical properties of a T-surface from \eqref{eq: T-surface 1} and \eqref{eq: T-surface 2}. Further,
\begin{equation*}
	\dfrac{\partial x}{\partial u_1} = \begin{bmatrix}
		\dfrac{\dif a}{\dif u_1} c \cos \phi \\[8pt]
		\dfrac{\dif a}{\dif u_1} c \sin \phi \\[8pt]
		\dfrac{\dif z}{\dif u_1}
	\end{bmatrix}, ~~ \dfrac{\partial x}{\partial u_2} = \begin{bmatrix}
		-b \sin \psi + a \left( \dfrac{\dif c}{\dif u_2} \cos \phi - c\sin \phi \dfrac{\dif \phi}{\dif u_2} \right) \\[8pt]
		b \cos \psi + a \left( \dfrac{\dif c}{\dif u_2} \sin \phi + c\cos \phi \dfrac{\dif \phi}{\dif u_2} \right) \\[8pt]
		0
	\end{bmatrix}
\end{equation*}
\begin{equation*}
	\dfrac{\partial^2 x}{\partial u_1 \partial u_2} = \begin{bmatrix}
		\dfrac{\dif a}{\dif u_1} \left( \dfrac{\dif c}{\dif u_2} \cos \phi - c\sin \phi \dfrac{\dif \phi}{\dif u_2} \right) \\[8pt]
		\dfrac{\dif a}{\dif u_1} \left( \dfrac{\dif c}{\dif u_2} \sin \phi + c\cos \phi \dfrac{\dif \phi}{\dif u_2} \right) \\[8pt]
		0
	\end{bmatrix}
\end{equation*}
\begin{equation*}
	\dfrac{\partial x}{\partial u_1} \times \dfrac{\partial x}{\partial u_2} = \begin{bmatrix}
		- \dfrac{\dif z}{\dif u_1} \left( 	b \cos \psi + a \left( \dfrac{\dif c}{\dif u_2} \sin \phi + c\cos \phi \dfrac{\dif \phi}{\dif u_2} \right) \right) \\[8pt]
		\dfrac{\dif z}{\dif u_1} \left(-b \sin \psi + a \left( \dfrac{\dif c}{\dif u_2} \cos \phi - c\sin \phi \dfrac{\dif \phi}{\dif u_2} \right)\right) \\[8pt]
		\star ~~(\mathrm{not~important~in~further~calculation})
	\end{bmatrix} 
\end{equation*}
The condition of conjugate net, $\mathrm{I}_{12} = 0$, means that:
\begin{equation} \label{eq: T-surface conjugate pre}
	\begin{gathered}
		\begin{bmatrix}
			- \sin \psi \\
			\cos \psi
		\end{bmatrix} \mathrm{~is~parallel~to~} \dfrac{\dif }{\dif u_2} \begin{bmatrix}
			c \cos \phi \\
			c \sin \phi
		\end{bmatrix} = \begin{bmatrix}
			\dfrac{\dif c}{\dif u_2} \cos \phi - c\sin \phi \dfrac{\dif \phi}{\dif u_2} \\[8pt]
			\dfrac{\dif c}{\dif u_2} \sin \phi + c\cos \phi \dfrac{\dif \phi}{\dif u_2} \\[8pt]
		\end{bmatrix} \\
		\Rightarrow ~~ \dfrac{\dif c}{\dif u_2} \cos(\phi - \psi) - c \dfrac{\dif \phi}{\dif u_2} \sin(\phi - \psi) = 0
	\end{gathered}
\end{equation}
Use $\eta(u_2) = \phi(u_2) - \psi(u_2)$, \eqref{eq: T-surface conjugate pre} is equivalent to
\begin{equation} \label{eq: T-surface conjugate}
	c \dfrac{\dif \phi}{\dif u_2} \tan \eta =  \dfrac{\dif c}{\dif u_2} 
\end{equation}
To summarize, a T-surface is described by \eqref{eq: T-surface 1} and \eqref{eq: T-surface 2}, upon the conjugate net condition \eqref{eq: T-surface conjugate}. As a construction method of a T-surface, one can specify a trajectory curve which provides the information of $[x_1(0, ~u_2); ~x_2(0, ~u_2)]$, and a profile curve which provides the information of $a(u_1) c(0) [\cos \phi(0);~ \sin \phi(0)]$ and $x_3(u_1)$. Further, with the information of $\phi(u_2)$ and $\psi(u_2)$, \eqref{eq: T-surface conjugate} is an ordinary differential equation for $c(u_2)$ with initial condition $c(0)$. The solution is listed below up to a constant factor, which can be determined from the given profile curve:
\begin{equation*}
	c(u_2) = c(0) \mathrm{exp} \left(\int \limits_{0}^{u_2} \dfrac{\dif \phi}{\dif u_2} (v) \tan \eta(v) \dif v \right)
\end{equation*}   

The special cases of a T-surface include 1) the molding surface: The condition that every trapezoid is isosceles and every trapezoid have same sector angles translate to the ratio $c(u_2) = \mathrm{Const}$. From \eqref{eq: T-surface conjugate}, $\eta(u_2) = 0$ for all $u_2$. 2) the axial surface, the trajectory curve at $u_1 = 0$, i.e., $[x_1(0; ~u_2); ~x_2(0; ~u_2)]$ degenerates to a single point, which means the speed $b(u_2) = 0$ for all $u_2$. 3) surface of revolution: being both a molding surface and an axial surface; 4) translational surface: the the trajectory curve at $u_1 = 0$, i.e., $[x_1(0; ~u_2); ~x_2(0; ~u_2)]$ degenerates to a single point at infinity, which means $\phi(u_2) = \mathrm{Const}$ for all $u_2$. From \eqref{eq: T-surface conjugate}, $c(u_2) = \mathrm{Const}$ for all $u_2$. The above information means that translational T-surface is a scanned surface -- a profile curve scanning along a trajectory curve.

The deformation of a T-surface resembles the deformation of a T-hedron, in the form of:
\begin{equation*}
	x(u_1, ~u_2; ~t) = \begin{bmatrix}
		x_1(u_1, ~u_2; ~t) \\
		x_2(u_1, ~u_2; ~t) \\ 
		z(u_1; ~t)
	\end{bmatrix} 
\end{equation*}
\begin{equation*}
	\begin{bmatrix}
		x_1(u_1, ~u_2; ~t) \\
		x_2(u_1, ~u_2; ~t) 
	\end{bmatrix} = \int \limits_{0}^{u_2} b(v) \begin{bmatrix}
		- \sin \psi(v; ~t) \\
		\cos \psi(v; ~t)
	\end{bmatrix} + a(u_1; ~t)c(u_2; ~t) \begin{bmatrix}
		\cos \phi(u_2; ~t) \\
		\sin \phi(u_2; ~t) 
	\end{bmatrix}
\end{equation*}
\begin{equation*}
	a(u_1; ~t) = a(u_1; ~0) \sqrt{1+t}
\end{equation*}
\begin{equation*}
	z(u_1;~t) = \int \limits_{0}^{u_1} \mathrm{sign}\left( \dfrac{\partial z}{\partial u_1}(u_1;~0) \right) \sqrt{\left( \dfrac{\partial z}{\partial u_1} (u_1;~0) \right)^2 - t \left(\dfrac{\partial a}{\partial u_1} (u_1;~0) \right)^2}
\end{equation*}
\begin{equation*}
	c(u_2;~t) = \dfrac{c^2(u_2;~0) + t}{1 + t}
\end{equation*}
In the discrete T-hedron we derive the expressions of deformation from the preserved crease length, in the smooth T-surface it is translated into the first fundamental form preserves. In combination with the parallel condition, $\mathrm{I}_{22} = 0$ for all $t$ if and only if
\begin{equation*}
	a \dfrac{\dif }{\dif u_2} \begin{bmatrix}
		c \cos \phi \\
		c \sin \phi
	\end{bmatrix} \mathrm{~has~a~constant~norm}
\end{equation*}
This condition is written as:
\begin{equation*}
	\begin{split}
		& \quad a^2(u_1;~t) \left( \dfrac{\partial c}{\partial u_2} (u_2;~t) \cos \phi(u_2;~t) - c(u_2;~t)\sin \phi(u_2;~t) \dfrac{\partial \phi}{\partial u_2} (u_2;~t) \right)^2 \\
		& + a^2(u_1;~t) \left( \dfrac{\partial c}{\partial u_2}(u_2;~t) \sin \phi(u_2;~t) + c(u_2;~t)\cos \phi(u_2;~t) \dfrac{\partial \phi}{\partial u_2}(u_2;~t) \right)^2 \\
		& = a^2(u_1;~0) \left( \dfrac{\partial c}{\partial u_2}(u_2;~0) \cos \phi(u_2;~0) - c(u_2;~0)\sin \phi(u_2;~0) \dfrac{\partial \phi}{\partial u_2}(u_2;~0) \right)^2 \\
		& + a^2(u_1;~0) \left( \dfrac{\partial c}{\partial u_2}(u_2;~0) \sin \phi(u_2;~0) + c(u_2;~0)\cos \phi(u_2;~0) \dfrac{\partial \phi}{\partial u_2}(u_2;~0) \right) ^2 
	\end{split}
\end{equation*}
which means
\begin{equation*}
	\begin{split}
		& \quad (1+t) \left(\left( \dfrac{\partial c}{\partial u_2}(u_2;~t) \right)^2 + c^2(u_2;~t) \left(\dfrac{\partial \phi}{\partial u_2}(u_2;~t) \right)^2 \right)  \\ & = \left( \dfrac{\partial c}{\partial u_2}(u_2;~0) \right)^2 + c^2(u_2;~0) \left(\dfrac{\partial \phi}{\partial u_2}(u_2;~0) \right)^2
	\end{split}
\end{equation*}
and we could obtain:
\begin{equation*}
	\dfrac{\partial \phi}{\partial u_2}(u_2;~t) = \dfrac{\sqrt{ c^4(u_2;~0) \left(\dfrac{\partial \phi}{\partial u_2}(u_2;~0) \right)^2 + t\left(c^2(u_2;~0) \left(\dfrac{\partial \phi}{\partial u_2}(u_2;~0) \right)^2 + \left(\dfrac{\partial c}{\partial u_2} (u_2;~0)\right)^2 \right)}}{c^2(u_2;~0)+t}
\end{equation*}
\begin{equation}
	\phi(u_2;~t) = \int \limits_{0}^{u_2} \dfrac{\partial \phi}{\partial v}(v;~t) \dif v
\end{equation}
The parallel condition \eqref{eq: T-surface conjugate} means:
\begin{equation*}
	\tan \eta(u_2;~t) = \dfrac{\dfrac{\partial c}{\partial u_2}(u_2;~0)}{c(u_2;~t)\dfrac{\partial \phi}{\partial u_2}(u_2;~t)} = \dfrac{c(u_2;~t) \sin \eta(u_2;~0)}{\sqrt{c^2(u_2;~t) \cos^2 \eta(u_2;~0) + t}}
\end{equation*}
and finally
\begin{equation*}
	\psi(u_2;~t) = \phi(u_2;~t) - \eta(u_2;~t)
\end{equation*}
It can be examined that the other two coefficients $\mathrm{I}_{11}, ~\mathrm{I}_{12}$ preserve for all $t$.

\section{Linear couplings of two-vertex systems} \label{section: linear coupling}

In Figure \ref{fig: flexibility condtion}(d), we described a two-vertex system from $y_{11}$ to $y_{21}$. A \textit{linear coupling} of such a degree-4 two-vertex system means $y_{21} = c_yy_{11}$ and $w_{21} = c_ww_{11}$ for all input $y_{11}$ (labelling provided in Figure \ref{fig: linear coupling}.). $c_y$ and $c_w$ are real coefficients dependent on sector angles. Changing $\gamma_{21} \rightarrow \pi - \gamma_{21}$, $\delta_{21} \rightarrow \pi - \delta_{21}$ leads to another form of linear dependence $y_{21} = c_y'y_{11}^{-1}$ and $w_{21} = c_w'w_{11}^{-1}$. Two linear couplings can stitch together to form a flexible Kokotsakis quadrilateral if they have the same linear dependence $c_y$ or $c_w$. The derivation of the linear dependence is based on the complexified configuration space for a degree-4 vertex as presented in \cite{he_real_2023}.

\begin{description}
	\item [isogram] (the non-self-intersecting branch)
	\begin{equation*}
		\begin{gathered}
			\dfrac{y_{21}}{y_{11}} = \dfrac{\cos \dfrac{\alpha_{21}+\beta_{21}}{2} \cos \dfrac{\alpha_{11}-\beta_{11}}{2} }{\cos \dfrac{\alpha_{21}-\beta_{21}}{2} \cos \dfrac{\alpha_{11}+\beta_{11}}{2} } , ~ \dfrac{w_{21}}{w_{11}} = \dfrac{y_{21}}{y_{11}} ~(\mathrm{let~} \alpha \rightarrow \beta, ~ \beta \rightarrow \alpha) \\ \dfrac{y_{31}}{y_{21}} = \dfrac{y_{21}}{y_{11}} ~(\mathrm{let~} \alpha \rightarrow \beta, ~ \beta \rightarrow \alpha), ~ \dfrac{w_{31}}{w_{21}} = \dfrac{y_{21}}{y_{11}} \\
		\end{gathered}
	\end{equation*}
	\item [anti-isogram] 
	\begin{equation*}
		\begin{aligned}
			\dfrac{y_{21}}{y_{11}} = & ~~ \dfrac{\sin \dfrac{\alpha_{21}-\beta_{21}}{2} \sin \dfrac{\alpha_{11}+\beta_{11}}{2}}{\sin \dfrac{\alpha_{21}+\beta_{21}}{2} \sin \dfrac{\alpha_{11}-\beta_{11}}{2}} ~~\mathrm{or}~~ \dfrac{\cos \dfrac{\alpha_{21}-\beta_{21}}{2} \cos \dfrac{\alpha_{11}+\beta_{11}}{2}}{\cos \dfrac{\alpha_{21}+\beta_{21}}{2} \cos \dfrac{\alpha_{11}-\beta_{11}}{2}} \\
			\mathrm{or~~} & \dfrac{\sin \dfrac{\alpha_{21}-\beta_{21}}{2} \cos \dfrac{\alpha_{11}+\beta_{11}}{2}}{\sin \dfrac{\alpha_{21}+\beta_{21}}{2} \cos \dfrac{\alpha_{11}-\beta_{11}}{2}} ~~\mathrm{or}~~ \dfrac{\cos \dfrac{\alpha_{21}-\beta_{21}}{2} \sin \dfrac{\alpha_{11}+\beta_{11}}{2}}{\cos \dfrac{\alpha_{21}+\beta_{21}}{2} \sin \dfrac{\alpha_{11}-\beta_{11}}{2}} \\
		\end{aligned}
	\end{equation*}
	\begin{equation*}
		\begin{gathered}
			\dfrac{w_{21}}{w_{11}} = \dfrac{y_{21}}{y_{11}} ~~\mathrm{or}~~ \dfrac{y_{21}}{y_{11}} ~~\mathrm{or}~~ -\dfrac{y_{21}}{y_{11}} ~~\mathrm{or}~~ -\dfrac{y_{21}}{y_{11}} ~(\mathrm{let~} \alpha \rightarrow \beta, ~ \beta \rightarrow \alpha) \\
			\dfrac{y_{31}}{y_{21}} = \dfrac{y_{21}}{y_{11}} ~~\mathrm{or}~~ \dfrac{y_{21}}{y_{11}} ~~\mathrm{or}~~ -\dfrac{y_{21}}{y_{11}} ~~\mathrm{or}~~ -\dfrac{y_{21}}{y_{11}} ~(\mathrm{let~} \alpha \rightarrow \pi - \beta, ~ \beta \rightarrow \pi - \alpha) \\
			\dfrac{w_{31}}{w_{21}} = \dfrac{y_{21}}{y_{11}} ~~\mathrm{or}~~ \dfrac{y_{21}}{y_{11}} ~~\mathrm{or}~~ \dfrac{y_{21}}{y_{11}} ~~\mathrm{or}~~ \dfrac{y_{21}}{y_{11}} ~(\mathrm{let~} \alpha \rightarrow \pi - \alpha, ~ \beta \rightarrow \pi - \beta) \\
		\end{gathered}
	\end{equation*}
	\item [deltoid I]
	The linear relation $y_{21} = c y_{11}$ holds when vertices 11 and 21 form an involutive coupling, i.e., they share the same involution factor.
	\begin{equation*}
		\begin{gathered}
			\lambda^x_{11} = \lambda^x_{21}, ~ \lambda = \dfrac{\sin (\beta + \alpha)}{\sin (\beta - \alpha)}    \Leftrightarrow \dfrac{\tan \alpha_{11} + \tan \beta_{11}}{\tan \alpha_{11} - \tan \beta_{11}} = \dfrac{\tan \alpha_{21} + \tan \beta_{21}}{\tan \alpha_{21} - \tan \beta_{21}} \Leftrightarrow \dfrac{\tan \alpha_{11}}{\tan \beta_{11}} = \dfrac{\tan \alpha_{21}}{\tan \beta_{21}}
		\end{gathered}	
	\end{equation*}
	Since
	\begin{equation*}
		\dfrac{2\sin \alpha}{\sin(\beta-\alpha)} = \dfrac{2\tan \alpha}{\sin \beta - \cos \beta \tan \alpha}  
		= \dfrac{2\tan \alpha}{\cos \beta (\tan \beta - \tan \alpha) } = \dfrac{1}{\cos \beta} \left(\dfrac{\tan \beta + \tan \alpha}{\tan \beta - \tan \alpha} -1 \right)
	\end{equation*}
	The ratios are:
	\begin{equation*}
		\begin{gathered}
			\dfrac{y_{21}}{y_{11}} = \left. \dfrac{2\sin \alpha_{11}}{\sin(\beta_{11}-\alpha_{11})} \middle/ \dfrac{2\sin \alpha_{21}}{\sin(\beta_{21}-\alpha_{21})} = \dfrac{\cos \beta_{21}}{\cos \beta_{11}} \right.,~
			\dfrac{w_{21}}{w_{11}} = \dfrac{\cos \alpha_{21}}{\cos \alpha_{11}}  ~(\mathrm{let~} \alpha \rightarrow \beta, ~ \beta \rightarrow \alpha) \\ \dfrac{y_{31}}{y_{21}} = \dfrac{\cos \beta_{21}}{\cos \beta_{11}} ,~ \dfrac{w_{31}}{w_{21}} = \dfrac{\cos \alpha_{21}}{\cos \alpha_{11}}  ~(\mathrm{let~} \alpha \rightarrow \beta, ~ \beta \rightarrow \alpha) 
		\end{gathered}
	\end{equation*}
	
	\begin{figure}[t]
		\noindent \begin{centering}
			\includegraphics[width=0.8\linewidth]{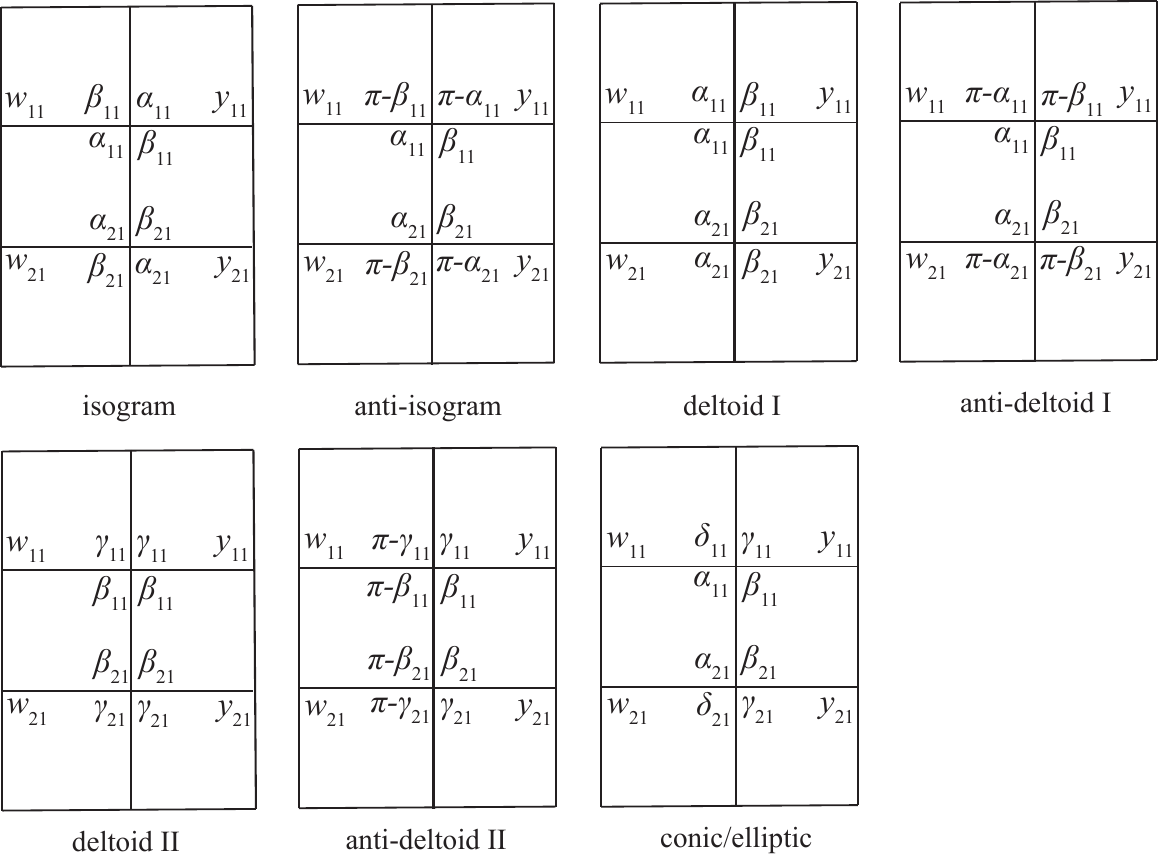}
			\par\end{centering}
		
		\caption{\label{fig: linear coupling}Labelling of linear couplings.}
	\end{figure}
	
	\item [anti-deltoid I]
	Similar to deltoid I:
	\begin{equation*}
		\dfrac{\tan \alpha_{11}}{\tan \beta_{11}} = \dfrac{\tan \alpha_{21}}{\tan \beta_{21}}
	\end{equation*}
	The ratios are:
	\begin{equation*}
		\begin{gathered}
			\dfrac{y_{21}}{y_{11}} = \dfrac{\cos \beta_{11}}{\cos \beta_{21}} , ~ \dfrac{w_{21}}{w_{11}} = \dfrac{\cos \alpha_{11}}{\cos \alpha_{21}}  ~(\mathrm{let~} \alpha \rightarrow \beta, ~ \beta \rightarrow \alpha) \\ \dfrac{y_{31}}{y_{21}} = \dfrac{\cos \beta_{11}}{\cos \beta_{21}} ~(\mathrm{let~} \alpha \rightarrow \pi-\alpha, ~ \beta \rightarrow \pi-\beta) , ~ \dfrac{w_{31}}{w_{21}} = \dfrac{\cos \alpha_{11}}{\cos \alpha_{21}}  ~(\mathrm{let~} \alpha \rightarrow \pi-\beta, ~ \beta \rightarrow \pi-\alpha) \\
		\end{gathered}
	\end{equation*}
	\item [deltoid II] The linear relation $y_{21} = c y_{11}$ holds when vertices 11 and 21 form a reducible coupling, i.e., they share the same amplitude:
	\begin{equation*}
		p^x_{11} = p^x_{21}, ~ p^x = \sqrt{\dfrac{\sin^2 \beta}{\sin^2 \gamma}-1}  ~~\Leftrightarrow~~ \dfrac{\sin \beta_{11}}{\sin \gamma_{11}} = \dfrac{\sin \beta_{21}}{\sin \gamma_{21}}
	\end{equation*}	
	The ratios are:
	\begin{equation*}
		\begin{gathered}
			\dfrac{y_{21}}{y_{11}} = \mathrm{sign} \left(\dfrac{\pi - \beta_{21} - \gamma_{21}}{\pi - \beta_{11} - \gamma_{11}} \right) \sqrt{\dfrac{\sin ( \beta_{21} + \gamma_{21}) \sin ( \beta_{11} - \gamma_{11}) }{\sin ( \beta_{21} - \gamma_{21})\sin ( \beta_{11} + \gamma_{11})}}, ~
			\dfrac{w_{21}}{w_{11}} = \dfrac{y_{21}}{y_{11}}  \\
			\dfrac{y_{31}}{y_{21}} = \dfrac{y_{21}}{y_{11}} ~(\mathrm{let~} \beta \rightarrow \gamma, ~\gamma \rightarrow \beta) , ~ \dfrac{w_{31}}{w_{21}} = \dfrac{y_{21}}{y_{11}} ~(\mathrm{let~} \beta \rightarrow \gamma, ~\gamma \rightarrow \beta) \\
		\end{gathered}
	\end{equation*}
	\item [anti-deltoid II] Similar to deltoid II:
	\begin{equation*}
		\dfrac{\sin \beta_{11}}{\sin \gamma_{11}} = \dfrac{\sin \beta_{21}}{\sin \gamma_{21}}
	\end{equation*}	
	The ratios are the same with deltoid II ($y_{21}$ and $y_{11}$ changed to its opposite):
	\begin{equation*}
		\begin{gathered}
			\dfrac{y_{21}}{y_{11}} = \mathrm{sign} \left(\dfrac{\pi - \beta_{21} - \gamma_{21}}{\pi - \beta_{11} - \gamma_{11}} \right) \sqrt{\dfrac{\sin ( \beta_{21} + \gamma_{21}) \sin ( \beta_{11} - \gamma_{11}) }{\sin ( \beta_{21} - \gamma_{21})\sin ( \beta_{11} + \gamma_{11})}} , ~ \dfrac{w_{21}}{w_{11}} = \dfrac{\cos \alpha_{11}}{\cos \alpha_{21}}  ~(\mathrm{let~} \alpha \rightarrow \beta, ~ \beta \rightarrow \alpha) \\
			\dfrac{y_{31}}{y_{21}} = \dfrac{\cos \beta_{11}}{\cos \beta_{21}} ~(\mathrm{let~} \alpha \rightarrow \pi-\alpha, ~ \beta \rightarrow \pi-\beta) ,~ \dfrac{w_{31}}{w_{21}} = \dfrac{\cos \alpha_{11}}{\cos \alpha_{21}}  ~(\mathrm{let~} \alpha \rightarrow \pi-\beta, ~ \beta \rightarrow \pi-\alpha) \\
		\end{gathered}
	\end{equation*}
	\item [conic] Two conic I or two conic IV form a linear coupling if both the amplitudes and phase shifts are equal:
	\begin{equation*}
		\begin{gathered}
			p^x_{11} = p^x_{21}, ~ p^x = \sqrt{\dfrac{\sin \alpha \sin \beta}{\sin \gamma \sin \delta}-1} , ~ \dfrac{\sin \beta_{11} \sin \delta_{11}}{\sin \alpha_{11} \sin \gamma_{11}} = \dfrac{\sin \beta_{21} \sin \delta_{21}}{\sin \alpha_{21} \sin \gamma_{21}} 
		\end{gathered}		
	\end{equation*}
	which leads to:
	\begin{equation*}
		\dfrac{\sin \delta_{11}}{\sin \alpha_{11}} = \dfrac{\sin \delta_{21}}{\sin \alpha_{21}}, ~~ \dfrac{\sin \beta_{11}}{\sin \gamma_{11}} = \dfrac{\sin \beta_{21}}{\sin \gamma_{21}}
	\end{equation*}
	The ratio for two conic I is:
	\begin{equation*}
		\dfrac{y_{21}}{y_{11}} = \mathrm{sign}\left(\dfrac{\pi - \sigma_{21}}{\pi - \sigma_{11}}\right) \dfrac{p^y_{21}}{p^y_{11}} = \left. \mathrm{sign}\left(\dfrac{\pi - \sigma_{21}}{\pi - \sigma_{11}}\right) \sqrt{\dfrac{\sin \beta_{21} \sin \gamma_{21}}{\sin \delta_{21} \sin \alpha_{21}}-1} \middle/ \sqrt{\dfrac{\sin \beta_{11} \sin \gamma_{11}}{\sin \delta_{11} \sin \alpha_{11}}-1} \right.
	\end{equation*}
	\begin{equation*}
		\dfrac{w_{21}}{w_{11}} = \left. \mathrm{sign}\left(\dfrac{\pi - \sigma_{21}}{\pi - \sigma_{11}}\right) \sqrt{\dfrac{\sin \delta_{21} \sin \alpha_{21}}{\sin \beta_{21} \sin \gamma_{21}}-1} \middle/ \sqrt{\dfrac{\sin \delta_{11} \sin \alpha_{11}}{\sin \beta_{11} \sin \gamma_{11}}-1} \right.  ~(\mathrm{let~} \alpha \rightarrow \beta, ~\beta \rightarrow \alpha, ~ \gamma \rightarrow \delta, ~ \delta \rightarrow \gamma)
	\end{equation*}
	\begin{equation*}
		\dfrac{y_{31}}{y_{21}} = \left. \mathrm{sign}\left(\dfrac{\pi - \sigma_{21}}{\pi - \sigma_{11}}\right) \sqrt{\dfrac{\sin \beta_{21} \sin \gamma_{21}}{\sin \delta_{21} \sin \alpha_{21}}-1} \middle/ \sqrt{\dfrac{\sin \beta_{11} \sin \gamma_{11}}{\sin \delta_{11} \sin \alpha_{11}}-1} \right. ~(\mathrm{let~} \alpha \rightarrow \delta, ~ \beta \rightarrow \gamma, ~\gamma \rightarrow \beta, ~\delta \rightarrow \alpha)
	\end{equation*}
	\begin{equation*}
		\dfrac{w_{31}}{w_{21}} = \left. \mathrm{sign}\left(\dfrac{\pi - \sigma_{21}}{\pi - \sigma_{11}}\right) \sqrt{\dfrac{\sin \delta_{21} \sin \alpha_{21}}{\sin \beta_{21} \sin \gamma_{21}}-1} \middle/ \sqrt{\dfrac{\sin \delta_{11} \sin \alpha_{11}}{\sin \beta_{11} \sin \gamma_{11}}-1} \right.  ~(\mathrm{let~} \alpha \rightarrow \gamma, ~ \beta \rightarrow \delta, ~\gamma \rightarrow \alpha, ~\delta \rightarrow \beta) \\
	\end{equation*}
	\begin{equation*}
		\begin{gathered}
			\sigma = \alpha + \gamma = \beta + \delta, ~~p^y = \sqrt{\dfrac{\sin \beta \sin \gamma}{\sin \delta \sin \alpha}-1}
		\end{gathered}
	\end{equation*}
	The ratio for two conic IV is:
	\begin{equation*}
		\dfrac{y_{21}}{y_{11}} = \mathrm{sign}\left(\dfrac{\pi - \sigma_{11}}{\pi - \sigma_{21}}\right) \dfrac{p^y_{11}}{p^y_{21}} = \left. \mathrm{sign}\left(\dfrac{\pi - \sigma_{11}}{\pi - \sigma_{21}}\right) \sqrt{\dfrac{\sin \beta_{11} \sin \gamma_{11}}{\sin \delta_{11} \sin \alpha_{11}}-1} \middle/ \sqrt{\dfrac{\sin \beta_{21} \sin \gamma_{21}}{\sin \delta_{21} \sin \alpha_{21}}-1} \right.
	\end{equation*}
	\begin{equation*}
		\dfrac{w_{21}}{w_{11}} = \left. \mathrm{sign}\left(\dfrac{\pi + \sigma_{11}}{\pi + \sigma_{21}}\right) \sqrt{\dfrac{\sin \delta_{11} \sin \alpha_{11}}{\sin \beta_{11} \sin \gamma_{11}}-1} \middle/ \sqrt{\dfrac{\sin \delta_{21} \sin \alpha_{21}}{\sin \beta_{21} \sin \gamma_{21}}-1} \right.  ~(\mathrm{let~} \alpha \rightarrow \beta, ~\beta \rightarrow \alpha, ~ \gamma \rightarrow \delta, ~ \delta \rightarrow \gamma)
	\end{equation*}
	\begin{equation*}
		\dfrac{y_{31}}{y_{21}} = \left. \mathrm{sign}\left(\dfrac{\pi + \sigma_{11}}{\pi + \sigma_{21}}\right) \sqrt{\dfrac{\sin \beta_{11} \sin \gamma_{11}}{\sin \delta_{11} \sin \alpha_{11}}-1} \middle/ \sqrt{\dfrac{\sin \beta_{21} \sin \gamma_{21}}{\sin \delta_{21} \sin \alpha_{21}}-1} \right. ~(\mathrm{let~} \alpha \rightarrow \delta, ~ \beta \rightarrow \gamma, ~\gamma \rightarrow \beta, ~\delta \rightarrow \alpha)
	\end{equation*}
	\begin{equation*}
		\dfrac{w_{31}}{w_{21}} = \left. \mathrm{sign}\left(\dfrac{\pi - \sigma_{11}}{\pi - \sigma_{21}}\right) \sqrt{\dfrac{\sin \delta_{11} \sin \alpha_{11}}{\sin \beta_{11} \sin \gamma_{11}}-1} \middle/ \sqrt{\dfrac{\sin \delta_{21} \sin \alpha_{21}}{\sin \beta_{21} \sin \gamma_{21}}-1} \right.  ~(\mathrm{let~} \alpha \rightarrow \gamma, ~ \beta \rightarrow \delta, ~\gamma \rightarrow \alpha, ~\delta \rightarrow \beta) \\
	\end{equation*}
	\begin{equation*}
		\begin{gathered}
			\sigma = \dfrac{-\alpha + \beta - \gamma + \delta}{2} + \pi, ~~p^y = \sqrt{\dfrac{\sin \beta \sin \gamma}{\sin \delta \sin \alpha}-1}
		\end{gathered}
	\end{equation*}
	\item [elliptic] Two elliptic vertices form a linear coupling if the elliptic modulus, amplitudes and phase shifts are equal:
	\begin{equation*}
		\begin{gathered}
			p^x_{11} = p^x_{21}, ~ p^x = \sqrt{\dfrac{\sin \alpha \sin \beta}{\sin (\sigma-\alpha) \sin (\sigma-\beta)}-1}, ~ \sigma = \dfrac{\alpha + \beta + \gamma + \delta}{2} \\ M_{11} = M_{21}, ~ M = \dfrac{\sin \alpha \sin \beta \sin \gamma \sin \delta}{\sin (\sigma-\alpha) \sin (\sigma-\beta) \sin (\sigma-\gamma) \sin(\sigma-\delta)} \\
			\dfrac{\sin \alpha_{11} \sin \gamma_{11}}{\sin (\sigma_{11} - \alpha_{11}) \sin (\sigma_{11} - \gamma_{11})} = \dfrac{\sin \alpha_{21} \sin \gamma_{21}}{\sin (\sigma_{21} - \alpha_{21}) \sin (\sigma_{21} - \gamma_{21})}
		\end{gathered}		
	\end{equation*}
	which leads to
	\begin{equation*}
		\dfrac{\sin \alpha_{11} \sin \beta_{11}}{\sin (\sigma_{11}-\alpha_{11})\sin (\sigma_{11}-\beta_{11})} = \dfrac{\sin \alpha_{21} \sin \beta_{21}}{\sin (\sigma_{21}-\alpha_{21})\sin (\sigma_{21}-\beta_{21})}
	\end{equation*}
	\begin{equation*}
		\dfrac{\sin \gamma_{11} \sin \delta_{11}}{\sin (\sigma_{11}-\gamma_{11})\sin (\sigma_{11}-\delta_{11})} = \dfrac{\sin \gamma_{21} \sin \delta_{21}}{\sin (\sigma_{21}-\gamma_{21})\sin (\sigma_{21}-\delta_{21})}
	\end{equation*}
	\begin{equation*}
		\dfrac{\sin \alpha_{11} \sin \gamma_{11}}{\sin (\sigma_{11}-\alpha_{11})\sin (\sigma_{11}-\gamma_{11})} = \dfrac{\sin \alpha_{21} \sin \gamma_{21}}{\sin (\sigma_{21}-\alpha_{21})\sin (\sigma_{21}-\gamma_{21})}
	\end{equation*}
	The ratio is:
	\begin{equation*}
		\begin{aligned}
			\dfrac{y_{21}}{y_{11}} & = \mathrm{sign}\left(\dfrac{\pi - \sigma_{21}}{\pi - \sigma_{11}}\right) \dfrac{p^y_{21}}{p^y_{11}} \\ 
			& = \left. \mathrm{sign}\left(\dfrac{\pi - \sigma_{21}}{\pi - \sigma_{11}}\right) \sqrt{\dfrac{\sin \beta_{21} \sin \gamma_{21}}{\sin (\sigma_{21} - \beta_{21}) \sin (\sigma_{21} - \gamma_{21})}-1} \middle/ \sqrt{\dfrac{\sin \beta_{11} \sin \gamma_{11}}{\sin (\sigma_{11} - \beta_{11}) \sin (\sigma_{11} - \gamma_{11})}-1} \right. 
		\end{aligned} 
	\end{equation*}
	\begin{equation*}
		\begin{gathered}
			\dfrac{w_{21}}{w_{11}} = \left. \mathrm{sign}\left(\dfrac{\pi - \sigma_{21}}{\pi - \sigma_{11}}\right) \sqrt{\dfrac{\sin \delta_{21} \sin \alpha_{21}}{\sin (\sigma_{21} - \delta_{21}) \sin (\sigma_{21} - \alpha_{21})}-1} \middle/ \sqrt{\dfrac{\sin \delta_{11} \sin \alpha_{11}}{\sin (\sigma_{11} - \delta_{11}) \sin (\sigma_{11} - \alpha_{11})}-1} \right. \\
			(\mathrm{let~} \alpha \rightarrow \beta, ~\beta \rightarrow \alpha, ~ \gamma \rightarrow \delta, ~ \delta \rightarrow \gamma)
		\end{gathered} 
	\end{equation*}
	\begin{equation*}
		\begin{gathered}
			\dfrac{y_{31}}{y_{21}} = \left. \mathrm{sign}\left(\dfrac{\pi - \sigma_{21}}{\pi - \sigma_{11}}\right) \sqrt{\dfrac{\sin \beta_{21} \sin \gamma_{21}}{\sin (\sigma_{21} - \beta_{21}) \sin (\sigma_{21} - \gamma_{21})}-1} \middle/ \sqrt{\dfrac{\sin \beta_{11} \sin \gamma_{11}}{\sin (\sigma_{11} - \beta_{11}) \sin (\sigma_{11} - \gamma_{11})}-1} \right. \\
			(\mathrm{let~} \alpha \rightarrow \delta, ~ \beta \rightarrow \gamma, ~\gamma \rightarrow \beta, ~\delta \rightarrow \alpha)
		\end{gathered} 
	\end{equation*}
	\begin{equation*}
		\begin{gathered}
			\dfrac{w_{31}}{w_{21}} = \left. \mathrm{sign}\left(\dfrac{\pi - \sigma_{21}}{\pi - \sigma_{11}}\right) \sqrt{\dfrac{\sin \delta_{21} \sin \alpha_{21}}{\sin (\sigma_{21} - \delta_{21}) \sin (\sigma_{21} - \alpha_{21})}-1} \middle/ \sqrt{\dfrac{\sin \delta_{11} \sin \alpha_{11}}{\sin (\sigma_{11} - \delta_{11}) \sin (\sigma_{11} - \alpha_{11})}-1} \right. \\
			(\mathrm{let~} \alpha \rightarrow \gamma, ~ \beta \rightarrow \delta, ~\gamma \rightarrow \alpha, ~\delta \rightarrow \beta)
		\end{gathered} 
	\end{equation*}
\end{description}

\section{Equimodular couplings of two-vertex systems} \label{section: equimodular coupling}

In Figure \ref{fig: flexibility condtion}(d), we described a two-vertex system from $y_{11}$ to $y_{21}$. An \textit{equimodular coupling} of such a degree-4 two-vertex system means that $y_{11}$ and $y_{21}$ are periodical functions over a parameter $t$ in the complexified configuration space, oscillating at the same frequency. $y_{11}$ and $y_{21}$ may differ in amplitude and phase shift. 

\begin{description}
	\item[Conic] \citep[Section 17]{he_real_2023} Vertices $_{11}$ and $_{21}$ share the same amplitude:
	\begin{equation*}
		p^{x}_{11} = p^{x}_{21}, ~ p^x=\sqrt{\dfrac{\sin \alpha \sin \beta}{\sin \gamma \sin \delta}-1}
	\end{equation*}
	This is also the condition for two conic II, III, or IV vertices to form an equimodular coupling.
	
	\item[Elliptic] \citep[Section 21]{he_real_2023} Vertices $_{11}$ and $_{21}$ share the same amplitude and moduli:
	\begin{equation*}
		\begin{gathered}
			M_{11} = M_{21}, ~ p^{x}_{11} = p^{x}_{21} \\ 
			M = \dfrac{\sin \alpha \sin \beta \sin \gamma \sin \delta}{\sin (\sigma-\alpha) \sin (\sigma-\beta) \sin (\sigma-\gamma) \sin(\sigma-\delta)} , ~ p^x=\sqrt{\dfrac{\sin \alpha \sin \beta}{\sin \gamma \sin \delta}-1}
		\end{gathered}
	\end{equation*}
\end{description}

Finally, for a conic equimodular Kokotsakis quadrilateral, the flexibility condition is:
\begin{description}
	\item[Equal amplitude] 
	\begin{equation*}
		\begin{gathered}
			p^{x}_{11} = p^{x}_{21} , ~ p^{x}_{12} = p^{x}_{22} , ~
			p^{y}_{11} = p^{y}_{12} , ~
			p^{y}_{21} = p^{y}_{22} \\
			p^x=\sqrt{\dfrac{\sin \alpha \sin \beta}{\sin \gamma \sin \delta}-1}, ~ p^y=\sqrt{\dfrac{\sin \beta \sin \gamma}{\sin \delta \sin \alpha}-1}
		\end{gathered}
	\end{equation*}
	\item[Equal phase shift]  (detailed expression in \citet[Section 17]{he_real_2023})
	\begin{equation*}
		\theta^a_{11} - \theta^a_{21} = \theta^a_{12} - \theta^a_{22}
	\end{equation*}
\end{description}
for an elliptic equimodular Kokotsakis quadrilateral, the flexibility condition is:
\begin{description}
	\item[Equal amplitude] 
	\begin{equation*}
		\begin{gathered}
			\begin{aligned}
				& M_{11} = M_{21} = M_{12} = M_{22} \\
				& p^{x}_{11} = p^{x}_{21} , ~ p^{x}_{12} = p^{x}_{22} , ~
				p^{y}_{11} = p^{y}_{12} , ~
				p^{y}_{21} = p^{y}_{22}
			\end{aligned} \\ 
			p^x=\sqrt{\dfrac{\sin \alpha \sin \beta}{\sin \gamma \sin \delta}-1}, ~ p^y=\sqrt{\dfrac{\sin \beta \sin \gamma}{\sin \delta \sin \alpha}-1}
		\end{gathered}
	\end{equation*}
	\item[Equal phase shift] (detailed expression in \citet[Section 21]{he_real_2023})
	\begin{equation*}
		\theta^a_{11} - \theta^a_{21} = \theta^a_{12} - \theta^a_{22}
	\end{equation*}
\end{description}

\section{Details for the examples in the main text} \label{section: examples}

Using the repetitive stitching method described in the main text and the information provided in Sections \ref{section: linear coupling} and \ref{section: equimodular coupling}, one can create a large library of sector-angle-periodic patterns formed by the linear couplings and equimodular couplings. Below we provide the constraints on the sector angles \textbf{within a unit} and \textbf{ensure the flexibility of the entire pattern} for the six examples presented in Figure 3 of the main text. The shape of each unit is provided in Figure \ref{fig: details for examples}. The flexibility of the entire pattern is guaranteed by the periodicity of sector angles, which ensures the flexibility of new Kokotsakis quadrilaterals generated in between units among the stitching process.

The exact solutions of each set of constraints is provided in the associated MATLAB application \citep{he_sector-angle-periodic_2024}. These numerical solutions are verified from plotting the folding motion in the 3-dimensional space. The pattern is plotted from one input folding angle, all the sector angles and uniform input crease lengths. The above parameters are fully adjustable, providing significant freedom in shaping the quad-mesh rigid origami. 

\begin{figure*}[htbp]
	\centering
	\includegraphics[width=\linewidth]{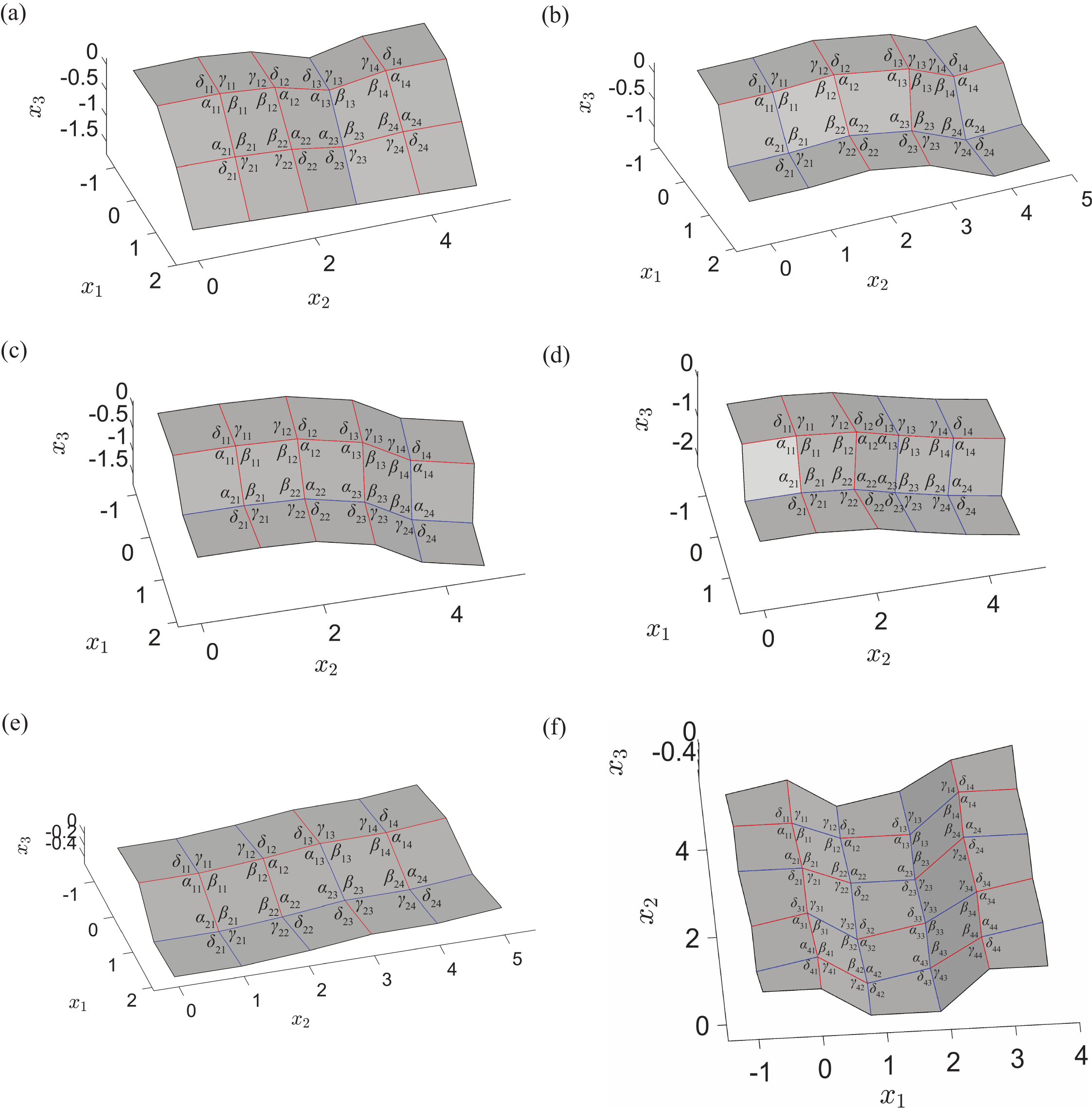}
	\caption{\label{fig: details for examples}Labelling of sector angles and shapes of single units from (a) to (f) in Figure 2 of the main text.}
\end{figure*} 

\subsection*{Example 1, Figure 3(a), Figure \ref{fig: details for examples}(a), non-developable, linear coupling}

The unit size is $3 \times 5$. A unit contains 8 interior vertices and 32 sector angles. The sector angles $\alpha_{ij}, ~\beta_{ij}, ~\gamma_{ij}, ~\delta_{ij}$, $i, ~j \in \mathbb{Z}^+, ~i \le 2, ~j \le 4$ meet the constraints below. There are 30 constraints for 32 sector angles, allowing two independent input sector angles. Columns 1 and 3 are isogram/eggbox vertices, columns 2 and 4 are deltoid I vertices.

Vertex type condition:
\begin{equation*}
	\begin{dcases}
		\gamma_{11} = \alpha_{11}, ~\delta_{11} = \beta_{11}, ~ \gamma_{12} = \beta_{12}, ~\delta_{12} =  \alpha_{12} \\
		\gamma_{13} = \alpha_{13}, ~\delta_{13} = \beta_{13}, ~ \gamma_{14} = \beta_{14}, ~\delta_{14} = \alpha_{14} \\
		\gamma_{21} = \alpha_{21}, ~\delta_{21} = \beta_{21}, ~\gamma_{22} = \beta_{22}, ~\delta_{22} = \alpha_{22} \\
		\gamma_{23} = \alpha_{23}, ~\delta_{23} = \beta_{23}, ~\gamma_{24} = \beta_{24}, ~\delta_{24} = \alpha_{24} \\
	\end{dcases}
\end{equation*} 
Planarity condition considering the periodicity of sector angles:
\begin{equation*}
	\begin{dcases}
		\beta_{11}+\beta_{21}+\beta_{12}+\beta_{12} = 2\pi, ~\gamma_{11}+\gamma_{21}+\gamma_{12}+\gamma_{12} = 2\pi \\
		\delta_{12}+\delta_{22}+\delta_{13}+\delta_{23} = 2\pi, ~\alpha_{12}+\alpha_{22}+\alpha_{13}+\alpha_{23} = 2\pi \\ \beta_{13}+\beta_{23}+\beta_{14}+\beta_{14} = 2\pi, ~\gamma_{13}+\gamma_{23}+\gamma_{14}+\gamma_{14} = 2\pi \\
		\delta_{14}+\delta_{24}+\delta_{11}+\delta_{21} = 2\pi, ~\alpha_{14}+\alpha_{24}+\alpha_{11}+\alpha_{21} = 2\pi \\
	\end{dcases}
\end{equation*}
Condition for being linear units:
\begin{equation*}
	\begin{dcases}
		\dfrac{\tan \beta_{12}}{\tan \alpha_{12}} = \dfrac{\tan \beta_{22}}{\tan \alpha_{22}} \\
		\dfrac{\tan \beta_{14}}{\tan \alpha_{14}} = \dfrac{\tan \beta_{24}}{\tan \alpha_{24}} \\
	\end{dcases}
\end{equation*}
Condition on equal ratio for linear units:
\begin{equation*}
	\begin{dcases}
		\dfrac{\cos \dfrac{\alpha_{21} + \beta_{21}}{2}\cos \dfrac{\alpha_{11} - \beta_{11}}{2}}{\cos \dfrac{\alpha_{21} - \beta_{21}}{2}\cos \dfrac{\alpha_{11} + \beta_{11}}{2}} =  \dfrac{\cos \beta_{22}}{\cos \beta_{12}} \\
		\dfrac{\cos \beta_{22}}{\cos \beta_{12}} = \dfrac{\cos \dfrac{\alpha_{23} + \beta_{23}}{2}\cos \dfrac{\alpha_{13} - \beta_{13}}{2}}{\cos \dfrac{\alpha_{23} - \beta_{23}}{2}\cos \dfrac{\alpha_{13} + \beta_{13}}{2}} \\
		\dfrac{\cos \dfrac{\alpha_{23} + \beta_{23}}{2}\cos \dfrac{\alpha_{13} - \beta_{13}}{2}}{\cos \dfrac{\alpha_{23} - \beta_{23}}{2}\cos \dfrac{\alpha_{13} + \beta_{13}}{2}} = \dfrac{\cos \beta_{24}}{\cos \beta_{14}} \\
		\dfrac{\cos \beta_{24}}{\cos \beta_{14}} = \dfrac{\cos \dfrac{\alpha_{21} + \beta_{21}}{2}\cos \dfrac{\alpha_{11} - \beta_{11}}{2}}{\cos \dfrac{\alpha_{21} - \beta_{21}}{2}\cos \dfrac{\alpha_{11} + \beta_{11}}{2}}
	\end{dcases}
\end{equation*}

\subsection*{Example 2, Figure 3(b), Figure \ref{fig: details for examples}(b), non-developable, linear coupling}

The unit size is $3 \times 5$. A unit contains 8 interior vertices and 32 sector angles. The sector angles $\alpha_{ij}, ~\beta_{ij}, ~\gamma_{ij}, ~\delta_{ij}$, $i, ~j \in \mathbb{Z}^+, ~i \le 2, ~j \le 4$ meet the constraints below. There are 29 constraints for 32 sector angles, allowing three independent input sector angles. Columns 1 and 3 are isogram/eggbox vertices, columns 2 and 4 are deltoid II vertices.

Vertex type condition:
\begin{equation*}
	\begin{dcases}
		\gamma_{11} = \alpha_{11}, ~\delta_{11} = \beta_{11}, ~ \gamma_{12} = \delta_{12}, ~\beta_{12} =  \alpha_{12} \\
		\gamma_{13} = \alpha_{13}, ~\delta_{13} = \beta_{13}, ~ \gamma_{14} = \delta_{14}, ~\beta_{14} = \alpha_{14} \\
		\gamma_{21} = \alpha_{21}, ~\delta_{21} = \beta_{21}, ~\gamma_{22} = \delta_{22}, ~\beta_{22} = \alpha_{22} \\
		\gamma_{23} = \alpha_{23}, ~\delta_{23} = \beta_{23}, ~\gamma_{24} = \delta_{24}, ~\beta_{24} = \alpha_{24} \\
	\end{dcases}
\end{equation*} 
Planarity condition considering the periodicity of sector angles:
\begin{equation*}
	\begin{dcases}
		\beta_{11}+\beta_{21}+\beta_{12}+\beta_{12} = 2\pi, ~\gamma_{11}+\gamma_{21}+\gamma_{12}+\gamma_{12} = 2\pi \\
		\delta_{12}+\delta_{22}+\delta_{13}+\delta_{23} = 2\pi, ~\alpha_{12}+\alpha_{22}+\alpha_{13}+\alpha_{23} = 2\pi \\ \beta_{13}+\beta_{23}+\beta_{14}+\beta_{14} = 2\pi, ~\gamma_{13}+\gamma_{23}+\gamma_{14}+\gamma_{14} = 2\pi \\
		\delta_{14}+\delta_{24}+\delta_{11}+\delta_{21} = 2\pi, ~\alpha_{14}+\alpha_{24}+\alpha_{11}+\alpha_{21} = 2\pi \\
	\end{dcases}
\end{equation*}
Condition for being linear units:
\begin{equation*}
	\begin{dcases}
		\dfrac{\sin \beta_{12}}{\sin \gamma_{12}} = \dfrac{\sin \beta_{22}}{\sin \gamma_{22}} \\
		\dfrac{\sin \beta_{14}}{\sin \gamma_{14}} = \dfrac{\sin \beta_{24}}{\sin \gamma_{24}} \\
	\end{dcases}
\end{equation*}
Condition on equal ratio for linear units:
\begin{equation*}
	\begin{dcases}
		\dfrac{\cos \dfrac{\alpha_{21} + \beta_{21}}{2}\cos \dfrac{\alpha_{11} - \beta_{11}}{2}}{\cos \dfrac{\alpha_{21} - \beta_{21}}{2}\cos \dfrac{\alpha_{11} + \beta_{11}}{2}} = \mathrm{sign}\left(\dfrac{\pi-\beta_{22}-\gamma_{22}}{\pi-\beta_{12}-\gamma_{12}}\right) \sqrt{\dfrac{\sin(\beta_{22}+\gamma_{22})\sin(\beta_{12}-\gamma_{12})}{\sin(\beta_{22}-\gamma_{22})\sin(\beta_{12}+\gamma_{12})}} \\
		\mathrm{sign} \left(\dfrac{\pi-\beta_{22}-\gamma_{22}}{\pi-\beta_{12}-\gamma_{12}}\right) \sqrt{\dfrac{\sin(\beta_{22}+\gamma_{22})\sin(\beta_{12}-\gamma_{12})}{\sin(\beta_{22}-\gamma_{22})\sin(\beta_{12}+\gamma_{12})}} = \dfrac{\cos \dfrac{\alpha_{23} + \beta_{23}}{2}\cos \dfrac{\alpha_{13} - \beta_{13}}{2}}{\cos \dfrac{\alpha_{23} - \beta_{23}}{2}\cos \dfrac{\alpha_{13} + \beta_{13}}{2}} \\
		\dfrac{\cos \dfrac{\alpha_{23} + \beta_{23}}{2}\cos \dfrac{\alpha_{13} - \beta_{13}}{2}}{\cos \dfrac{\alpha_{23} - \beta_{23}}{2}\cos \dfrac{\alpha_{13} + \beta_{13}}{2}} = \mathrm{sign} \left(\dfrac{\pi-\beta_{24}-\gamma_{24}}{\pi-\beta_{14}-\gamma_{14}}\right) \sqrt{\dfrac{\sin(\beta_{24}+\gamma_{24})\sin(\beta_{14}-\gamma_{14})}{\sin(\beta_{24}-\gamma_{24})\sin(\beta_{14}+\gamma_{14})}}
	\end{dcases}
\end{equation*}
Note that the equation below can be inferred from the above three equations for equal ratio:
\begin{equation*}
	\mathrm{sign} \left(\dfrac{\pi-\beta_{24}-\gamma_{24}}{\pi-\beta_{14}-\gamma_{14}}\right) \sqrt{\dfrac{\sin(\beta_{24}+\gamma_{24})\sin(\beta_{14}-\gamma_{14})}{\sin(\beta_{24}-\gamma_{24})\sin(\beta_{14}+\gamma_{14})}} = \dfrac{\cos \dfrac{\alpha_{21} + \beta_{21}}{2}\cos \dfrac{\alpha_{11} - \beta_{11}}{2}}{\cos \dfrac{\alpha_{21} - \beta_{21}}{2}\cos \dfrac{\alpha_{11} + \beta_{11}}{2}}
\end{equation*}

\subsection*{Example 3, Figure 3(c), Figure \ref{fig: details for examples}(c), non-developable, linear coupling}

The unit size is $3 \times 5$. A unit contains 8 interior vertices and 32 sector angles. The sector angles $\alpha_{ij}, ~\beta_{ij}, ~\gamma_{ij}, ~\delta_{ij}$, $i, ~j \in \mathbb{Z}^+, ~i \le 2, ~j \le 4$ meet the constraints below. There are 30 constraints for 32 sector angles, allowing two independent input sector angles. Columns 1 and 3 are deltoid I vertices, columns 2 and 4 are conic I vertices.

Vertex type condition:
\begin{equation*}
	\begin{dcases}
		\alpha_{11} = \delta_{11}, ~\beta_{11} = \gamma_{11}, ~\alpha_{13} = \delta_{13}, ~\beta_{13} = \gamma_{13} \\
		\alpha_{21} = \delta_{21}, ~\beta_{21} = \gamma_{21}, ~\alpha_{23} = \delta_{23}, ~\beta_{23} = \gamma_{23} \\
		\alpha_{12} + \gamma_{12} = \beta_{12} + \gamma_{12}, ~ \alpha_{14} + \gamma_{14} = \beta_{14} + \gamma_{14}\\
		\alpha_{22} + \gamma_{22} = \beta_{22} + \gamma_{22}, ~ \alpha_{24} + \gamma_{24} = \beta_{24} + \gamma_{24}\\
	\end{dcases}
\end{equation*} 
Planarity condition considering the periodicity of sector angles:
\begin{equation*}
	\begin{dcases}
		\beta_{11}+\beta_{21}+\beta_{12}+\beta_{12} = 2\pi, ~\gamma_{11}+\gamma_{21}+\gamma_{12}+\gamma_{12} = 2\pi \\
		\delta_{12}+\delta_{22}+\delta_{13}+\delta_{23} = 2\pi, ~\alpha_{12}+\alpha_{22}+\alpha_{13}+\alpha_{23} = 2\pi \\ \beta_{13}+\beta_{23}+\beta_{14}+\beta_{14} = 2\pi, ~\gamma_{13}+\gamma_{23}+\gamma_{14}+\gamma_{14} = 2\pi \\
		\delta_{14}+\delta_{24}+\delta_{11}+\delta_{21} = 2\pi, ~\alpha_{14}+\alpha_{24}+\alpha_{11}+\alpha_{21} = 2\pi \\
	\end{dcases}
\end{equation*}
Condition for being linear units:
\begin{equation*}
	\begin{dcases}
		\dfrac{\tan \beta_{11}}{\tan \alpha_{11}} = \dfrac{\tan \beta_{21}}{\tan \alpha_{21}} \\
		\dfrac{\tan \beta_{13}}{\tan \alpha_{13}} = \dfrac{\tan \beta_{23}}{\tan \alpha_{23}} \\
		\dfrac{\sin \beta_{12}}{\sin \gamma_{12}} = \dfrac{\sin \beta_{22}}{\sin \gamma_{22}}, ~ \dfrac{\sin \delta_{12}}{\sin \alpha_{12}} = \dfrac{\sin \delta_{22}}{\sin \alpha_{22}} \\
		\dfrac{\sin \beta_{14}}{\sin \gamma_{14}} = \dfrac{\sin \beta_{24}}{\sin \gamma_{24}}, ~ \dfrac{\sin \delta_{14}}{\sin \alpha_{14}} = \dfrac{\sin \delta_{24}}{\sin \alpha_{24}} \\
	\end{dcases}
\end{equation*}
Condition on equal ratio for linear units:
\begin{equation*}
	\begin{dcases}
		\dfrac{\cos \beta_{21}}{\cos \beta_{11}} = \left. \mathrm{sign}\left(\dfrac{\pi - \sigma_{22}}{\pi - \sigma_{12}}\right) \sqrt{\dfrac{\sin \beta_{22} \sin \gamma_{22}}{\sin \delta_{22} \sin \alpha_{22}}-1} \middle/ \sqrt{\dfrac{\sin \beta_{12} \sin \gamma_{12}}{\sin \delta_{12} \sin \alpha_{12}}-1} \right. \\
		\left. \mathrm{sign}\left(\dfrac{\pi - \sigma_{22}}{\pi - \sigma_{12}}\right) \sqrt{\dfrac{\sin \delta_{22} \sin \alpha_{22}}{\sin \beta_{22} \sin \gamma_{22}}-1} \middle/ \sqrt{\dfrac{\sin \delta_{12} \sin \alpha_{12}}{\sin \beta_{12} \sin \gamma_{12}}-1} \right. = \dfrac{\cos \alpha_{23}}{\cos \alpha_{13}} \\
		\dfrac{\cos \beta_{23}}{\cos \beta_{13}} = \left. \mathrm{sign}\left(\dfrac{\pi - \sigma_{24}}{\pi - \sigma_{14}}\right) \sqrt{\dfrac{\sin \beta_{24} \sin \gamma_{24}}{\sin \delta_{24} \sin \alpha_{24}}-1} \middle/ \sqrt{\dfrac{\sin \beta_{14} \sin \gamma_{14}}{\sin \delta_{14} \sin \alpha_{14}}-1} \right. \\
		\left. \mathrm{sign}\left(\dfrac{\pi - \sigma_{24}}{\pi - \sigma_{14}}\right) \sqrt{\dfrac{\sin \delta_{24} \sin \alpha_{24}}{\sin \beta_{24} \sin \gamma_{24}}-1} \middle/ \sqrt{\dfrac{\sin \delta_{14} \sin \alpha_{14}}{\sin \beta_{14} \sin \gamma_{14}}-1} \right. = \dfrac{\cos \alpha_{21}}{\cos \alpha_{11}} \\
	\end{dcases}
\end{equation*}

\subsection*{Example 4, Figure 3(d), Figure \ref{fig: details for examples}(d), non-developable, linear coupling}

The unit size is $3 \times 5$. A unit contains 8 interior vertices and 32 sector angles. The sector angles $\alpha_{ij}, ~\beta_{ij}, ~\gamma_{ij}, ~\delta_{ij}$, $i, ~j \in \mathbb{Z}^+, ~i \le 2, ~j \le 4$ meet the constraints below. There are 30 constraints for 32 sector angles, allowing two independent input sector angles. Columns 1 and 3 are deltoid II vertices, columns 2 and 4 are conic I vertices.

Vertex type condition:
\begin{equation*}
	\begin{dcases}
		\beta_{11} = \alpha_{11}, ~\gamma_{11} = \delta_{11}, ~\beta_{13} = \alpha_{13}, ~\gamma_{13} = \delta_{13} \\
		\beta_{21} = \alpha_{21}, ~\gamma_{21} = \delta_{21}, ~\beta_{23} = \alpha_{23}, ~\gamma_{23} = \delta_{23} \\
		\alpha_{12} + \gamma_{12} = \beta_{12} + \gamma_{12}, ~ \alpha_{14} + \gamma_{14} = \beta_{14} + \gamma_{14}\\
		\alpha_{22} + \gamma_{22} = \beta_{22} + \gamma_{22}, ~ \alpha_{24} + \gamma_{24} = \beta_{24} + \gamma_{24}\\
	\end{dcases}
\end{equation*} 
Planarity condition considering the periodicity of sector angles:
\begin{equation*}
	\begin{dcases}
		\beta_{11}+\beta_{21}+\beta_{12}+\beta_{12} = 2\pi, ~\gamma_{11}+\gamma_{21}+\gamma_{12}+\gamma_{12} = 2\pi \\
		\delta_{12}+\delta_{22}+\delta_{13}+\delta_{23} = 2\pi, ~\alpha_{12}+\alpha_{22}+\alpha_{13}+\alpha_{23} = 2\pi \\ \beta_{13}+\beta_{23}+\beta_{14}+\beta_{14} = 2\pi, ~\gamma_{13}+\gamma_{23}+\gamma_{14}+\gamma_{14} = 2\pi \\
		\delta_{14}+\delta_{24}+\delta_{11}+\delta_{21} = 2\pi, ~\alpha_{14}+\alpha_{24}+\alpha_{11}+\alpha_{21} = 2\pi \\
	\end{dcases}
\end{equation*}
Condition for being linear units:
\begin{equation*}
	\begin{dcases}
		\dfrac{\sin \beta_{11}}{\sin \gamma_{11}} = \dfrac{\sin \beta_{21}}{\sin \gamma_{21}} \\
		\dfrac{\sin \beta_{13}}{\sin \gamma_{13}} = \dfrac{\sin \beta_{23}}{\sin \gamma_{23}} \\
		\dfrac{\sin \beta_{12}}{\sin \gamma_{12}} = \dfrac{\sin \beta_{22}}{\sin \gamma_{22}}, ~ \dfrac{\sin \delta_{12}}{\sin \alpha_{12}} = \dfrac{\sin \delta_{22}}{\sin \alpha_{22}} \\
		\dfrac{\sin \beta_{14}}{\sin \gamma_{14}} = \dfrac{\sin \beta_{24}}{\sin \gamma_{24}}, ~ \dfrac{\sin \delta_{14}}{\sin \alpha_{14}} = \dfrac{\sin \delta_{24}}{\sin \alpha_{24}} \\
	\end{dcases}
\end{equation*}
Condition on equal ratio for linear units:
\begin{equation*}
	\begin{dcases}
	 \mathrm{sign}\left(\dfrac{\pi-\beta_{21}-\gamma_{21}}{\pi-\beta_{11}-\gamma_{11}}\right) \sqrt{\dfrac{\sin(\beta_{21}+\gamma_{21})\sin(\beta_{11}-\gamma_{11})}{\sin(\beta_{21}-\gamma_{21})\sin(\beta_{11}+\gamma_{11})}} \\
	= \left. \mathrm{sign}\left(\dfrac{\pi - \sigma_{22}}{\pi - \sigma_{12}}\right) \sqrt{\dfrac{\sin \beta_{22} \sin \gamma_{22}}{\sin \delta_{22} \sin \alpha_{22}}-1} \middle/ \sqrt{\dfrac{\sin \beta_{12} \sin \gamma_{12}}{\sin \delta_{12} \sin \alpha_{12}}-1} \right. \\
		\left. \mathrm{sign}\left(\dfrac{\pi - \sigma_{22}}{\pi - \sigma_{12}}\right) \sqrt{\dfrac{\sin \delta_{22} \sin \alpha_{22}}{\sin \beta_{22} \sin \gamma_{22}}-1} \middle/ \sqrt{\dfrac{\sin \delta_{12} \sin \alpha_{12}}{\sin \beta_{12} \sin \gamma_{12}}-1} \right. \\ = 		\mathrm{sign}\left(\dfrac{\pi-\beta_{23}-\gamma_{23}}{\pi-\beta_{13}-\gamma_{13}}\right) \sqrt{\dfrac{\sin(\beta_{23}+\gamma_{23})\sin(\beta_{13}-\gamma_{13})}{\sin(\beta_{23}-\gamma_{23})\sin(\beta_{13}+\gamma_{13})}} \\
		\mathrm{sign}\left(\dfrac{\pi-\beta_{23}-\gamma_{23}}{\pi-\beta_{13}-\gamma_{13}}\right) \sqrt{\dfrac{\sin(\beta_{23}+\gamma_{23})\sin(\beta_{13}-\gamma_{13})}{\sin(\beta_{23}-\gamma_{23})\sin(\beta_{13}+\gamma_{13})}} \\ = \left. \mathrm{sign}\left(\dfrac{\pi - \sigma_{24}}{\pi - \sigma_{14}}\right) \sqrt{\dfrac{\sin \beta_{24} \sin \gamma_{24}}{\sin \delta_{24} \sin \alpha_{24}}-1} \middle/ \sqrt{\dfrac{\sin \beta_{14} \sin \gamma_{14}}{\sin \delta_{14} \sin \alpha_{14}}-1} \right. \\
		\left. \mathrm{sign}\left(\dfrac{\pi - \sigma_{24}}{\pi - \sigma_{14}}\right) \sqrt{\dfrac{\sin \delta_{24} \sin \alpha_{24}}{\sin \beta_{24} \sin \gamma_{24}}-1} \middle/ \sqrt{\dfrac{\sin \delta_{14} \sin \alpha_{14}}{\sin \beta_{14} \sin \gamma_{14}}-1} \right. \\ = 	\mathrm{sign}\left(\dfrac{\pi-\beta_{21}-\gamma_{21}}{\pi-\beta_{11}-\gamma_{11}}\right) \sqrt{\dfrac{\sin(\beta_{21}+\gamma_{21})\sin(\beta_{11}-\gamma_{11})}{\sin(\beta_{21}-\gamma_{21})\sin(\beta_{11}+\gamma_{11})}} \\
	\end{dcases}
\end{equation*}

\subsection*{Example 5, Figure 3(e), Figure \ref{fig: details for examples}(e), developable, linear coupling}

The unit size is $3 \times 5$. A unit contains 8 interior vertices and 32 sector angles. The sector angles $\alpha_{ij}, ~\beta_{ij}, ~\gamma_{ij}, ~\delta_{ij}$, $i, ~j \in \mathbb{Z}^+, ~i \le 2, ~j \le 4$ meet the constraints below. There are 30 constraints for 32 sector angles, allowing two independent input sector angles. Columns 1 and 3 are anti-deltoid I vertices, columns 2 and 4 are conic IV vertices. This example is formed from switching certain strips of Example 3, transforming it from a non-developable pattern into a developable one. For a solution from Example 3, let
\begin{equation*}
	\alpha_{ij} \rightarrow \pi - \alpha_{ij}, ~~\gamma_{ij} \rightarrow \pi - \gamma_{ij}
\end{equation*} 

\subsection*{Example 6, Figure 3(f), Figure \ref{fig: details for examples}(f), developable, equimodular coupling}

Here we choose a $5 \times 5$ unit. The constraints below are listed assuming all the vertices are conic I, then transforming them to conic IV (developable) from switching certain strips --- let
\begin{equation*}
	\alpha_{ij} \rightarrow \pi - \alpha_{ij}, ~~\gamma_{ij} \rightarrow \pi - \gamma_{ij}
\end{equation*} 

A $5 \times 5$ unit contains 16 interior vertices and 64 sector angles, the sector angles $\alpha_{ij}, ~\beta_{ij}, ~\gamma_{ij}, ~\delta_{ij}$, $i, ~j \in \mathbb{Z}^+, ~i \le 4, ~j \le 4$ meet the constraints below.

Vertex type condition (all the vertices are conic I):
\begin{equation*}
	\begin{dcases}
		\alpha_{11} + \gamma_{11} = \beta_{11} + \gamma_{11}, ~ \alpha_{12} + \gamma_{12} = \beta_{12} + \gamma_{12}\\
		\alpha_{13} + \gamma_{13} = \beta_{13} + \gamma_{13}, ~ \alpha_{14} + \gamma_{14} = \beta_{14} + \gamma_{14}\\
		\alpha_{21} + \gamma_{21} = \beta_{21} + \gamma_{21}, ~ \alpha_{22} + \gamma_{22} = \beta_{22} + \gamma_{22}\\
		\alpha_{23} + \gamma_{23} = \beta_{23} + \gamma_{23}, ~ \alpha_{24} + \gamma_{24} = \beta_{24} + \gamma_{24}\\
		\alpha_{31} + \gamma_{31} = \beta_{31} + \gamma_{31}, ~ \alpha_{32} + \gamma_{32} = \beta_{32} + \gamma_{32}\\
		\alpha_{33} + \gamma_{33} = \beta_{33} + \gamma_{33}, ~ \alpha_{34} + \gamma_{34} = \beta_{34} + \gamma_{34}\\
		\alpha_{41} + \gamma_{41} = \beta_{41} + \gamma_{41}, ~ \alpha_{42} + \gamma_{42} = \beta_{42} + \gamma_{42}\\
		\alpha_{43} + \gamma_{43} = \beta_{43} + \gamma_{43}, ~ \alpha_{44} + \gamma_{44} = \beta_{44} + \gamma_{44}\\
	\end{dcases}
\end{equation*} 
Planarity condition considering the periodicity of sector angles:
\begin{equation*}
	\begin{dcases}
		\beta_{11}+\beta_{21}+\beta_{12}+\beta_{12} = 2\pi, ~\gamma_{21}+\gamma_{31}+\gamma_{22}+\gamma_{32} = 2\pi \\
		\beta_{31}+\beta_{41}+\beta_{32}+\beta_{42} = 2\pi, ~\gamma_{41}+\gamma_{11}+\gamma_{42}+\gamma_{12} = 2\pi \\
		\alpha_{12}+\alpha_{22}+\alpha_{13}+\alpha_{23} = 2\pi, ~\delta_{22}+\delta_{32}+\delta_{23}+\delta_{33} = 2\pi \\
		\alpha_{32}+\alpha_{42}+\alpha_{33}+\alpha_{43} = 2\pi, ~\delta_{42}+\delta_{12}+\delta_{43}+\delta_{13} = 2\pi \\
		\beta_{13}+\beta_{23}+\beta_{14}+\beta_{14} = 2\pi, ~\gamma_{23}+\gamma_{33}+\gamma_{24}+\gamma_{34} = 2\pi \\
		\beta_{33}+\beta_{43}+\beta_{34}+\beta_{44} = 2\pi, ~\gamma_{43}+\gamma_{13}+\gamma_{44}+\gamma_{14} = 2\pi \\
		\alpha_{14}+\alpha_{24}+\alpha_{11}+\alpha_{21} = 2\pi, ~\delta_{24}+\delta_{34}+\delta_{21}+\delta_{31} = 2\pi \\
		\alpha_{34}+\alpha_{44}+\alpha_{31}+\alpha_{41} = 2\pi, ~\delta_{44}+\delta_{14}+\delta_{41}+\delta_{11} = 2\pi \\
	\end{dcases}
\end{equation*}
Condition on equal amplitudes:
\begin{equation*}
	\begin{dcases}
		\dfrac{\sin \alpha_{11} \sin \beta_{11}}{\sin \gamma_{11} \sin \delta_{11}} = \dfrac{\sin \alpha_{21} \sin \beta_{21}}{\sin \gamma_{21} \sin \delta_{21}}, ~		\dfrac{\sin \alpha_{21} \sin \beta_{21}}{\sin \gamma_{21} \sin \delta_{21}} = \dfrac{\sin \alpha_{31} \sin \beta_{31}}{\sin \gamma_{31} \sin \delta_{31}}, ~		\dfrac{\sin \alpha_{31} \sin \beta_{31}}{\sin \gamma_{31} \sin \delta_{31}} = \dfrac{\sin \alpha_{41} \sin \beta_{41}}{\sin \gamma_{41} \sin \delta_{41}} \\
		\dfrac{\sin \alpha_{12} \sin \beta_{12}}{\sin \gamma_{12} \sin \delta_{12}} = \dfrac{\sin \alpha_{22} \sin \beta_{22}}{\sin \gamma_{22} \sin \delta_{22}}, ~		\dfrac{\sin \alpha_{22} \sin \beta_{22}}{\sin \gamma_{22} \sin \delta_{22}} = \dfrac{\sin \alpha_{32} \sin \beta_{32}}{\sin \gamma_{32} \sin \delta_{32}}, ~		\dfrac{\sin \alpha_{32} \sin \beta_{32}}{\sin \gamma_{32} \sin \delta_{32}} = \dfrac{\sin \alpha_{42} \sin \beta_{42}}{\sin \gamma_{42} \sin \delta_{42}} \\
		\dfrac{\sin \alpha_{13} \sin \beta_{13}}{\sin \gamma_{13} \sin \delta_{13}} = \dfrac{\sin \alpha_{23} \sin \beta_{23}}{\sin \gamma_{23} \sin \delta_{23}}, ~		\dfrac{\sin \alpha_{23} \sin \beta_{23}}{\sin \gamma_{23} \sin \delta_{23}} = \dfrac{\sin \alpha_{33} \sin \beta_{33}}{\sin \gamma_{33} \sin \delta_{33}}, ~		\dfrac{\sin \alpha_{33} \sin \beta_{33}}{\sin \gamma_{33} \sin \delta_{33}} = \dfrac{\sin \alpha_{43} \sin \beta_{43}}{\sin \gamma_{43} \sin \delta_{43}} \\
		\dfrac{\sin \alpha_{14} \sin \beta_{14}}{\sin \gamma_{14} \sin \delta_{14}} = \dfrac{\sin \alpha_{24} \sin \beta_{24}}{\sin \gamma_{24} \sin \delta_{24}}, ~		\dfrac{\sin \alpha_{24} \sin \beta_{24}}{\sin \gamma_{24} \sin \delta_{24}} = \dfrac{\sin \alpha_{34} \sin \beta_{34}}{\sin \gamma_{34} \sin \delta_{34}}, ~		\dfrac{\sin \alpha_{34} \sin \beta_{34}}{\sin \gamma_{34} \sin \delta_{34}} = \dfrac{\sin \alpha_{44} \sin \beta_{44}}{\sin \gamma_{44} \sin \delta_{44}} \\
		\dfrac{\sin \beta_{11} \sin \gamma_{11}}{\sin \delta_{11} \sin \alpha_{11}} = \dfrac{\sin \beta_{12} \sin \gamma_{12}}{\sin \delta_{12} \sin \alpha_{12}}, ~		\dfrac{\sin \beta_{12} \sin \gamma_{12}}{\sin \delta_{12} \sin \alpha_{12}} = \dfrac{\sin \beta_{13} \sin \gamma_{13}}{\sin \delta_{13} \sin \alpha_{13}}, ~		\dfrac{\sin \beta_{13} \sin \gamma_{13}}{\sin \delta_{13} \sin \alpha_{13}} = \dfrac{\sin \beta_{14} \sin \gamma_{14}}{\sin \delta_{14} \sin \alpha_{14}} \\
		\dfrac{\sin \beta_{21} \sin \gamma_{21}}{\sin \delta_{21} \sin \alpha_{21}} = \dfrac{\sin \beta_{22} \sin \gamma_{22}}{\sin \delta_{22} \sin \alpha_{22}}, ~		\dfrac{\sin \beta_{22} \sin \gamma_{22}}{\sin \delta_{22} \sin \alpha_{22}} = \dfrac{\sin \beta_{23} \sin \gamma_{23}}{\sin \delta_{23} \sin \alpha_{23}}, ~		\dfrac{\sin \beta_{23} \sin \gamma_{23}}{\sin \delta_{23} \sin \alpha_{23}} = \dfrac{\sin \beta_{24} \sin \gamma_{24}}{\sin \delta_{24} \sin \alpha_{24}} \\
		\dfrac{\sin \beta_{31} \sin \gamma_{31}}{\sin \delta_{31} \sin \alpha_{31}} = \dfrac{\sin \beta_{32} \sin \gamma_{32}}{\sin \delta_{32} \sin \alpha_{32}}, ~		\dfrac{\sin \beta_{32} \sin \gamma_{32}}{\sin \delta_{32} \sin \alpha_{32}} = \dfrac{\sin \beta_{33} \sin \gamma_{33}}{\sin \delta_{33} \sin \alpha_{33}}, ~		\dfrac{\sin \beta_{33} \sin \gamma_{33}}{\sin \delta_{33} \sin \alpha_{33}} = \dfrac{\sin \beta_{34} \sin \gamma_{34}}{\sin \delta_{34} \sin \alpha_{34}} \\
		\dfrac{\sin \beta_{41} \sin \gamma_{41}}{\sin \delta_{41} \sin \alpha_{41}} = \dfrac{\sin \beta_{42} \sin \gamma_{42}}{\sin \delta_{42} \sin \alpha_{42}}, ~		\dfrac{\sin \beta_{42} \sin \gamma_{42}}{\sin \delta_{42} \sin \alpha_{42}} = \dfrac{\sin \beta_{43} \sin \gamma_{43}}{\sin \delta_{43} \sin \alpha_{43}}, ~		\dfrac{\sin \beta_{43} \sin \gamma_{43}}{\sin \delta_{43} \sin \alpha_{43}} = \dfrac{\sin \beta_{44} \sin \gamma_{44}}{\sin \delta_{44} \sin \alpha_{44}} \\
	\end{dcases}
\end{equation*}
Regarding the condition on equal phase shifts for every Kokotsakis quadrilateral,
\begin{equation*}
	\begin{dcases}
		\theta^a_{11} - \theta^a_{21} =  \theta^a_{12} - \theta^a_{22}, ~\theta^a_{21} - \theta^a_{31} =  \theta^a_{22} - \theta^a_{32}, ~\theta^a_{31} - \theta^a_{41} =  \theta^a_{32} - \theta^a_{42} \\
		\theta^b_{12} - \theta^b_{22} =  \theta^b_{13} - \theta^b_{23}, ~\theta^b_{22} - \theta^b_{32} =  \theta^b_{23} - \theta^b_{33}, ~\theta^b_{32} - \theta^b_{42} =  \theta^b_{33} - \theta^b_{43} \\
		\theta^a_{13} - \theta^a_{23} =  \theta^a_{14} - \theta^a_{24}, ~\theta^a_{23} - \theta^a_{33} =  \theta^a_{24} - \theta^a_{34}, ~\theta^a_{33} - \theta^a_{43} =  \theta^a_{34} - \theta^a_{44} \\
	\end{dcases}
\end{equation*}
There are a total of 65 constraints on the 64 sector angles, and we find a solution that aligns with the `tiling' mentioned in \citet{dieleman_jigsaw_2020}. 

\subsection*{Calculation of crease lengths}

Plotting the entire pattern requires the calculation of all the crease lengths. As depicted in Figure \ref{fig: labelling discrete net}, suppose the input crease lengths are located in row $i_1 = 0$ and column $i_2 = 0$, the first step is to rearrange the sector angles $\alpha, ~\beta, ~\gamma, ~\delta$ to the angles $S^\mathrm{a}, ~S^\mathrm{b}, ~S^\mathrm{c}, ~S^\mathrm{d}$. Furthermore, for the right bottom region ($i_1 > 0$, $i_2 > 0$):
\begin{equation*}
	\begin{gathered}
		l_x(i,~j + 1) =
		\dfrac{l_x(i,~j) \sin S^{\mathrm{a}}(i,~j) - l_y(i,~j) \sin(S^{\mathrm{d}}(i,~j) + S^{\mathrm{a}}(i,~j))}{	\sin S^{\mathrm{b}}(i,~j)} \\
		l_y(i+1,~j) =
		\dfrac{l_y(i,~j) \sin S^{\mathrm{c}}(i,~j) - l_x(i,~j) \sin(S^{\mathrm{d}}(i,~j) + S^{\mathrm{c}}(i,~j))}{	\sin S^{\mathrm{b}}(i,~j)} \\
	\end{gathered}
\end{equation*}
for the right top region ($i_1 < 0$, $i_2 > 0$):
\begin{equation*}
	\begin{gathered}
		l_x(i,~j + 1) =
		\dfrac{l_x(i,~j) \sin S^{\mathrm{d}}(i,~j) - l_y(i+1,~j) \sin(S^{\mathrm{a}}(i,~j) + S^{\mathrm{d}}(i,~j))}{	\sin S^{\mathrm{c}}(i,~j)} \\
		l_y(i,~j) =
		\dfrac{l_y(i+1,~j) \sin S^{\mathrm{b}}(i,~j) - l_x(i,~j) \sin(S^{\mathrm{a}}(i,~j) + S^{\mathrm{b}}(i,~j))}{	\sin S^{\mathrm{c}}(i,~j)} \\
	\end{gathered}
\end{equation*}
for the left bottom region ($i_1 > 0$, $i_2 < 0$): 
\begin{equation*}
	\begin{gathered}
		l_x(i,~j) =
		\dfrac{l_x(i,~j+1) \sin S^{\mathrm{b}}(i,~j) - l_y(i,~j) \sin(S^{\mathrm{c}}(i,~j) + S^{\mathrm{b}}(i,~j))}{	\sin S^{\mathrm{a}}(i,~j)} \\
		l_y(i+1,~j) =
		\dfrac{l_y(i,~j) \sin S^{\mathrm{d}}(i,~j) - l_x(i,~j+1) \sin(S^{\mathrm{c}}(i,~j) + S^{\mathrm{d}}(i,~j))}{	\sin S^{\mathrm{a}}(i,~j)} \\
	\end{gathered}
\end{equation*}
for the left top region ($i_1 < 0$, $i_2 < 0$)
\begin{equation*}
	\begin{gathered}
		l_x(i,~j) =
		\dfrac{l_x(i,~j+1) \sin S^{\mathrm{c}}(i,~j) - l_y(i+1,~j) \sin(S^{\mathrm{b}}(i,~j) + S^{\mathrm{c}}(i,~j))}{	\sin S^{\mathrm{d}}(i,~j)} \\
		l_y(i,~j) =
		\dfrac{l_y(i+1,~j) \sin S^{\mathrm{a}}(i,~j) - l_x(i,~j+1) \sin(S^{\mathrm{b}}(i,~j) + S^{\mathrm{a}}(i,~j))}{	\sin S^{\mathrm{d}}(i,~j)} \\
	\end{gathered}
\end{equation*}

\bibliographystyle{plainnat}


\end{document}